\documentclass[11pt,letterpaper]{amsart}
\usepackage{pgfplots}
\pgfplotsset{compat=1.18}
\usepackage[margin=1in]{geometry}

\usepackage[numbers,sort&compress]{natbib}
\usepackage{amsmath,amssymb,amsthm,mathtools}
\usepackage{bm}
\usepackage{microtype}
\usepackage[hidelinks]{hyperref}
\usepackage{enumitem}
\usepackage{xcolor}
\usepackage{tikz}
\usetikzlibrary{arrows.meta,positioning,fit,backgrounds}
\usepackage{graphicx}
\newtheorem{theorem}{Theorem}[section]
\newtheorem{proposition}[theorem]{Proposition}
\newtheorem{lemma}[theorem]{Lemma}
\newtheorem{corollary}[theorem]{Corollary}

\theoremstyle{definition}
\newtheorem{definition}[theorem]{Definition}

\theoremstyle{remark}
\newtheorem{remark}[theorem]{Remark}

\newcommand{\R}{\mathbb{R}}
\newcommand{\C}{\mathbb{C}}
\newcommand{\Z}{\mathbb{Z}}
\newcommand{\N}{\mathbb{N}}
\newcommand{\T}{\mathbb{T}}

\newcommand{\HH}{\mathcal{H}}
\newcommand{\cA}{\mathcal{A}}
\newcommand{\cL}{\mathcal{L}}
\newcommand{\dom}{\mathcal{D}}
\newcommand{\cD}{\mathcal{D}}

\newcommand{\sap}{\sigma_{\mathrm{ap}}}
\newcommand{\spp}{\sigma_{\mathrm{p}}}

\newcommand{\norm}[1]{\lVert #1\rVert}
\newcommand{\abs}[1]{\lvert #1\rvert}

\newcommand{\Om}{\Omega}

\newcommand{\eps}{\varepsilon}

\DeclareMathOperator{\dist}{dist}
\DeclareMathOperator{\supp}{supp}

\IfFileExists{ulem.sty}{\usepackage[normalem]{ulem}%
  }{%
  }

\title[Resonance in Partially Dissipative Systems]
{Range Failure, Resolvent Growth, and Resonance 
in Partially Dissipative Systems}

\author{Boris Muha}
\address{Department of Mathematics, Faculty of Science,
University of Zagreb, Croatia}
\email{borism@math.hr}
\email{boris.muha@math.pmf.unizg.hr}

\author{Sebastian Schwarzacher}
\address{University of Uppsala, Sweden}
\email{sebastian.schwarzacher@math.uu.se}
\address{Charles University, Prague, Czechia}
\email{schwarz@karlin.mff.cuni.cz}

\author{Justin T. Webster}
\address{Department of Mathematics and Statistics,
University of Maryland, Baltimore County,
1000 Hilltop Circle, Baltimore, MD 21250, USA}
\email{websterj@umbc.edu}

\begin{document}

\begin{abstract}
\noindent We analyze resonance for partially dissipative systems,  $\dot u=Au+f(t)$, where $A$ generates a strongly stable semigroup $S(t)$ on a Hilbert space $H$. 
{Resonance is} the phenomenon whereby a bounded, time-periodic $f$  yields an unbounded $u$, classically occurring when $A$ has imaginary eigenvalues. In infinite dimensions, however, 
several other resonance notions, introduced here, may occur.

Classically, {$T$-periodic solvability} is equivalent to a range condition: $\mathcal{R}(I-S(T))=H$. 
We describe a hierarchy of range failures, and demonstrate how the growth rate of~$||(A-i\omega I)^{-1}||_H$ dictates a regularity gap between $u$ and $f$. For super-polynomial growth, that  gap can be infinite. We demonstrate this phenomenon and its equivalence to the existence of smooth resonant forces.

The theory is motivated by a well-studied hyperbolic-parabolic system, which is a simplification of fluid-solid interactions.
Our main result shows that, for a carefully constructed geometry,  the PDE system exhibits smooth-forcing resonance for a dense set of periods. The geometry sharply violates geometric conditions for which finite regularity losses were shown previously.
\end{abstract}

\subjclass[2020]{Primary 35B34; Secondary 35B10, 47D06, 47A10, 35M33, 35B40}
\renewcommand{\subjclassname}{2020 Mathematics Subject Classification}
\keywords{periodic solutions, resonance, partially dissipative systems, resolvent analysis, hyperbolic-parabolic systems, semigroup stability}
\thanks{\textit{Use of AI tools.} The authors used the large language models Claude (Anthropic) and ChatGPT (OpenAI) during the preparation of this manuscript.  These tools were utilized for literature search and review, early drafting of expository passages, editing, and language polishing. All mathematical content and proofs of statements were conceived by the authors. The authors certify that they have checked the entire manuscript, including all references, and bear sole responsibility for its content and correctness.}

\maketitle

\section{Introduction}\label{sec:intro}
For elastic systems under time-dependent loading the concept of {\em resonance} has been studied for hundreds of years across science, engineering, and mathematics. The word conjures images of storied occurrences involving bridges and buildings, as well as physics experiments (and perhaps even a night at the opera). A na\"ive definition of a resonance might be: a small time-periodic input produces a large---perhaps unboundedly growing---response from a solid body.  In the world of ODEs and finite dimensional systems, the story is quite
familiar: a forced frequency finds a natural underlying frequency, and the response grows {\em secularly}, yielding a solution that {grows linearly} in time. This picture is so compelling and scientifically ubiquitous that it often becomes a default intuition for
resonance itself.  In infinite dimensional systems, however, we shall see that this intuition is too narrow. It is the unboundedness of the response, rather than the mechanism producing it, that survives the passage to infinite dimensions. 

We will consider an abstract system on an arbitrary (possibly infinite dimensional) Hilbert space:
\begin{align}\label{PDS}
   \dot u = Au+f(t)\in H.
\end{align} 
Accordingly, we focus on {\em unbounded response} in our central definition of resonance: \begin{quote} {\em The system \eqref{PDS} is {\em secularly resonant} if some time-periodic forcing yields an unbounded solution (given precisely in Definition~\ref{def:secular}).}\end{quote} Every other notion considered in this paper is measured against this one.

Most physical systems exhibit some type of energy dissipation, for instance through friction, viscosity, or thermal diffusion. As such, there is some competition between the manner in which energy enters the system through $f(t)$ and leaves the system via dissipation, $(Au,u)_H$. Many damped physical systems yield no imaginary eigenvalues of $A$, and may have strong enough damping that unforced, finite-energy trajectories decay to zero as $t \to \infty$. It is natural to ask, for such systems, whether a traditional resonance can still occur.

The specific motivating class here is that of so-called {\em partially dissipative systems}. These are dynamics in which dissipation is present but spatially or systemically localized. In these cases, it may not be clear that damping mechanisms have sufficient strength to ``control" the entire state (hence may be referred to as weak).  Notable examples include: hyperbolic-parabolic couplings,
thermo- and magneto-elastic systems, indirectly or locally damped waves,
acoustic or boundary-damped hyperbolic models, and fluid-structure interactions where
dissipation is transmitted through an interface
\cite{Galdi,ZhangZuazua2003,RauchZhangZuazua2005,ZhangZuazua2007,Duyckaerts2007,AvalosLasieckaTriggiani2016,BattyPaunonenSeifert2019,Avalos2007,AvalosTriggiani2007,AvalosTriggiani2009}.
We will discuss various notions of stability below, and relate them to the problem of forced resonance.

A complementary property to system resonance is {\em periodic well-posedness} thereof. By this we mean that the system yields existence and uniqueness of finite energy, $T$-periodic solutions for a given class of $T$-periodic forcing functions \cite{Pruss1984}. Of course, the  value of $T$ may be central to these properties. Periodic problems have a notably different structure than Cauchy (initial value) problems.  In a Cauchy
problem, the initial state $u_0$ is prescribed and propagated by $A$ through $f$ in the dynamics.  In a periodic problem,
{\em the initial state is the central unknown}.  As we shall see below, this converts the periodic problem into a static equation for the (possible) range
defect of a period map.  This point of view is classical in monodromy theory, going back at least to
\cite{Massera,Krein1971,KreinMonodromy}, and it appears in the context of several other classical references
\cite{Taam1966,Browder1965,StraskrabaVejvoda1973,Herrmann1980,DanersKochMedina1992} and \cite{Pruss1984}, which motivated \cite{Galdi}. 
Closely related ideas also appear in dissipative-wave and almost-periodic
theory \cite{AmerioProuse1971,BiroliHaraux1980,Haraux1983,Haraux2018,ArendtBatty1997},
as well as nonlinear semigroup and monotone operator theory
\cite{Brezis1973,Barbu1976,Bostan2002}, and finally in output regulation and
the internal model principles
\cite{ImmonenPohjolainen2006,PaunonenPohjolainen2010}.
These strands of literature address neighboring questions with different terminology.  In much of the PDE literature, resonance is
 invoked more often than it is defined. Our central motivating reference 
\cite{Galdi} gives a  clean formulation of resonance and elucidates its connection with the semigroup decay.

Indeed,  \cite{Galdi} shows that if all unforced trajectories decay to zero (strong stability) in the space of finite energy $C([0,\infty); H)$,
then there exists a dense set of forcings inside the mild solution forcing class, $L^1(0,T;H)$, for
which the periodic problem is uniquely solvable. On the other hand, \cite{Galdi} also shows that if the damping is strong enough to yield {\em uniform stability} (prescribed exponential decay rates) for the homogeneous system, then periodic well-posedness is unconditional: it holds for arbitrary $T>0$ and $f\in L^1(0,T;H)$. Thus the basic picture for partially dissipative systems is delicate: {\bf non-resonant
forcings are dense, while pathological resonant forcings may exist.} The hard questions
are then how to
identify this dense, periodic well-posedness space, and also how to construct a forcing external to it, yielding resonance. For these fundamental problems, the question may not be whether
periodic well-posedness fails, but how badly it does so and how to best measure this.

Our perspective here centers on {\em derivative losses}, and we show that these provide a quantitative language for the aforementioned set of problems.  In
\cite{MMSW2024}, time-periodic solutions for the heat-wave system were
constructed directly through PDE estimates.  The constructed forcing classes exhibited clear Sobolev regularity ``losses"
between the periodic data and the finite-energy solution. This reflected
the need to reconstruct hyperbolic energies from parabolic dissipation, acting only through
the coupling interface. In the context of stability theory, this derivative loss phenomenon is closely associated with so called rational (or polynomial) decay of the solution semigroup. In \cite{GMW2026}, this connection
was isolated abstractly: rational stability of the semigroup
equivalently corresponds to polynomial resolvent growth, which translates to Fourier derivative losses.
Related results, in different weakly damped frameworks, were obtained in \cite{LaurentRivas2026} based explicitly on prescribed decay rates in the time domain.  These two recent works explain how finite derivative loss can produce explicit, dense periodic well-posedness spaces which are proper inside the mild forcing class.  These finite-loss well-posedness results do not, however, explain how resonance may occur inside the remaining range defect. What happens beyond every finite Sobolev loss, and what does it imply for resonance?

{\bf This paper answers those questions.}  We organize the possible failures
of periodic well-posedness into a hierarchy of  mechanisms. Some  notions we describe are classical under
other names; some are implicit in the semigroup or control literature; and some are apparently
new.  The equivalence of range failure and necessary growth in time is the backbone of this story: if there exists a pair $(T,f)$ so that the periodic problem is not solvable, then all resulting mild solutions for said $f$ necessarily have temporal growth. This is Massera's principle, which for linear problems in reflexive Banach spaces follows from Zubelevich's work on affine maps \cite{Zubelevich2006,Zubelevich2010}; Theorem~\ref{thm:equiv} states this in the form we use, together with its range-theoretic and ergodic counterparts. A central contribution, then, is to identify which failures of periodic estimates force resonance (and at what forcing regularity). We also show how these mechanisms operate in explicit partially dissipative systems.  Our examples demonstrate that these are truly
separate notions of resonance, and also that this ``hierarchy of resonances" is genuinely
infinite-dimensional in nature. Our principal concern is therefore not merely the failure of periodic solvability, but the precise regularity threshold at which that failure can occur.

The main analysis of the paper establishes infinite loss in a very well-studied partially dissipative PDE system. We construct a two-dimensional heat-wave geometry that violates the control-theoretic geometric conditions, typically ensuring propagation of dissipation across the interface.  This constructed geometry supports high-frequency, localized, and focusing wave behavior that is
effectively invisible to the interface and hence does not see the damping. The exponential localization of ellipse eigenfunctions (along the minor axis) that drives our construction is known \cite[Theorem 3.1]{NguyenGrebenkov2013}; see also \cite{CardosoPopov}. What is new is the conversion of this localization into quasimodes of the coupled system that  the interface cannot see and the sampling of their frequencies on a lattice. On the resolvent side, this
produces super-polynomial growth along a sampled lattice of frequencies for suitable
periods $T$.  On the periodic side, for each such period, it implies that no finite
Sobolev loss can control the data-to-solution map. An abstract Baire argument shows that the periods with infinite derivative loss form a dense $G_\delta$ set; more generally, when $i\R\subset\rho(A)$, the failure of every polynomial bound on the resolvent along $i\R$ is equivalent to infinite derivative loss on a dense $G_\delta$ set of periods (Theorem~\ref{thm:hfc-baire}). This argument does not select a numerical period. For the ellipse, the arithmetic of the quasimode frequencies identifies a concrete subset of resonant periods, again a dense $G_\delta$ set but of Lebesgue measure zero (Proposition~\ref{prop:explicit-R}).  Combining this ``sampled
resolvent catastrophe" with an abstract ``stacking" theorem, we obtain a
$C^\infty$ smooth time-periodic forcing for which the corresponding mild
solution grows unboundedly.

The aforementioned construction is the central example and the main payoff of our analysis.  Plainly, {\bf it shows that
damping may eliminate free oscillations, yielding strong stability, without eliminating the possibility of resonance.}  {\em This mechanism is also sharp} at the abstract level: for a fixed period with frequency lattice in the resolvent set, a smooth resonant forcing exists if and only if the sampled resolvent does not admit a polynomial bound (Corollary~\ref{cor:smooth-iff}). Combined with Theorem~\ref{thm:hfc-baire}, this shows that smooth resonant forcings exist if and only if the semigroup has no polynomial decay rate (Corollary~\ref{cor:poly-smooth}), a converse to the finite-loss well-posedness results of \cite{GMW2026,LaurentRivas2026}. Related resonant constructions can be found in: (1) \cite{LaurentRivas2026} (as discussed in Section~\ref{sec:pdamped}) for the damped Schr\"odinger equation on the round sphere; here the arithmetic of the spectrum aligns the trapped frequencies with the lattice of a single period. And (2), an abstract example \cite[Example 9.1]{BattyChill1999} with non-periodic forcing. {\em We are not aware of a comparable construction for the more challenging second order hyperbolic dynamics}. The construction also yields a sharp geometric threshold within a family of truncated ellipses: extending the wave domain beyond the semi-ellipse destroys every finite Sobolev-loss periodic estimate. As a by-product, the resolvent bound behind Theorem~\ref{thm:ellipse} shows that the logarithmic decay rate of \cite{Fathallah2011} for the heat-wave system cannot be improved for an ellipse with an interior heat domain away from the minor axis (Remark~\ref{rem:logdecay}).


\subsection{Precise terminology and foundational notions}\label{sec:terminology}
Let $H$ be a (complex) Hilbert space with norm $\norm{\cdot}$ and let $A:\dom(A)\subseteq H\to H$ be a closed, densely defined operator that generates a \emph{uniformly bounded} $C_0$-semigroup $\bigl(S(t)\bigr)_{t\ge0}$ on $H$ (written $S(t)$ for short) \cite{EngelNagel2000,Pazy1983}. Hence, throughout, we assume 
\begin{equation}\label{eq:bdd}
\norm{S(t)}_{\cL(H)}\le M\qquad\text{for all }t\ge0,
\end{equation}
for a constant $M\ge1$ which will remain fixed. As noted above, we are interested in problems that exhibit {\em some} dissipation. Here we will understand this to yield, at minimum, the \emph{strong stability} of $S(t)$, i.e.  we will assume $S(t)x\to0$ in $H$ for every $x\in H$.\footnote{There are dynamics which have such weak dissipation that the semigroup is not strongly stable, but we are not considering such dynamics here.}

Accordingly, we will study the inhomogeneous problem
\begin{equation}\label{eq:main}
\dot u=Au+f(t),
\end{equation}
where $f$ is a given time-periodic forcing with minimal\footnote{By requiring minimality of the period, we exclude constant in time forcing. Constant forcing is a steady-state problem of a different
character, governed by whether $0\in\rho(A)$.} period $T>0$.
For $u_0\in H$ and any $f\in L^1_{\text{loc}}(\mathbb R;H)$, the unique {\em mild solution} associated with \eqref{eq:main} is given by
\begin{equation}\label{eq:mild}
u(t)=S(t)u_0+\int_0^t S(t-s)f(s)\,ds\in C([0,\infty);H).
\end{equation}

We are now in position to define the notion of periodic well-posedness of mild solutions, following the earlier conventions of \cite{Galdi,Pruss1984}.
\begin{definition}\label{def:PWP}
   Let $T>0$ and $\mathcal Q\subset L^1_\#(0,T;H)$. We say that the system \eqref{eq:main} is \emph{periodically well-posed} (PWP) on $\mathcal Q$ (for the period $T$) if, for every $f\in \mathcal Q$, there exists a unique $T$-periodic mild solution $u\in C([0,\infty);H)$ of \eqref{eq:main}.
   When $\mathcal Q = L^1_\#(0,T;H)$ we say that \eqref{eq:main} is {\em unconditionally} PWP.
\end{definition}

Two comments are in order. First, we choose the notational convention that $L^1_\#$ refers to $L^1(0,T)$-type functions that are extended a.e. to $\R$ by $T$-periodicity. Secondly, we omit from the definition the requirement that the solution depends continuously on the forcing, as this is a very subtle and central issue addressed later.

We write: $\rho(A)\subseteq \mathbb C$ for the (open) resolvent set of the operator $A$, $\sigma(A)=\mathbb C\setminus\rho(A)$ for the
spectrum, and $\spp(A)$ for the point spectrum (the eigenvalues) of $A$. The
approximate point spectrum of $A$ is given by
\[
\sap(A):=\bigl\{\lambda\in\C:\ \exists\,x_n\in\dom(A),\ \norm{x_n}=1,\
\norm{(A-\lambda I)x_n}\to0\bigr\}
\]
and is the set of $\lambda$ for which $A-\lambda I$ is \emph{not bounded below}; it
contains $\spp(A)$. Writing 
$$\sigma_{\mathrm r}(A):=\{\lambda:\ A-\lambda I\ \text{is bounded below with non-dense range}\}$$
for the residual spectrum of $A$, one
has the disjoint decomposition
\[
\sigma(A)=\sap(A)\ \sqcup\ \sigma_{\mathrm r}(A).
\]
We note a standard fact from functional analysis: $\partial\sigma(A)\subseteq\sap(A)$ \cite{Yosida1995}. Moreover, uniform boundedness of the semigroup provides $\{\operatorname{Re}\lambda>0\}\subseteq
\rho(A)$. Therefore $\sigma(A)\subseteq\{\operatorname{Re}\lambda\le0\}$ and, for the semigroups of interest here, every
imaginary spectral point is a boundary point of $\sigma(A)$, so
\[
\sigma(A)\cap i\R=\sap(A)\cap i\R.
\]
As the resolvent operator will be a central object in our analysis, we recall that 
$$R_{\lambda}(A):=(\lambda I-A)^{-1}: \rho(A) \to \mathcal L(H)~~\text{ (written}~ R_{\lambda} ~ \text{ for short)}.$$ 
We are primarily concerned with the  growth rate of $R_{\lambda}$ when restricted to $i\mathbb R$, in particular, for a fixed period $T>0$ we will focus on the growth of $R_\lambda$ along the lattice $i\frac{2\pi}{T}\Z=:\{i\omega_n : n\in\Z\}$. These \emph{sampled frequencies} will be central to Section~\ref{sec:taxonomy}.

We now turn to the classical notion of resonance, which is characterized by the point spectrum. 
\begin{definition}\label{def:spectral}
The system \eqref{eq:main} is \emph{point-spectrally} (classically)
resonant if $\spp(A)\cap i\R\neq\emptyset$.
\end{definition}
\noindent Definition \ref{def:spectral} is the notion of resonance informed by the finite dimensional setting, where $S(t)$ is given by $e^{At}$ with $A\in\mathcal{M}_n(\C)$. The textbook example is of course {\em the harmonic oscillator}, where $A=\begin{pmatrix}0&1\\-\omega_0^2&0\end{pmatrix}$, with eigenvalues $\spp(A)\cap i\R=\{\pm i\omega_0\}$. 

As we are interested in semigroups $S(t)$ demonstrating strong stability, we note that this property  is incompatible with imaginary
eigenvalues. 
\begin{lemma}
\label{prop:no-classical}
If $(S(t))_{t\ge0}$ is strongly stable then $\spp(A)\cap i\R=\emptyset$, i.e., \eqref{eq:main} is not
point-spectrally resonant.
\end{lemma}
\begin{proof}
If $A\phi=i\omega\phi$ with $\omega\in\R$ and $\phi\neq0$, then we have
$S(t)\phi=e^{i\omega t}\phi$, whence $\norm{S(t)\phi}=\norm{\phi}$ for all $t$. But this  implies
that $S(t)\phi\not\to0$, contradicting the definition of strong stability of $S(t)$.
\end{proof}
We note that the Arendt-Batty stability theorem \cite{ArendtBatty1988} provides sufficient conditions for strong stability, one of which is the absence of imaginary point spectra.
 In line with these observations,  the only available mechanism for {\em spectral obstruction} on $i\mathbb R$ is through $\sap(A)$. We will characterize this as an additional resonance mechanisms in the sequel---{\em spectral resonance}. 
 
We now turn to another point of view  eliciting a static perspective for PWP. Imposing the $T$-periodicity requirement, $u(T)=u(0)=u_0$ (assuming such $u_0$ exists) in \eqref{eq:main}, we obtain
\begin{equation}\label{kick}
(I-S(T))u_0=\int_0^TS(T-s)f(s)ds =: b_f,
\end{equation}
written in terms of the ``kick" $b_f$ for the given $T$-periodic $f$. 
The central perspective in \cite{Galdi} concerns the solvability of the above equation. Namely, given $f \in L^1_{\#}(0,T;H)$, we ask whether there is a $u_0 \in H$ (or some smoother space) satisfying \eqref{kick}. From the operator point of view, we see immediately that the null space and range of the bounded linear operator $(I-S(T))$ are at issue, and will be discussed in the sequel.

\subsection{Literature review and established notions for PWP and resonance}\label{sec:notions}

The notion of point-spectral resonance defined before is, in some sense, purely algebraic. The growth associated with the eigenvalue is robust with respect to the topology (norm) in which we frame the operator. On the other hand, the notions of resonance  presented here and in the next section are essentially dependent on topological choices.
We will now recall the notion of resonance in \cite{Galdi} and connect it to our discussion of static solvability in \eqref{kick}.  
\begin{definition}[\cite{Galdi}]
\label{def:galdi}
The system \eqref{eq:main} is \emph{statically resonant} if there exist a
period $T>0$ and a $T$-periodic forcing $f\in L_{\#}^1(0,T;H)$ such that \eqref{eq:main} admits no $T$-periodic mild
solution.  The pair $(T,f)$ as above is then a  \emph{statically resonant pair}.
\end{definition}
 For the next part of our discussion, we invoke Yosida's Mean Ergodic Theorem, as applied to our semigroup $S(t)$. We state this as a lemma, combining \cite[pp.213--214]{Yosida1995} and \cite[p.225]{Galdi}.
\begin{lemma}\label{lem:MET}
Suppose that $A$ generates the $C_0$-semigroup $\big(S(t)\big)_{t\ge 0}$ on the Hilbert space $H$ and that it is uniformly bounded. Then for each $T>0$:
$$H = \mathcal N(I-S(T)) \oplus \overline{\mathcal R(I-S(T))}.$$
Moreover, if $S(t)$ is strongly stable on $H$, then $\mathcal N(I-S(T))=\{0\}$ and 
$$H = \overline{\mathcal R(I-S(T))}.$$
\end{lemma}
In terms of \eqref{kick}, \cite[Theorem 3.3]{Galdi} asserts that static resonance occurs exactly when $\mathcal R(I-S(T))$ is not closed, so one can find $b_f$ not in $\mathcal R(I-S(T))$. Furthermore, \cite[Proposition 3.3]{Galdi} states that when $S(t)$  is {\em uniformly stable} on $H$ (i.e., $||S(t)||_{\mathcal L(H)} \to 0$), then $\mathcal R(I-S(T))=H$ for all $T$. From uniform stability of $S(t)$, then, one obtains  {\em unconditional} PWP on $L^1_{\#}(0,T;H)$.

\begin{remark}Since $S(t)$ is strongly stable (our standing assumption), Lemma~\ref{lem:MET} gives $\mathcal N(I-S(T))=\{0\}$ for every $T>0$, so a $T$-periodic mild solution is unique whenever it exists. Consequently, for fixed $T>0$ and any $\mathcal Q\subset L^1_\#(0,T;H)$, the system \eqref{eq:main} is PWP on $\mathcal Q$ if and only if there is no $f\in\mathcal Q$ for which $(T,f)$ is a statically resonant pair. Static resonance, then, is the negation of PWP, with the caveat that Definition~\ref{def:PWP} fixes the period $T$, where Definition~\ref{def:galdi} ranges over all $T>0$.
\end{remark}

In the partially dissipative setting, models are typically {\bf strongly stable but  not uniformly stable}. Strong stability is usually accessible, through the Arendt-Batty criterion \cite{ArendtBatty1988} and a unique-continuation (Holmgren-type) argument excluding imaginary spectra, whereas deciding uniform stability requires precise control of the resolvent on $i\R$; forcing-based resonance criteria are valuable in exactly this situation. For the heat-wave system, uniform stability is known to fail (Section~\ref{sec:heatwave}).
In the scenario where all that is known is strong stability of $S(t)$, \cite[Theorem 3.4 and Corollary 3.1]{Galdi} give a conditional and non-constructive PWP result. 
\begin{proposition}[Theorem 3.4 in \cite{Galdi}] \label{prop:Galdi}  Let $\big(S(t)\big)_{t \ge 0}$ be a bounded semigroup on the Hilbert space $H$. If $S(t)$ is strongly stable, then for all $T>0$ there exists a dense subset $\mathcal Q_T\subset L^1_{\#}(0,T;H)$ such that \eqref{eq:main} is PWP on $\mathcal Q_T$.
\end{proposition}

In recent work \cite{GMW2026,LaurentRivas2026}, some dense PWP spaces were constructed through the concept of {rational stability}. Rational stability exists as an intermediate notion between strong and uniform stability, and we now provide the definition and a well-known characterization from \cite{BattyDuyckaerts2008,BorichevTomilov2010}.
\begin{theorem}\label{polystab}
Let $\alpha>0$. Then the following are equivalent:
\begin{itemize}
\item~{\it [Rational stability]} $0\in\rho(A)$ and
\[
   \norm{S(t)A^{-1}}_{\cL(H)}
   \le C(1+t)^{-1/\alpha},\qquad t\ge 0;
\]
\item~{\it [Polynomial resolvent growth]} $i\R\subset\rho(A)$ and
\[
   \norm{(isI-A)^{-1}}_{\cL(H)}
   \le C(1+|s|)^{\alpha},\qquad s\in\R .
\]
\end{itemize}
\end{theorem}
In \cite{GMW2026}, it is shown that rational stability of $S(t)$ gives rise to an $L^1_{\#}(0,T;H)$ dense set of the form $H^m_{\#}(0,T;H)$ (where $m>\alpha +1$) on which PWP is obtained.\footnote{These are $H^m$-type spaces, extended periodically to $\mathbb R$; as the index is greater than one, these functions are continuous in time, and hence necessarily exhibit $u(0)=u(T)$.} The reference \cite{LaurentRivas2026} also obtains a similar connection, under decay hypotheses of the semigroup in the time domain. 

The idea that resolvent growth yields PWP spaces was indirectly motivated to the authors by the recent work \cite{MMSW2024}, where the heat-wave dynamics produced ``derivative losses"  observed in constructing periodic solutions through estimates and boundary observability inequalities. This yielded well-posedness results for mild solutions, posed with data on $H^m_{\#}$ type spaces, motivating a broader connection between PWP and data-to-solution derivative loss. The work in \cite{ZhangZuazua2003,RauchZhangZuazua2005,Duyckaerts2007} demonstrates that heat-wave dynamics are in fact rationally stable, and this is shown through observability-type estimates, precisely measuring derivative losses. Owing to the equivalence in Theorem \ref{polystab}, the derivative loss is captured through resolvent growth, which was directly exploited to produce the result in \cite{GMW2026}. It therefore became clear that derivative loss is a central aspect of PWP for \eqref{eq:main}, and that violations thereof may result from bad behavior associated to such losses.

We now turn to our  application-driven definition of resonance: unbounded-in-time growth.
\begin{definition}\label{def:secular}
The system \eqref{eq:main} is \emph{secularly resonant} if there exist a $T>0$, and a $T$-periodic forcing
$f\in L^1_{\#}(0,T;H)$, such that some mild solution
$u\in C([0,\infty);H)$ of \eqref{eq:main} is unbounded:
\begin{equation}\label{eq:unbounded}
\sup_{t\ge0}\norm{u(t)}_{H}=\infty . 
\end{equation}
A pair $(T,f)$ as above is a \emph{secularly resonant pair}.
\end{definition}

It is straightforward to observe that when \eqref{eq:main} is PWP, then a given $f \in L^1_{\#}(0,T;H)$ cannot yield an initial condition $u_0 \in H$ which results in an unbounded solution. Thus, PWP precludes secular resonance. In \cite{LaurentRivas2026} this is shown in detail in a broader context. We state here that, for a uniformly bounded and strongly stable semigroup, if $f$ gives rise to a unique periodic solution emanating from $u_0$, then any other initial data $\tilde u_0$ gives rise to a bounded trajectory that must converge to the periodic solution as $t\to\infty$, since the two trajectories differ by $S(t)(u_0-\tilde u_0)\to0$.
Bounded responses to bounded, not necessarily periodic, forcings are also available under spectral hypotheses (such as separation between $f$ and $A$ \cite[Theorem 5.5]{BattyChill1999}); see also  \cite[Theorem 5.2, Corollary 5.6]{BasitGunzler2013}.

We now informally state the central resonance equivalence: For a partially dissipative system \eqref{eq:main}, where the associated semigroup $S(t)$ is uniformly bounded and strongly stable, 
\begin{center} {\bf secular resonance is equivalent to static resonance}.
\end{center} 
This equivalence is essentially Massera's principle \cite{Massera}; for linear problems in reflexive Banach spaces it follows from \cite[Theorem 1]{Zubelevich2006} and \cite[Theorem 1.1]{Zubelevich2010} (see also \cite{DanersKochMedina1992,KreinMonodromy,Krein1971}). It is stated as part of Theorem~\ref{thm:equiv} here; we mention some corollaries now.

First, if one can show (by any means) that the periodic problem is {\em not solvable} for some $f \in L^1_{\#}(0,T;H)$, then every corresponding mild solution is unbounded; this will be central to our arguments concerning infinite derivative loss in the sequel. In particular, constructing an unbounded mild solution reduces to finding, for some $T>0$, an element $b\in H\setminus \mathcal R(I-S(T))$, since the ``trick" of \cite[p.222]{Galdi} (attributed to J. Pr\"uss) produces a forcing $f_b \in L^1_{\#}(0,T;H)$ with kick $b$. Secondly, if $S(t)$ is strongly stable but not uniformly stable, such $b$ exist for some $T>0$, so resonance is possible.
\begin{proposition}
Let $S(t)$ be a strongly stable $C_0$ semigroup such that $\mathcal{R}(I-S(T))=H$ for every $T>0$. Then $S(t)$ is uniformly stable. 
\end{proposition}
\noindent The above result follows from the abstract framework in \cite{Pruss1984} using a construction as in our Baire  argument in Section~\ref{sec:ellipse}; for this reason we omit the proof here.

\subsection{Principal example: Heat-wave interaction}\label{sec:heatwave}
We now introduce the {\em heat-wave system}, the central example of this study \cite{ZhangZuazua2003,RauchZhangZuazua2005,Duyckaerts2007,AvalosTriggiani2009,AvalosTriggiani2007}. It originates in fluid-structure interactions, which often have more complex vectorial and nonlinear features, and it represents {\em hyperbolic-parabolic coupling} across a common interface \cite{ZhangZuazua2007}; it is a canonical partially dissipative system \cite{AvalosLasieckaTriggiani2016,BattyPaunonenSeifert2019}.

Let $\Om\subset\R^d$ be a bounded, connected, Lipschitz domain, decomposed into bounded, connected, Lipschitz subdomains $\Om_W,\Om_H\subset\Om$ with $\Om_W\cap\Om_H=\emptyset$ and $\overline{\Om}=\overline{\Om_W\cup\Om_H}$, sharing a nonempty interface $\Gamma=\overline{\Om_W}\cap\overline{\Om_H}$ of class $\mathcal C^1$. On the remaining (exterior) boundary $\Gamma_0=\Gamma_W\cup\Gamma_H$, where $\Gamma_W:=\partial\Om_W\setminus\Gamma$ and $\Gamma_H:=\partial\Om_H\setminus\Gamma$, a homogeneous Dirichlet condition is prescribed; we assume throughout that $\Gamma_W$ has positive measure to prevent a nontrivial equilibrium state. Fixing a choice of unit normal  $\mathbf n$ to $\Gamma$, the system can be written
\begin{equation}\label{eq:hw}
\left\{
\begin{aligned}
w_{tt}-\Delta w&=f_W &&\text{in }(0,\infty)\times\Om_W,\\
u_t-\Delta u&=f_H &&\text{in }(0,\infty)\times\Om_H,\\
w_t=u,\quad \partial_{\mathbf n}w&=\partial_{\mathbf n}u &&\text{on }(0,\infty)\times\Gamma,\\
w=0\ \text{on }\Gamma_W,\quad u&=0 &&\text{on }\Gamma_H.
\end{aligned}\right.
\end{equation}
This can be written as a first-order system $\dot U = \mathcal A U + \begin{bmatrix} 0, f_W, f_H \end{bmatrix}^T$ for the state $U=[w,w_t,u]^T$. The heat-wave problem then takes the abstract form \eqref{eq:main} on the finite-energy space
\[
\HH:=H^1_{\Gamma_W}(\Om_W)\times L^2(\Om_W)\times L^2(\Om_H),
\]
where $H^1_{\Gamma_W}(\Om_W)$ denotes the space of $H^1(\Om_W)$-functions vanishing on $\Gamma_W$. The generator $\cA$ encodes the spatial operators together with the interface/coupling conditions, and generates a $C_0$-semigroup of contractions $\big(S(t)\big)_{t\ge0}$ on $\HH$ \cite{ZhangZuazua2007,AvalosLasieckaTriggiani2016,AvalosTriggiani2007}. Two foundational facts frame subsequent analysis of this system. First, $S(t)$ is strongly stable, and this property requires no geometric conditions on the domain  decomposition of $\Omega$ or $(\Gamma_W,\Gamma)$ \cite[Theorem 4]{ZhangZuazua2007}. Secondly, $S(t)$ fails to be uniformly stable: the Gaussian beam constructions of \cite[Section 6]{ZhangZuazua2007} produce approximate solutions whose energy concentrates in $\Om_W$  (\cite[Theorems 5 and 6]{ZhangZuazua2007}). The heat-wave system therefore falls squarely within the class of partially dissipative systems of central concern in this paper.

Any analysis  beyond strong stability, however, is contingent on geometry of tuple $(\Omega_W, \Gamma)$---as is familiar from the analysis of damped wave equations. In these dynamics, decay rates are governed by the interaction between the rays of geometric optics and the ``damping region" \cite{BardosLebeauRauch1992}. For the heat-wave system, the role of the damping region is played by the intrinsic heat dissipation in $\Omega_H$, acting through the interface $\Gamma$. Under the Geometric Control Condition (GCC)---that every ray of geometric optics in $\Om$ enters $\Om_H$ in some uniform time---the system is rationally stable \cite[Theorem 11]{ZhangZuazua2007} (see also \cite{ZhangZuazua2003,RauchZhangZuazua2005,AvalosLasieckaTriggiani2016,Avalos2007}, and \cite{Duyckaerts2007}). One proof method proceeds through a weakened ``observability inequality", with observations taken on the heat domain: for data in $\cD(\cA^3)$ and $T$ sufficiently large,
\begin{equation}\label{obsest}
\norm{U(0)}_{\HH}\le C\,\norm{\nabla u}_{H^3(0,T;L^2(\Om_H))},
\end{equation}
see \cite[Theorem 12]{ZhangZuazua2007} and \cite[Theorem 1]{Duyckaerts2007}. We note that a classical observability inequality carries no time-derivatives on the right-hand side, in contrast to \eqref{obsest}. In this paper we use the expression \emph{loss of derivatives} in this sense: the right-hand side of a central estimate requires more derivatives (in the Sobolev sense) than the corresponding classical estimate would at the finite energy level.

On the periodic side, the abstract result of \cite{GMW2026} furnishes explicit PWP spaces for the heat-wave system under the GCC. Alternatively, one can proceed directly through PDE estimates, as carried out in \cite{MMSW2024}, under geometric conditions broader than the classical GCC that admit certain degenerate spatial domains; the crucial a priori estimates there are reminiscent of the observability estimates above, likewise exhibiting a clear loss of derivatives.

A subtle point,  which we will return to below, distinguishes between full and sampled resolvent information \cite{Pruss1984}. Polynomial growth of $\norm{R_\lambda}$ along all of $i\R$ implies, in particular, polynomial growth along any lattice $i\frac{2\pi}{T}\Z$; it is the sampled information that governs the $T$-periodic problem. The converse is not immediate: a resolvent bound along a lattice need not carry information between sample points. Thus, there are two scenarios of interest: First, geometric conditions yield controlled resolvent growth, and, with it, a Sobolev description of PWP spaces. In the second case, geometric conditions are violated and  polynomial resolvent bounds are not available, yet sampled resolvent control for particular periods $T$ remain accessible. However, when geometric conditions are violated, the quantitative picture is largely open, both for stability as well as  for the periodic problem.

Motivated by the connection between derivative loss and constructing PWP spaces, we shall introduce  a further notion of resonance, formulated entirely through derivative loss.  The main analytical result of the latter part of the paper is the construction of a geometry---which we refer to as the \emph{resonant ellipse}---that exhibits infinite derivative loss for suitable periods $T$. This demonstrates, in particular, that geometric conditions are necessary for PWP. Complementing it, we give an abstract construction of PWP purely in terms of norms of $R_{\lambda}$ restricted to the lattice $i\frac{2\pi}{T}\Z$. Finally, we ask {\em can a smooth forcing produce resonance}? Indeed, it is possible and we will explicitly demonstrate it for the heat-wave system.

\section{New notions of resonance and taxonomy}\label{sec:taxonomy}
To exploit range failure, we interpret the range $\mathcal R(I-S(T))$ through derivative loss, that is, through the growth (or finer structure) of $(\lambda I-A)^{-1}$ on the lattice $i\frac{2\pi}{T}\Z$. Point-spectral resonance, being algebraic in nature, is agnostic to the choice of norms; {\em each notion in this section is topological}, depending on {\em how} forcing and response are measured.

We begin with the spectral notion promised in Section~\ref{sec:terminology}. Strong stability removes the point spectrum from the imaginary axis (Lemma~\ref{prop:no-classical}), but the full spectrum may still reach it.
\begin{definition}[Spectral resonance]\label{def:spectrum}
System \eqref{eq:main} is \emph{spectrally resonant} if
\[
\sap(A)\cap i\R\neq\emptyset,
\]
equivalently, in this setting, if $\sigma(A)\cap i\R\neq\emptyset$.
\end{definition}
We now properly define the notion of derivative loss to support other notions of resonance below. 
\begin{definition}[Derivative loss]\label{def:finite-infinite-loss}
Let $T>0$ and let $m\ge1$ be an integer. We say the operator $A$ \emph{exhibits $m$-derivative loss} for the period $T$ {\bf if there does not} exist a constant $C_m>0$ such that
\[
\|u\|_{L^2(0,T;H)}\le C_m\|f\|_{H^{m-1}_{\#}(0,T;H)}
\]
holds for all $f\in H^{m-1}_{\#}(0,T;H)$ that admit a $T$-periodic mild solution $u$ of \eqref{eq:main}. 

We say that $A$ \emph{exhibits infinite derivative loss} for the period $T$ if it exhibits $m$-derivative loss for every $m\ge1$.
\end{definition}
In the above, we note that such a $u$ is necessarily unique by our strong stability hypothesis; as such, we refer to such a pair $(f,u)$  as an \emph{admissible pair} for the period $T$.

\begin{remark}\label{rem:loss-resolvent}
Definition~\ref{def:finite-infinite-loss} can be equivalently stated through sampled resolvent growth. Assuming $i\omega_n\in\rho(A)$ for all $n\in\Z$, the operator $A$ exhibits $m$-derivative loss for the period $T$ if and only if
\begin{align}\label{eq:sampled-growth}
\limsup_{\abs{n}\to\infty}\,(1+\abs{\omega_n})^{-(m-1)}\,\norm{R_{i\omega_n}}_{\cL(H)}=\infty,
\end{align}
and, consequently, infinite derivative loss holds if and only if the sampled resolvent grows faster than every polynomial along some subsequence of the frequencies $\omega_n$. This observation will be quite useful in the sequel.
\end{remark}

Infinite derivative loss is therefore a fixed-period notion, and Remark~\ref{rem:loss-resolvent} shows that it demands a certain arithmetic alignment by sampling the resolvent along the $T$-lattice. Our final notion removes  alignment, allowing the period to drift. Near a target period, there may exist nearby periods at which the periodic problem is ``catastrophically" ill-conditioned, though the target period itself need not carry infinite loss.

\begin{definition}[High-frequency catastrophe]\label{def:period-accum}
Let $T_*>0$. We say that the operator $A$ \emph{exhibits a high-frequency catastrophe at $T_*$} if for every neighborhood $I\subset(0,\infty)$ of $T_*$  and every integer $m\ge1$, {\bf there does not} exist a constant $C_{m,I}>0$ such that
\[
\|u\|_{L^2(0,T;H)}\le C_{m,I}\,\|f\|_{H^{m-1}_{\#}(0,T;H)}
\]
holds for all $T\in I$ and all $f\in H^{m-1}_{\#}(0,T;H)$ admitting a $T$-periodic mild solution $u$ of \eqref{eq:main}. 

Equivalently, $A$ exhibits a high-frequency catastrophe at $T_*$ if and only if for every integer $m\ge1$ there exist periods $T_j\to T_*$ and admissible pairs $(f_j,u_j)$ for the periods $T_j$ with
\[
\|u_j\|_{L^2(0,T_j;H)}\ge c_0>0,
\qquad
\|f_j\|_{H^{m-1}_{\#}(0,T_j;H)}\longrightarrow 0 .
\]
\end{definition}
\noindent Such a family of pairs $(f_j,u_j)$ as above will be referred to as {\em witnessing pairs}. 

 When $i\R\subset\rho(A)$, the notion of high-frequency catastrophe {\em does not depend on $T_*$} (Theorem~\ref{thm:hfc-baire}); it is a property of the operator itself. In this sense, notions of resonance connected to the previous two definitions are non-spectral and novel; under the same assumption, Theorem~\ref{thm:hfc-baire} shows that a high-frequency catastrophe is equivalent to the failure of every polynomial bound on $\norm{R_{i\omega}}$ along $i\R$.
Definition~\ref{def:period-accum} negates a period-uniform form of the estimate in Definition~\ref{def:finite-infinite-loss}, and is accordingly weaker: infinite derivative loss at $T$ immediately implies a high-frequency catastrophe at $T$, whereas the catastrophe itself requires no alignment of  ``bad" frequencies with any  lattice. In our examples, the mechanism producing such catastrophes is quasimodal and will live at high frequencies, whence the name. Two warnings are in order. A high-frequency catastrophe at $T_*$ does not by itself yield secular resonance at $T_*$, as each witnessing response is bounded and periodic, and the period is not held fixed. It does, however, force infinite derivative loss on a dense $G_\delta$ set of periods, by an abstract Baire argument given in Theorem~\ref{thm:hfc-baire}. 

We now demonstrate that the above notions are not vacuous, for a wave domain that concentrates energy along a billiard orbit invisible to the heat domain.
Consider
\[
\Omega_W:=\{(x,y)\in\R^2:\ \frac{x^2}{a^2}+\frac{y^2}{b^2}<1\},\quad a>b
\]
coupled across a relatively open interface arc $\Gamma\subset\partial\Om_W$ with $\overline\Gamma\subset\partial\Om_W\cap\{x>0\}$ to a bounded, Lipschitz heat domain $\Om_H$ (so that $\overline\Gamma\subset\{x\ge\delta_\Gamma\}$ for some $\delta_\Gamma>0$, by compactness); the remaining boundary $\Gamma_W=\partial\Om_W\setminus\Gamma$ carries the Dirichlet condition. Upon the minor axis $\{x=0,\ \abs{y}\le b\}$ we consider a stable billiard orbit of $\Om_W$ that never meets $\overline\Gamma$, so the geometric conditions of Section~\ref{sec:heatwave} fail. We call this configuration the \emph{resonant ellipse} explained in detail in Definition \ref{def:ellipse}.

\begin{figure}[htbp]
    \centering
    \includegraphics[width=0.58\textwidth]{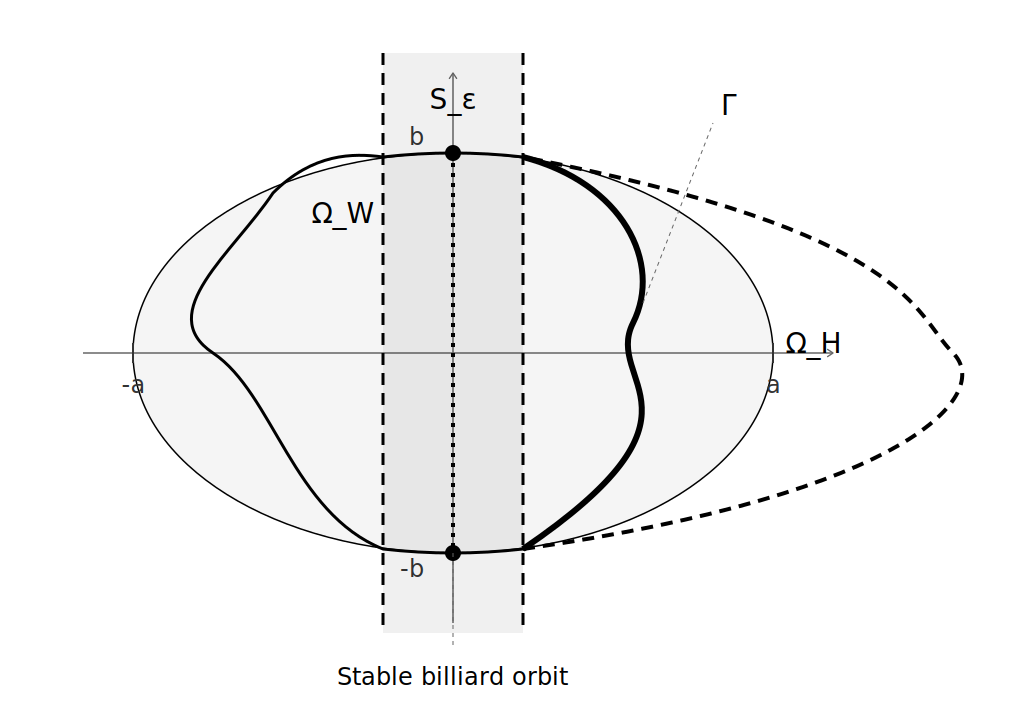}
    \caption{The resonant ellipse configuration.}
    \label{fig:resonant-ellipse}
\end{figure}

\begin{theorem}[Heat-wave resonant ellipse]\label{thm:ellipse}
Consider the heat-wave system \eqref{eq:hw} defined for the resonant ellipse configuration above.
\begin{enumerate}
\item[\textup{(i)}] The system exhibits a high-frequency catastrophe at every $T_*>0$, in the sense of Definition~\ref{def:period-accum}. The witnessing pairs may be taken to be smooth in time with values in $\dom(\cA)$, with vanishing heat component and forcing supported in $\Omega_W$ at positive distance from $\Gamma$.
\item[\textup{(ii)}] The set $\mathcal R_\infty$ of periods $T>0$ at which the system
exhibits infinite derivative loss, in the sense of
Definition~\ref{def:finite-infinite-loss}, is a dense $G_\delta$ subset of $(0,\infty)$.
For every $T\in\mathcal R_\infty$ there is $f\in C^\infty_{\#}(0,T;\HH)$ such that
$(T,f)$ is statically (Definition~\ref{def:galdi}), and therefore secularly
(Definition~\ref{def:secular}), resonant.
\end{enumerate}
\end{theorem}

Part (i) is the manifestation of a specific, focusing trapped orbit. Part (ii) is the
endpoint: it requires an arithmetic alignment of the trapped frequencies with a single period
lattice. {\em We stress that this alignment is obtained by a
non-constructive (Baire) argument.} For a period $T$, given in advance, we cannot use this
machinery to decide whether the ``trapped'' frequencies are sampled to the required
``accuracy''. With
a period $T\in\mathcal R_\infty$ in hand, a resonant forcing is constructed by the
stacking argument of Theorem~\ref{thm:loss-secular}, but its Fourier coefficients are
not identified, since the abstract argument controls only the size of the sampled
resolvent. For the resonant ellipse constructed here, a concrete subset of $\mathcal R_\infty$ can be
described (Proposition~\ref{prop:explicit-R}). Indeed, the set $\mathcal R$ of periods whose
lattice comes exponentially close to infinitely many quasimodal frequencies. By a direct
Baire argument it is again a dense $G_\delta$, now of Lebesgue measure zero; for
$T\in\mathcal R$ the sampled resolvent grows exponentially and the resonant forcing is
explicit, acting only on the wave component, away from the interface. These proofs are
given in Section~\ref{sec:ellipse}.

Theorem~\ref{thm:ellipse} is the simplest instance of Theorem~\ref{thm:ellipse-general} in Section~\ref{sec:ellipse}, which covers every wave domain that coincides with an ellipse in a slab around its minor axis, the interface staying outside the slab (Definition~\ref{def:ellipse}). One family of such domains measures the sharpness of the geometric hypotheses under which finite-loss periodic well-posedness is known. For $-a<d<a$ let
\[
\Om_W^{d}:=E\cap\{x>-d\},\qquad \Gamma^{d}:=\{x=-d\}\cap E,
\]
with any admissible heat domain attached along the chord $\Gamma^d$: the wave domain is the part of the ellipse to the right of the chord, and it contains the minor axis if and only if $d>0$. For $d\le0$ (at most half of the ellipse) $\Om_W^d$ lies between the graphs $x=|d|$ and $x=a\sqrt{1-y^2/b^2}$ over $\abs y<b\sqrt{1-d^2/a^2}$, and the interface is the entire lower graph. This is the graph condition of \cite{MMSW2024}, which implies the so-called graph optics condition, and the corresponding theorem there provides, for every period $T$, a unique finite-energy periodic solution for every wave forcing in $H^8_{\#}(0,T;L^2(\Om_W))$, together with the corresponding estimate. This yields a finite loss, in the sense of Definition~\ref{def:finite-infinite-loss}, for every period. For $d>0$ (anything more than half) the wave domain protrudes beyond both ends of the chord, so the interface is not the lower graph of $\Om_W^d$ in any direction and the graph condition fails. At the same time, Theorem~\ref{thm:ellipse-general} applies, so the system exhibits a high-frequency catastrophe at every period and infinite derivative loss at every period of a dense $G_\delta$: no estimate with any finite loss survives there, and near every period the constant of any such estimate must necessarily blow up. {\em The threshold is exactly the semi-ellipse}. In this sense the graph condition of \cite{MMSW2024} is optimal within this family---the interface cannot be moved past the minor axis by any amount, however small, without losing every Sobolev-type periodic estimate.

\subsection{Implications and equivalences}\label{sec:implications}

We now establish the mathematical relationships among the resonance notions introduced heretofore. They group naturally into: \emph{static} (the range condition \eqref{kick} and Definition~\ref{def:galdi}), \emph{dynamic} (secular resonance: Definition~\ref{def:secular} and  loss-based Definitions~\ref{def:finite-infinite-loss} and~\ref{def:period-accum}), and \emph{spectral} (Definitions~\ref{def:spectral} and~\ref{def:spectrum}); the implications below cut across the groupings.

For $T$-periodic forcing, the mild solution is determined {\em stroboscopically} by the affine Poincar\'e map determined by the kick $b_f$: ~$x\mapsto S(T)x+b_f.$ A $T$-periodic solution is precisely a fixed point of this mapping, and an unbounded solution is an escaping orbit. Massera's principle \cite{Massera,MurakamiNaitoMinh2004} asserts that any bounded orbit forces a fixed point (\cite{Zubelevich2006,Zubelevich2010} in Banach spaces). We recall now.

For a bounded operator \(V\in\mathcal L(X)\), set
\[
    S_0^V:=0,
    \qquad
    S_n^V:=\sum_{k=0}^{n-1}V^k,\qquad n\ge1.
\]
We say that \(V\) is power bounded \cite{KatznelsonTzafriri1986} if
\[
    M_V:=\sup_{n\ge0}\|V^n\|_{\mathcal L(X)}<\infty .
\]

\begin{lemma}\label{lem:ergodic}
Let \(V\in\mathcal L(X)\) be power bounded on a reflexive Banach space \(X\),
and let \(b\in X\). Then
\[
    b\in \mathcal R(I-V)
    \qquad\Longleftrightarrow\qquad
    \sup_{n\ge1}\|S_n^V b\|_X<\infty .
\]
\end{lemma}

The forward implication is immediate. The converse is due to
Browder~\cite{Browder1958}; see also Butzer-Westphal~\cite{ButzerWestphal} (and
\cite{LinSine1983} for related extensions and historical discussion).

We now state the resonance equivalence.
\begin{theorem}[Static $\Leftrightarrow$ secular resonance]\label{thm:equiv}
Let $A$ generate a uniformly bounded $C_0$-semigroup; fix $T>0$ and let $f\in L^1_{\#}(0,T;H)$, and set $V:=S(T)$, $b:=b_f=\int_0^TS(T-s)f(s)ds$ as in \eqref{kick}. 

Then TFAE:
\begin{enumerate}[label=\textup{(\roman*)}]
\item $\dot u = Au+f$~ has no $T$-periodic mild solution;
\item every mild solution of \eqref{eq:main} is unbounded on $[0,\infty)$;
\item some mild solution of \eqref{eq:main} is unbounded on $[0,\infty)$;
\item $b\notin\mathcal R(I-V)$;
\item $\sup_{n\ge1}\norm{S_n^Vb}=\infty$.
\end{enumerate}
In particular, $(T,f)$ is a statically resonant pair if and only if it is a secularly resonant pair.
\end{theorem}

\begin{proof}
$V=S(T)$ is power-bounded by the semigroup property and assumption \eqref{eq:bdd}. For a mild solution $u$ with $u(0)=u_0$, $T$-periodicity of $f$ gives
\begin{equation}\label{eq:monodromy}
u\bigl((n{+}1)T\bigr)=Vu(nT)+b,
\qquad
u(nT)=V^nu_0+S_n^Vb .
\end{equation}
Moreover $u$ is $T$-periodic if and only if $u(T)=u(0)$: in that case $u(\cdot+T)$ and $u$ are mild solutions with the same initial value and the same forcing, hence coincide. Therefore a $T$-periodic mild solution exists if and only if $(I-V)u_0=b$ is solvable, which is (i)$\Leftrightarrow$(iv). Lemma~\ref{lem:ergodic} gives (iv)$\Leftrightarrow$(v). If (v) holds, then \eqref{eq:monodromy} gives $\norm{u(nT)}\ge\norm{S_n^Vb}-M\norm{u_0}$, so every mild solution is unbounded, so (v)$ \Rightarrow$ (ii); (ii)$\Rightarrow$(iii) is immediate. Finally, if a $T$-periodic mild solution $v$ existed, any mild solution $u$ would satisfy $u(t)-v(t)=S(t)\bigl(u(0)-v(0)\bigr)$, hence---taking norms---would immediately be bounded; this is (iii)$\Rightarrow$(i) by contraposition.
\end{proof}

It is immediately verified that point-spectral resonance produces secular growth.
\begin{proposition}[Point-spectral $\Rightarrow$ secular]\label{prop:spec-sec}
Let $A\phi=i\omega\phi$, $\omega\in\R\setminus\{0\}$ and $\phi\in\dom(A)\setminus\{0\}$. Put $T=2\pi/\abs{\omega}$ and $f(t)=e^{i\omega t}\phi$. Then $u(t)=t\,e^{i\omega t}\phi$ is a classical, hence mild, solution of \eqref{eq:main} and $\norm{u(t)}=t\norm{\phi}\to\infty$. Thus point-spectral resonance implies secular resonance.
\end{proposition}

The converse fails. For strongly stable semigroups, Lemma~\ref{prop:no-classical} rules out point-spectral resonance, while secular resonance can still occur (as we have seen for the heat-wave system---Theorem~\ref{thm:ellipse}). The following example is also instructive, albeit in the non-strongly stable setting;  secular growth results from  continuous spectra, at a sublinear rate---growth rates are revisited in Section~\ref{sec:growth}.
\begin{proposition}[Secular $\not\Rightarrow$ point-spectral]\label{prop:notconverse}
There exist a Hilbert space $H$ and a  generator $A$ of a uniformly bounded $C_0$-semigroup on $H$ with $\spp(A)=\emptyset$ such that \eqref{eq:main} is secularly resonant.
\end{proposition}

\begin{proof}
Let \(H=L^2(\mathbb R)\), and let \(M\) be the multiplication operator
\[
    (M\psi)(x)=x\psi(x),
    \qquad
    \mathcal D(M)=\{\psi\in L^2(\mathbb R):x\psi(x)\in L^2(\mathbb R)\}.
\]
Then it is a standard semigroup exercise to observe that \(M\) is self-adjoint,  hence \(A:=iM\) is skew-adjoint, and by Stone's Theorem \cite{Pazy1983} \(A\)
generates the unitary group
\[
    (S(t)\psi)(x)=e^{itx}\psi(x),
    \qquad t\in\mathbb R .
\]
Moreover \(A\) has no eigenvalues, since if \(A\psi=\lambda\psi\), then $
    (ix-\lambda)\psi(x)=0
    ~~\text{for a.e. }x\in\mathbb R,$
 so \(\psi=0\) in
\(L^2(\mathbb R)\). Hence
  $  \sigma_p(A)=\emptyset.$

Fix \(\omega_0\neq0\), set \(T=2\pi/|\omega_0|\), and define
  $  f(t)=e^{i\omega_0 t}g,
    \qquad
    g(x)=e^{-x^2}.$
Then \(f\) is \(T\)-periodic with values in \(L^2(\mathbb R)\) and the zero-initial
mild solution of~
$    \dot u=Au+f$ is
\[
    u(t)=\int_0^t S(t-s)e^{i\omega_0 s}g\,ds .
\]
Pointwise in \(x\), this gives
\[
\begin{aligned}
    u(t,x)
  =g(x)\int_0^t e^{i(t-s)x}e^{i\omega_0 s}\,ds  
    =g(x)\frac{e^{i\omega_0 t}-e^{itx}}{i(\omega_0-x)} .
\end{aligned}
\]
Consequently,
\[
    |u(t,x)|^2
    =
    |g(x)|^2
    \frac{4\sin^2\bigl(t(\omega_0-x)/2\bigr)}
         {(\omega_0-x)^2},
\]
and for $~
    K_t(y):=\frac{4\sin^2(ty/2)}{t y^2},
    \qquad t>0$
we have
\[
    \|u(t)\|_{L^2(\mathbb R)}^2
    =
    t\int_{\mathbb R}|g(x)|^2K_t(\omega_0-x)\,dx .
\]
It is an exercise to check that \(K_t(y)\,dy\) converges weakly as a family of finite measures to
\(2\pi\delta_0\) as $t\to \infty$. 
Since \(|g|^2\) is bounded and continuous, it follows that
\[
    \int_{\mathbb R}|g(x)|^2K_t(\omega_0-x)\,dx
    \longrightarrow
    2\pi |g(\omega_0)|^2, ~~t\to \infty,
\]
and so
\[
    \|u(t)\|_{L^2(\mathbb R)}
    \sim
    \sqrt{2\pi}\,e^{-\omega_0^2}\sqrt{t}
    \qquad\text{as }t\to\infty .
\]
Thus the \(T\)-periodic forcing \(f(t)=e^{i\omega_0t}g\) produces an unbounded
mild solution even though \(A\) has no point spectrum. Hence \((T,f)\) is a
secularly resonant pair.
\end{proof}


The remaining implications concern temporal Fourier coefficients. As such we assert without proof the straightforward regularity of Fourier coefficients in the mild setting. See the related discussion in the classical reference \cite{Pruss1984}.
\begin{lemma}\label{lem:fourier}
Let $f\in L^1_{\#}(0,T;H)$ and let $u$ be a $T$-periodic mild solution of \eqref{eq:main}. Then for every $n\in\Z$ the coefficient $\hat u_n:=\frac1T\int_0^Te^{-i\omega_nt}u(t)\,dt$ belongs to $\dom(A)$ and
\begin{equation}\label{eq:modes}
(i\omega_nI-A)\hat u_n=\hat f_n .
\end{equation}
\end{lemma}

With Lemma~\ref{lem:fourier} in hand, we note that spectra on sampling lattice obstructs PWP.
\begin{theorem}[Spectral $\Rightarrow$ secular]\label{thm:sampled-spectrum}
Suppose \eqref{eq:main} is spectrally resonant, and say $i\omega\in\sigma(A)$ with $\omega\in\R\setminus\{0\}$. Let $T>0$ sample the frequency, so $\omega\in\frac{2\pi}{T}\Z$. Then there is a harmonic forcing $f(t)=e^{i\omega t}g$ for which \eqref{eq:main} admits no $T$-periodic mild solution, and, by Theorem \ref{thm:equiv}, secular resonance is observed for the forcing $f(t)$. 
\end{theorem}

\begin{proof}
If $i\omega\in\spp(A)$, Proposition~\ref{prop:spec-sec} provides an unbounded mild solution and Theorem~\ref{thm:equiv}\,(iii)$\Rightarrow$(i) then excludes a periodic one. 

If $i\omega\in\sigma(A)\setminus\spp(A)$, then $i\omega I-A$ is injective, so $\mathcal R(i\omega I-A)\neq H$ (by the bounded inverse theorem); pick $g\notin\mathcal R(i\omega I-A)$. A $T$-periodic mild solution for $f(t) = e^{i\omega t}g$ would give, by Lemma~\ref{lem:fourier} at $\omega_n=\omega$, an element $\hat u_n\in\dom(A)$ with $(i\omega I-A)\hat u_n=g$, which is a contradiction.
\end{proof}

 Finally, we show that infinite derivative loss forces secular resonance, in fact with \emph{smooth} forcing. We refer to this as a {\em stacking argument}. 
 \begin{theorem}[Infinite loss $\Rightarrow$ secular]\label{thm:loss-secular}
Fix $T>0$, and assume $i\omega_n\in\rho(A)$ for all $n\in\Z$. Assume that $A$ exhibits infinite derivative loss for the period $T$. Then there exists $f\in C^\infty_{\#}(0,T;H)$ so  \eqref{eq:main} admits no $T$-periodic mild solution; consequently, $(T,f)$ is a statically (and secularly) resonant pair.
\end{theorem}

\begin{proof}
By \eqref{eq:sampled-growth} applied for each $m$, we may choose  distinct $n_k$ with $\abs{n_k}\to\infty$ and
\[
\norm{(i\omega_{n_k}I-A)^{-1}}\ \ge\ 2^{2k}\,(1+\abs{\omega_{n_k}})^{k}.
\]
Thence, by the operator norm, we may pick $x_k\in H$ with $\norm{x_k}=1$ and $$\norm{(i\omega_{n_k}I-A)^{-1}x_k}\ge\frac12\norm{(i\omega_{n_k}I-A)^{-1}};$$ then set
\[
y_k:=\frac{(i\omega_{n_k}I-A)^{-1}x_k}{\norm{(i\omega_{n_k}I-A)^{-1}x_k}},
\qquad
r_k:=\frac{x_k}{\norm{(i\omega_{n_k}I-A)^{-1}x_k}},
\]
so that $\norm{y_k}=1$, $(i\omega_{n_k}I-A)y_k=r_k$, and
$$\norm{r_k}\le2\norm{(i\omega_{n_k}I-A)^{-1}}^{-1}\le2^{1-2k}(1+\abs{\omega_{n_k}})^{-k}.$$ Define a periodic forcing
\[
f(t):=\sum_{k\ge1}e^{i\omega_{n_k}t}r_k .
\]
For every $q\ge0$ we have $\sum_k(1+\abs{\omega_{n_k}})^q\norm{r_k}<\infty$; hence the series and all its termwise time-derivatives converge absolutely and uniformly in $H$. We conclude that $f\in C^\infty_{\#}(0,T;H)$, and $\hat f_{n_k}=r_k$ with all other Fourier coefficients being zero.

If a $T$-periodic mild solution $u$ existed, Lemma~\ref{lem:fourier} would give $(i\omega_{n_k}I-A)\hat u_{n_k}=r_k$, and since $i\omega_{n_k}\in\rho(A)$, $\hat u_{n_k}=y_k$ and so $\norm{\hat u_{n_k}}=1$ for every $k$. However, this contradicts Parseval's identity $\sum_n\norm{\hat u_n}^2=\frac1T\norm{u}_{L^2(0,T;H)}^2<\infty$. Theorem~\ref{thm:equiv} gives the remaining conclusions.
\end{proof}
The stacking construction above is abstract, as only $i\omega_n\in\rho(A)$ and \eqref{eq:sampled-growth} for each $m$ are at play, so the theorem applies to any generator with these properties. 
\begin{theorem}[High-frequency catastrophe $\Rightarrow$ infinite loss at generic periods]\label{thm:hfc-baire}
Assume $i\R\subset\rho(A)$. Then the following statements are equivalent:
\begin{enumerate}[label=\textup{(\roman*)}]
\item $A$ exhibits a high-frequency catastrophe at some $T_*>0$;
\item $A$ exhibits a high-frequency catastrophe at every $T_*>0$;
\item $\sup_{\omega\in\R}(1+\abs\omega)^{-m}\norm{R_{i\omega}}_{\cL(H)}=\infty$ for every integer $m\ge0$;
\item $A$ exhibits infinite derivative loss for some period $T>0$;
\item the set $\mathcal R_\infty$ of periods $T>0$ for which $A$ exhibits infinite derivative loss is a dense $G_\delta$ subset of $(0,\infty)$.
\end{enumerate}
\end{theorem}
\begin{proof}
We first introduce notation that emphasizes the dependence on the period. For $T>0$,
$n\in\Z$ and $m\in\N$ we set
\[
\omega_n(T):=\frac{2\pi n}{T},\qquad
g_m(T):=\sup_{n\in\Z}\,(1+\abs{\omega_n(T)})^{-(m-1)}\norm{R_{i\omega_n(T)}}_{\cL(H)}.
\]
As in Remark~\ref{rem:loss-resolvent}, Parseval's identity and the Fourier
characterization of $H^{m-1}_{\#}(0,T;H)$ show that $A$ exhibits a high-frequency
catastrophe at $T_*$ if and only if for every $m\in\N$, $g_m$ is unbounded on every
neighborhood of $T_*$.

\smallskip\noindent\textbf{(i)$\Rightarrow$(iii).} Fix $m\ge1$. By (i) there are
$T_k\to T_*$ and $n_k\in\Z$ with $(1+\abs{\nu_k})^{-(m-1)}\norm{R_{i\nu_k}}\ge k$, where
$\nu_k:=\omega_{n_k}(T_k)$. Hence the supremum in (iii) is infinite for the exponent
$m-1$, and since $m\ge1$ is arbitrary, (iii) follows.

\smallskip\noindent\textbf{(iii)$\Rightarrow$(v).} For $m,K,N\in\N$ let
\[
\mathcal O_{m,K,N}:=\bigl\{T>0:\ (1+\abs{\omega_n(T)})^{-m}\norm{R_{i\omega_n(T)}}>K\
\text{for some }n\in\Z\text{ with }\abs n\ge N\bigr\}.
\]
We claim that each $\mathcal O_{m,K,N}$ is open and dense. It is open, being a union of
preimages of $(K,\infty)$ under the continuous maps
$T\mapsto(1+\abs{\omega_n(T)})^{-m}\norm{R_{i\omega_n(T)}}$. To prove density, let
$0<p<q$; we show that $\mathcal O_{m,K,N}\cap(p,q)\neq\emptyset$. By (iii) and the
boundedness of the resolvent on compact subsets of $i\R$, we can choose $\omega\in\R$ with
\[
(1+\abs\omega)^{-m}\norm{R_{i\omega}}>K
\qquad\text{and}\qquad
\abs\omega>\max\Bigl\{\frac{2\pi N}{p},\,\frac{2\pi}{q-p}\Bigr\}.
\]
The interval $(\abs\omega p/2\pi,\,\abs\omega q/2\pi)$ has length greater than $1$, and
thus contains an integer $n\in\N$. Setting $T:=2\pi n/\abs\omega$, we obtain
$T\in\mathcal O_{m,K,N}\cap(p,q)$, which proves density. By
Remark~\ref{rem:loss-resolvent}, $T\in\mathcal R_\infty$ if and only if
$T\in\mathcal O_{m,K,N}$ for all $m,K,N\in\N$, so
$\mathcal R_\infty=\bigcap_{m,K,N}\mathcal O_{m,K,N}$ is a dense $G_\delta$ by the Baire
category theorem.

\smallskip\noindent\textbf{(v)$\Rightarrow$(iv).} Trivial.

\smallskip\noindent\textbf{(iv)$\Rightarrow$(iii).} This follows directly from
Remark~\ref{rem:loss-resolvent}.

\smallskip\noindent\textbf{(iii)$\Rightarrow$(ii).} Let $T_*>0$, let $I$ be a
neighborhood of $T_*$, and let $m\in\N$ and $K>0$. Choose $0<p<q$ with $(p,q)\subset I$.
As in the proof of (iii)$\Rightarrow$(v), there are $\omega\in\R$ with
$(1+\abs\omega)^{-(m-1)}\norm{R_{i\omega}}>K$ and $T\in(p,q)$ with
$\omega\in\frac{2\pi}{T}\Z$. Then $g_m(T)>K$, so $g_m$ is unbounded on $I$.

\smallskip\noindent\textbf{(ii)$\Rightarrow$(i).} Trivial.
\end{proof}

\section{Further properties of resonance}\label{sec:props}
\subsection{Growth rates and stacking}\label{sec:growth}

Secular resonance asserts unbounded growth and we quantify its possible rates. Classical resonance indeed provides the ceiling. From the uniform boundedness of the semigroup, we can read off the growth rate from the variation of parameters formula.
\begin{proposition}\label{prop:linear-upper}
Let $A$ generate a uniformly bounded $C_0$-semigroup on $H$ satisfying \eqref{eq:bdd} and let $f\in L^1_{\#}(0,T;H)$. Every mild solution of \eqref{eq:main} satisfies
\[
\norm{u(NT)}\le M\norm{u(0)}+NM\norm{b_f},\qquad N\in\N,
\]
and, for $0\le\tau\le T$, $\norm{u(NT+\tau)}\le M\norm{u(NT)}+M\int_0^T\norm{f(s)}\,ds$.
\end{proposition}

Point-spectral resonance yields a solution $u(t)=te^{i\omega t}\phi$ of Proposition~\ref{prop:spec-sec}, which achieves linear growth. Within the partially dissipative class, as we shall see, the ceiling is not reached.

\begin{corollary}[Sublinear growth under strong stability]\label{cor:sublinear-ss}
If $S(t)$ is strongly stable, in addition to the assumptions of Proposition~\ref{prop:linear-upper}, then every mild solution of \eqref{eq:main} with $T$-periodic $f\in L^1_{\#}(0,T;H)$ satisfies $\norm{u(t)}_H=o(t)$ as $t\to\infty$. Secular growth in the partially dissipative class is thus sublinear.
\end{corollary}

\begin{proof}
Recall the notations $V=S(T)$ and $S_N^V b=\sum_{k=0}^{N-1}V^kb$. As observed in \cite[pp. 220--221]{Galdi}, we can infer from the Mean Ergodic Theorem (Lemma~\ref{lem:MET}) in the strongly stable case that the Ces\`aro averages go to zero $\displaystyle\frac1NS_N^Vb\to0$ for every $b\in H$. 

Hence $\norm{u(NT)}\le M\norm{u(0)}+N\norm{\frac1NS_N^Vb_f}=o(N)$, and the second estimate of Proposition~\ref{prop:linear-upper} extends this to $\norm{u(t)}=o(t)$. 
\end{proof}

\begin{remark}\label{rem:sublinear} The gain over Proposition~\ref{prop:linear-upper} comes from strong stability, as power-boundedness alone gives only $\norm{S_N^Vb_f}\le NM\norm{b_f}$, and $O(N)$ is thereby improved to $o(N)$. Genuinely sublinear rates occur: the continuous-spectrum example of Proposition~\ref{prop:notconverse} grows as $\sqrt t$. An interesting area of inquiry is the classification of the rates attainable within the strongly stable framework. \end{remark}

We again turn to the stacking mechanism, this time in the context of \emph{finite derivative loss}. The two forthcoming statements should be calibrated against Theorem~\ref{thm:loss-secular}. A polynomial sampled bound protects all smooth forcings, while finite derivative loss produces resonance from forcings of finite Sobolev regularity.
\begin{proposition}[Finite loss protects smooth forcings]\label{prop:conditional-converse-loss}
Fix $T>0$ and suppose $i\omega_n\in\rho(A)$ for every $n\in\Z$ together with the sampled bound
\begin{equation}\label{eq:sampled-poly}
\norm{(i\omega_nI-A)^{-1}}\le C(1+\abs{\omega_n})^m,\qquad n\in\Z,
\end{equation}
for some $m\in\N$ and $C>0$. Then every $f\in C^\infty_{\#}(0,T;H)$ admits a unique classical $T$-periodic solution in $C^\infty_{\#}(0,T;H)$. Consequently, a smooth statically resonant forcing necessitates infinite derivative loss.
\end{proposition}

\begin{proof}
Classical Fourier theory ensures that smoothness of $f$ yields $\sum_n(1+\abs{\omega_n})^{2q}\norm{\hat f_n}^2<\infty$ for every $q\ge0$. Set $u_n:=(i\omega_nI-A)^{-1}\hat f_n$; by \eqref{eq:sampled-poly},
\[
\sum_n(1+\abs{\omega_n})^{2q}\norm{u_n}^2\lesssim\sum_n(1+\abs{\omega_n})^{2(q+m)}\norm{\hat f_n}^2<\infty
\]
for every $q$, so $u(t):=\sum_ne^{i\omega_nt}u_n$ converges, with all its time-derivatives, absolutely and uniformly in $H$. Since $Au_n=i\omega_nu_n-\hat f_n$, the series for $Au$ converges in the same manner. Thus, $u(t)\in\dom(A)$, $u\in C^\infty_{\#}(0,T;H)$, and $\dot u=Au+f$ holds coefficient-wise. This means that $u$ is a classical $T$-periodic solution and uniqueness follows from Lemma~\ref{lem:fourier} with $i\omega_n\in\rho(A)$.
\end{proof}
\begin{remark}Proposition~\ref{prop:conditional-converse-loss} is a sampled, periodic counterpart of \cite[Theorem 5.5]{BattyChill1999}, where a polynomial resolvent bound along all of $i\R$ gives bounded responses to forcings in $W^{k,\infty}(\R_+;H)$ with $k$ depending on the growth exponent.\end{remark}

\begin{corollary}[Smooth resonance $\Leftrightarrow$ infinite loss]\label{cor:smooth-iff}
Fix $T>0$ and suppose $i\omega_n\in\rho(A)$ for every $n\in\Z$. Then some $f\in C^\infty_{\#}(0,T;H)$ forms a statically, equivalently secularly, resonant pair $(T,f)$ if and only if $A$ exhibits infinite derivative loss for the period $T$, that is, if and only if
\[
\sup_{n\in\Z}\,(1+\abs{\omega_n})^{-m}\norm{R_{i\omega_n}}_{\cL(H)}=\infty\qquad\text{for every integer }m\ge0 .
\]
\end{corollary}
\begin{proof}
Theorem~\ref{thm:loss-secular} gives sufficiency. If infinite derivative loss fails, Remark~\ref{rem:loss-resolvent} gives \eqref{eq:sampled-poly} for some exponent, and Proposition~\ref{prop:conditional-converse-loss} provides a $T$-periodic solution for every $f\in C^\infty_{\#}(0,T;H)$. Theorem~\ref{thm:equiv} identifies the static and secular resonance of a pair.
\end{proof}
Corollary~\ref{cor:smooth-iff} concerns one period at a time. Letting the period vary, we obtain a statement about the semigroup itself.
\begin{corollary}[Rational stability $\Leftrightarrow$ no smooth resonance]\label{cor:poly-smooth}
Assume $i\R\subset\rho(A)$.
\begin{enumerate}[label=\textup{(\alph*)}]
\item If $S(t)$ is rationally stable for some $\alpha>0$, in the sense of Theorem~\ref{polystab}, then for every $T>0$ every $f\in C^\infty_{\#}(0,T;H)$ admits a unique $T$-periodic mild solution, and every mild solution of \eqref{eq:main} with this forcing is bounded.
\item If $S(t)$ is not rationally stable for any $\alpha>0$, then the set of periods $T>0$ for which some $f\in C^\infty_{\#}(0,T;H)$ forms a secularly resonant pair $(T,f)$ is a dense $G_\delta$ subset of $(0,\infty)$, namely the set $\mathcal R_\infty$ of Theorem~\ref{thm:hfc-baire}.
\end{enumerate}
In particular, a secularly resonant pair $(T,f)$ with $f\in C^\infty_{\#}(0,T;H)$ exists if and only if $S(t)$ is not rationally stable for any $\alpha>0$.
\end{corollary}
\begin{proof}
(a) By Theorem~\ref{polystab}, $\norm{R_{is}}_{\cL(H)}\le C(1+\abs s)^{\alpha}$ for all $s\in\R$, so \eqref{eq:sampled-poly} holds for every $T>0$ and any integer $m\ge\alpha$. Proposition~\ref{prop:conditional-converse-loss} gives the periodic solution, and Theorem~\ref{thm:equiv} the boundedness of every mild solution.

(b) By Theorem~\ref{polystab} and $i\R\subset\rho(A)$, no polynomial bound on $\norm{R_{is}}_{\cL(H)}$ holds along $i\R$, which is statement (iii) of Theorem~\ref{thm:hfc-baire}. Hence $\mathcal R_\infty$ is a dense $G_\delta$ subset of $(0,\infty)$, and by Corollary~\ref{cor:smooth-iff} it is exactly the set of periods admitting a smooth secularly resonant forcing.
\end{proof}
Part (a) is the smooth-forcing case of \cite[Theorem~4]{GMW2026}; see also \cite[Theorem~5.5]{BattyChill1999} and \cite[Theorems~1.1 and~1.2]{LaurentRivas2026}. Part (b) is the converse: without a polynomial decay rate, periodic well-posedness of smooth forcings fails at generic periods; compare the example of \cite[Theorem~1.4]{LaurentRivas2026} discussed in Section~\ref{sec:pdamped}.

Stacking also produces resonance at finite loss, at the cost of regularity.
\begin{proposition}[Stacking at finite loss]\label{prop:finite-stacking}
Fix $T>0$ and an integer $m\ge1$, and suppose $i\omega_n\in\rho(A)$ for every $n\in\Z$. If $A$ exhibits $m$-derivative loss for the period $T$, then there exists $f\in H^{m-1}_{\#}(0,T;H)$ admitting no $T$-periodic mild solution; the pair $(T,f)$ is statically, hence secularly, resonant.
\end{proposition}
\begin{proof}
By \eqref{eq:sampled-growth} we may choose distinct $n_k$ with $\abs{n_k}\to\infty$ and
\[
\norm{(i\omega_{n_k}I-A)^{-1}}\ \ge\ 2k\,(1+\abs{\omega_{n_k}})^{m-1}.
\]
As in the proof of Theorem~\ref{thm:loss-secular}, choose unit vectors $x_k$ with $\norm{(i\omega_{n_k}I-A)^{-1}x_k}\ge\frac12\norm{(i\omega_{n_k}I-A)^{-1}}$ and again set
\[
y_k:=\frac{(i\omega_{n_k}I-A)^{-1}x_k}{\norm{(i\omega_{n_k}I-A)^{-1}x_k}},
\qquad
r_k:=\frac{x_k}{\norm{(i\omega_{n_k}I-A)^{-1}x_k}},
\]
so that $\norm{y_k}=1$, $(i\omega_{n_k}I-A)y_k=r_k$ and $\norm{r_k}\le k^{-1}(1+\abs{\omega_{n_k}})^{-(m-1)}$. Define
\[
f(t):=\sum_{k\ge1}e^{i\omega_{n_k}t}r_k .
\]
Since
\[
\sum_k(1+\abs{\omega_{n_k}})^{2(m-1)}\norm{r_k}^2\le\sum_kk^{-2}<\infty,
\]
we have $f\in H^{m-1}_{\#}(0,T;H)$. If a $T$-periodic mild solution $u$ existed, Lemma~\ref{lem:fourier} and $i\omega_{n_k}\in\rho(A)$ would give $\hat u_{n_k}=y_k$ for every $k$, hence $\sum_n\norm{\hat u_n}^2=\infty$, contradicting Parseval's identity for $u\in L^2(0,T;H)$.
\end{proof}

Together, the two propositions calibrate resonance within the Sobolev scale. We say that $A$ has loss of order $m$ for the period $T$ if it exhibits $m$-derivative loss but not $(m+1)$-derivative loss; by Remark~\ref{rem:loss-resolvent}, $(1+\abs{\omega_n})^{-m}\norm{R_{i\omega_n}}$ is then bounded while $(1+\abs{\omega_n})^{-(m-1)}\norm{R_{i\omega_n}}$ is not. In this case every $f\in H^m_{\#}(0,T;H)$ admits a $T$-periodic mild solution (Proposition~\ref{prop:weighted-pwp} below), so smooth forcings are ``safe", yet by Proposition~\ref{prop:finite-stacking} resonance occurs from some forcing in $H^{m-1}_{\#}(0,T;H)$. In particular, {\em for finite loss, the resonant forcing guaranteed by the theory in \cite{Galdi} is necessarily non-smooth}, and the loss order locates a clear Sobolev regularity gap for resonant forcings.
\subsection{General sampled growth spaces}\label{sec:loss}

In the final part of this section, we describe how sampled resolvent information can allow us to build  PWP spaces directly, without any necessary polynomial structure. We will assume throughout this subsection that $i\omega_n\in\rho(A)$ for every $n\in\Z$. For $F=\sum_ne^{i\omega_nt}F_n\in L^2_{\#}(0,T;H)$, we will associate the formal periodic response $$u \sim \sum_ne^{i\omega_nt}(i\omega_nI-A)^{-1}F_n,$$ and square-integrability of the response is exactly membership in the class we will now describe.

\begin{definition}[Sampled response space]\label{def:weighted-space}
\begin{align*}
\mathfrak W_T(A):=&~\Bigl\{F=\textstyle\sum_ne^{i\omega_nt}F_n\in L^2_{\#}(0,T;H):\ \sum_{n\in\Z}\norm{(i\omega_nI-A)^{-1}F_n}^2<\infty\Bigr\},\\
\norm{F}_{\mathfrak W_T(A)}^2:=&~\sum_{n\in\Z}\Bigl(\norm{F_n}^2+\norm{(i\omega_nI-A)^{-1}F_n}^2\Bigr).
\end{align*}
\end{definition}
Since each $(i\omega_nI-A)^{-1}$ is bounded, $F\mapsto U:=\sum_ne^{i\omega_nt}(i\omega_nI-A)^{-1}F_n$ is a closed operator in $L^2_{\#}(0,T;H)$ with domain $\mathfrak W_T(A)$, and $T\norm{F}_{\mathfrak W_T(A)}^2=\norm{F}_{L^2(0,T;H)}^2+\norm{U}_{L^2(0,T;H)}^2$ is the square of its graph norm; $\mathfrak W_T(A)$ is therefore a Hilbert space. By Proposition~\ref{prop:weighted-pwp}(ii) below, $\mathfrak W_T(A)$ is exactly the set of forcings in $L^2_{\#}(0,T;H)$ that admit a $T$-periodic mild solution, not merely a sufficient class. Part (i) of that proposition connects to an explicit observation in \cite{Galdi}. Namely, the density of trigonometric polynomials used in (i) of Proposition~\ref{prop:weighted-pwp} is \cite[Lemma 3.2]{Galdi}, and the solvability for finitely many modes is the mechanism of \cite[Theorem 3.5]{Galdi}.

\begin{proposition}[Sampled-growth periodic well-posedness]\label{prop:weighted-pwp}
Assume $i\omega_n\in\rho(A)$ for all $n\in\Z$.
\begin{enumerate}[label=\textup{(\roman*)}]
\item Every $H$-valued trigonometric polynomial $F$ admits a unique $T$-periodic mild solution, which is in fact classical with values in $\dom(A)$. Since the trigonometric polynomials are dense in $L^p_{\#}(0,T;H)$, $1\le p<\infty$, they form an explicit dense periodic well-posedness class.
\item If $F\in L^2_{\#}(0,T;H)$ admits a $T$-periodic mild solution, then $F\in\mathfrak W_T(A)$. Conversely, every $F\in\mathfrak W_T(A)$ admits a unique $T$-periodic mild solution $U$, with $$\norm{U}_{L^2(0,T;H)}^2+\norm{F}_{L^2(0,T;H)}^2=T\norm{F}_{\mathfrak W_T(A)}^2.$$\end{enumerate}
\end{proposition}

\begin{proof}
(i) is taken from \cite[pp.226--227]{Galdi}. For (ii), a mild solution is continuous, hence in $L^2(0,T;H)$, and Lemma~\ref{lem:fourier} gives $\hat U_n=(i\omega_nI-A)^{-1}F_n$. From Parseval we have $\frac1T\norm{U}_{L^2}^2=\sum_n\norm{(i\omega_nI-A)^{-1}F_n}^2$, which proves the forward direction and the identity. For the converse, the same formula defines $U\in L^2_{\#}(0,T;H)$, and $U$ is a mild solution by standard Fourier techniques (the Fej\'er means of $U$ are classical $T$-periodic solutions of \eqref{eq:main} with the Fej\'er means of $F$ as forcing, and passing to a subsequential limit yields that $U$ coincides a.e.\ with a continuous $T$-periodic mild solution). Uniqueness follows from Lemma~\ref{lem:fourier} and $i\omega_n\in\rho(A)$.\end{proof}

In the definition of $\mathfrak W_T(A)$, one may consider removing the $L^2$ component of the norm. However, in that case, one must confirm that the series represents a mild solution with the regularity at hand. We do not belabor this point, since the space $\mathfrak W_T(A)$ can then be quite general and exhibit pathological regularity properties. 

Once the forcing is measured by the size of its own response, periodic solvability is automatic, no matter how rapidly the sampled resolvent grows. More usable subclasses arise from ``majorants": namely, if~ $M_n(T)\ge\norm{(i\omega_nI-A)^{-1}}$, then
\[
\Bigl\{F:\ \sum_nM_n(T)^2\norm{F_n}^2<\infty\Bigr\}\subseteq\mathfrak W_T(A).
\]
Note that when $M_n(T)\lesssim(1+\abs{\omega_n})^m$, this class contains $H^m_{\#}(0,T;H)$. In this way, one recovers the finite derivative-loss PWP of \cite{GMW2026,LaurentRivas2026}, with the forcing paying the price of $m$ derivatives. However, when sampled growth is, for instance,  logarithmic, exponential along a sparse subsequence, or otherwise irregular, the natural forcing class is no longer expected to be a proper Sobolev space of the form $H^m(0,T;H)$, but a weighted space itself. 

We will see below that the above discussion will precisely apply to the resonant ellipse scenario for the heat-wave dynamics. For the periods of Theorem~\ref{thm:ellipse}(ii), no Sobolev-type estimate will be available, yet $\mathfrak W_T(\cA)$ will remain a dense class on which the periodic response is bounded. {\em The failure of finite loss, therefore, does not mean failure of PWP.} We note that similar observations have been made concerning rates of decay and irregular resolvent growth in the semigroup/stability context \cite{RozendaalSeifertStahn2019,LaurentRivas2026}.

\section{Model problems}\label{sec:examples}

We now apply the notions of Sections~\ref{sec:taxonomy} and \ref{sec:props} on three models. First, a wave equation with exponentially weak damping, where every quantity is made explicit; secondly, the heat-wave system \eqref{eq:hw}, whose spectral and periodic properties we record for later involved discussion; and finally, a locally damped wave, which places the geometric mechanisms in a familiar setting. In each subsection $\cA$ and $\HH$ denote the generator and the energy space of the model under discussion.

\subsection{Pseudo-spectrally damped wave equation}\label{sec:pdwave}

The simplest realization of an example of infinite derivative loss can be arrived at for a wave equation. We consider  damping that is exponentially weak at high spatial frequencies. The damped generator's spectrum misses the imaginary axis, so no spectral obstruction is available, and the stacking argument from the proof of Theorem~\ref{thm:loss-secular} applied directly in closed form.

Let $x\in\T=\R/2\pi\Z$, let $D=-i\partial_x$, and let $|D|$ be the operator obtained from the symbol $\abs{n}$ so we observe $e^{-|D|}e^{inx} = e^{-n}e^{inx}$. The operator $e^{-|D|}$ then is a positive, smoothing, self-adjoint contraction, which is exponentially weak on high frequencies. Consider
\begin{equation}\label{eq:pdwave}
u_{tt}-u_{xx}+e^{-|D|}u_t=f,\qquad x\in\T,
\end{equation}
restricted to mean-zero functions. With $U=(u,u_t)$, the semigroup generator is
\begin{gather*}
\cA=\begin{pmatrix}0&I\\ \partial_x^2&-e^{-|D|}\end{pmatrix}
\quad\text{on}\quad
\HH=\dot H^1(\T)\times L^2_0(\T),\\
\norm{(u,v)}_{\HH}^2=\norm{u_x}_{L^2}^2+\norm{v}_{L^2}^2,
\end{gather*}
and the energy satisfies $$\frac{d}{dt}\norm{U}_{\HH}^2=-2\norm{e^{-|D|/2}u_t}_{L^2}^2\le0.$$ On the mode $e^{inx}$, the modal equation reads $\ddot q+e^{-\abs n}\dot q+n^2q=0$, with roots
\[
\lambda_{n,\pm}=-\frac{e^{-\abs n}}{2}\pm i\sqrt{n^2-\frac{e^{-2\abs n}}{4}} .
\]
We observe that the eigenvalues of $\cA$ are precisely the $\lambda_{n,\pm}$, $n\neq0$, a set with no finite accumulation point, lying in the open left half-plane with real parts tending to $0$ as $\abs n\to\infty$. Since $\abs{\operatorname{Im}\lambda_{n,\pm}}\sim\abs n$, the ``modal" resolvent points are uniformly bounded at each imaginary point, so $i\R\subset\rho(\cA)$.
The system is therefore not spectrally resonant in the sense of Definition~\ref{def:spectrum}. Each mode decays and the semigroup is contractive, from which we may infer that $S(t)$ is strongly stable. The modal rates $e^{-\abs n}/2$ tend to $0$, so uniform stability is not possible. The model is truly partially dissipative in the sense of Section~\ref{sec:terminology} in that $S(t)$ is strongly but not uniformly stable. 

\begin{proposition}[Infinite loss and smooth resonance for \eqref{eq:pdwave}]\label{prop:pdwave}
Take $T=2\pi$, so that $\omega_n=n$. Then
\[
\norm{(inI-\cA)^{-1}}_{\cL(\HH)}\ \ge\ \sqrt2\,e^{n},\qquad n\ge1,
\]
and $\cA$ exhibits infinite derivative loss for the period $2\pi$. Moreover, with $e_n:=(2\pi)^{-1/2}e^{inx}$, the forcing
\[
F(t)=\sum_{n\ge1}e^{-n}e^{int}\,(0,e_n)\ \in\ C^\infty_{\#}(0,2\pi;\HH)
\]
admits no $2\pi$-periodic mild solution; the pair $(2\pi,F)$ is statically, hence secularly, resonant.
\end{proposition}

\begin{proof}
Put $Z_n=(e_n,in\,e_n)$. Then $\norm{Z_n}_{\HH}^2=n^2+n^2=2n^2$ and
\[
(inI-\cA)Z_n=\bigl(in\,e_n-in\,e_n,\ -n^2e_n+n^2e_n+in\,e^{-n}e_n\bigr)=(0,\ in\,e^{-n}e_n),
\]
of norm $\abs n e^{-n}$. With $Y_n:=Z_n/\norm{Z_n}_{\HH}$ we get $\norm{Y_n}_{\HH}=1$ and $\norm{(inI-\cA)Y_n}_{\HH}=e^{-n}/\sqrt2$, from which we obtain the resolvent bound. Since $e^n$ beats every polynomial growth, \eqref{eq:sampled-growth} holds for every $m$, which provides infinite derivative loss by Remark~\ref{rem:loss-resolvent}.

For the forcing, we choose $\widehat F_n=(0,e^{-n}e_n)$, and, by construction, this has $\norm{\widehat F_n}_{\HH}=e^{-n}$. Then we observe $\sum_n(1+n)^q\norm{\widehat F_n}_{\HH}<\infty$ for every $q\ge0$ and so $F\in C^\infty_{\#}(0,T;\mathcal H)$. Solving $(inI-\cA)(u,v)=(0,e^{-n}e_n)$ with $u=Qe_n$ gives $v=inQe_n$ and $in\,e^{-n}Q=e^{-n}$, that is $Q=1/(in)$; the resulting mode $\widehat U_n=(Qe_n,inQe_n)$ has $\norm{\widehat U_n}_{\HH}^2=1+1=2$ for every $n$. If a $2\pi$-periodic mild solution existed, Lemma~\ref{lem:fourier} and $in\in\rho(\cA)$  force these $\widehat U_n$ to be its Fourier coefficients, contradicting Parseval. Theorem~\ref{thm:equiv} gives the remaining conclusions.
\end{proof}
Proposition~\ref{prop:pdwave} demonstrates that this system is secularly resonant without being spectrally resonant. We do not observe the spectrum ``touching" the imaginary axis, but rather asymptotically ``closing in" on it.

\subsection{The heat-wave system: spectrum and known periodic well-posedness}\label{sec:hwfacts}

We return to \eqref{eq:hw} and record the two essential facts about it used in the sequel. The first is that, like the model of Section~\ref{sec:pdwave} and in contrast with the classical picture, the heat-wave generator has no spectrum on the imaginary axis. In fact, this holds with no geometric hypothesis beyond connectedness and a  general unique continuation property. We state the theorem here, but forgo the proof, as it reduces to a by now standard application of Holmgren's unique continuation theorem in the context of a compactness-uniqueness argument that precludes approximate eigenvectors.
\begin{lemma}[No imaginary spectrum]\label{lem:hw-nospec}
Assume, in addition to the standing hypotheses of Section~\ref{sec:heatwave}, that $\Om_W$ is connected and that $\Gamma$ contains a smooth open patch. Then 
\[
\sigma(\cA)\cap i\R=\emptyset,\qquad\text{equivalently}\qquad i\R\subset\rho(\cA).
\]
\end{lemma}

\begin{remark}\label{rem:hw-geom}
The geometry enters only through unique continuation, as the above uses the transmission structure and generalized trace theory on arbitrary admissible domains. We tacitly use the fact that a Dirichlet eigenfunction cannot have vanishing Neumann trace on $\Gamma$. 
Compare this result to the geometry-free strong stability of \cite[Theorem 4]{ZhangZuazua2007}, proven by different methods.
\end{remark}

Lemma~\ref{lem:hw-nospec} has two consequences worth isolating. First, the heat-wave system is never spectrally resonant in the sense of Definition~\ref{def:spectrum}, so Theorem~\ref{thm:sampled-spectrum} is never available for it: {\bf whatever resonance occurs must be non-spectral}. Second, for every $T>0$ all sampled frequencies satisfy $i\omega_n\in\rho(\cA)$, so the hypotheses of Theorem~\ref{thm:loss-secular} and of Propositions~\ref{prop:conditional-converse-loss}, \ref{prop:finite-stacking} and~\ref{prop:weighted-pwp} are met rather unconditionally in the general geometry.
The second issue is the state of the periodic theory at this juncture, recalled from Section~\ref{sec:heatwave}: under the geometric conditions of either \cite{ZhangZuazua2007} or \cite{MMSW2024}, periodic well-posedness holds in $H^m(0,T; \cdots)$ spaces of Sobolev type, with a clear finite loss of derivatives. Resonance is nevertheless already present  in these geometries, at the cost of lower regularity, and this is visible in the constructed examples \cite[Section 2.2]{MMSW2024} and recalled here.

\begin{proposition}[Resonance in the rectangle]\label{prop:mmsw}
Let $\Om_W=(0,\pi)\times(0,1)$, $\Om_H=(0,\pi)\times(-1,0)$, $\Gamma=(0,\pi)\times\{0\}$, and $T=2\pi$. Then there is $c>0$ with
\[
\norm{(inI-\cA)^{-1}}_{\cL(\HH)}\ \ge\ c\,n,\qquad n\ge1,
\]
so $\cA$ exhibits $1$-derivative loss for the period $2\pi$. Consequently there exists $F\in L^2_{\#}(0,2\pi;\HH)$ admitting no $2\pi$-periodic mild solution, and $(2\pi,F)$ is statically, hence secularly, resonant.
\end{proposition}

\begin{proof}
Fix $\phi\in C^2_0(0,1)$, $\phi\not\equiv0$, and set $w_n(t,x,y)=\sin(nt)\sin(nx)\phi(y)$. A direct computation gives $\partial_{tt}w_n-\Delta w_n=g_n:=-\sin(nt)\sin(nx)\phi''(y)$, and $w_n$ vanishes on $\partial\Om_W$. Since $\phi(0)=\phi'(0)=0$, both interface conditions in \eqref{eq:hw} hold with vanishing heat component: $\partial_tw_n|_\Gamma=0$ and $\partial_{\bm n}w_n|_\Gamma=0$. Hence $U_n:=(w_n,\partial_tw_n,0)$ takes values in $\dom(\cA)$ and is a classical $2\pi$-periodic solution of \eqref{eq:main} with forcing $F_n:=(0,g_n,0)$.

Writing $\sin(nt)=(e^{int}-e^{-int})/2i$ and comparing the $e^{int}$ coefficients,
\[
\widehat{(U_n)}_n=\tfrac1{2i}\bigl(\sin(nx)\phi,\ in\sin(nx)\phi,\ 0\bigr),
\qquad
\widehat{(F_n)}_n=\tfrac1{2i}\bigl(0,\ -\sin(nx)\phi'',\ 0\bigr),
\]
so that
\[
\norm{\widehat{(U_n)}_n}_{\HH}^2=\tfrac\pi8\bigl(2n^2\norm{\phi}_{L^2}^2+\norm{\phi'}_{L^2}^2\bigr),
\qquad
\norm{\widehat{(F_n)}_n}_{\HH}^2=\tfrac\pi8\norm{\phi''}_{L^2}^2 .
\]
By Lemma~\ref{lem:hw-nospec} we have $in\in\rho(\cA)$, so $\widehat{(U_n)}_n=(inI-\cA)^{-1}\widehat{(F_n)}_n$ and the resolvent bound follows with $c=\norm{\phi}_{L^2}/\norm{\phi''}_{L^2}$. As $\norm{(inI-\cA)^{-1}}_{\cL(\HH)}\to\infty$, $\cA$ exhibits $1$-derivative loss by Remark~\ref{rem:loss-resolvent}, and Proposition~\ref{prop:finite-stacking} with $m=1$ produces $F\in L^2_{\#}(0,2\pi;\HH)$ with the stated properties.
\end{proof}

The solutions $w_n$ are wave motions completely invisible to the interface: the heat component vanishes identically, so the dissipation never engages, and the response exceeds the forcing by one derivative. Each $w_n$ is periodic; Proposition~\ref{prop:finite-stacking} stacks the corresponding Fourier modes into a single forcing with no periodic mild solution, and through Theorem~\ref{thm:equiv} the regularity gap becomes a statement about resonance. The forcing produced here is not smooth, and by Proposition~\ref{prop:conditional-converse-loss} it cannot be at any period for which the sampled resolvent obeys a polynomial bound. A smooth resonant forcing requires infinite derivative loss, and hence a geometry outside  of this construction.

\subsection{Partially damped wave systems}\label{sec:pdamped}

The interplay between geometry, resolvent growth, and periodic resonance is visible in the locally damped wave equation
\[
u_{tt}-\Delta u+a(x)u_t=f\ \text{ in }\Om,\qquad u|_{\partial\Om}=0,
\]
with $0\le a$ supported in a damping region. This system has been extensively studied, and is written in first-order form as $\partial_t(u,u_t)=\cA_a(u,u_t)+(0,f)$, for the damped wave generator $\mathcal A_a$. The geometric control condition famously requires ``every generalized bicharacteristic to meet $\{a>0\}$ within a uniform time" \cite{BardosLebeauRauch1992}. Under the GCC, one has exponential energy decay, and hence a uniform resolvent bound on $i\R$ by the Gearhart--Pr\"uss theorem \cite{Gearhart1978,Pruss1984,Huang1985}; thence, one has unconditional PWP in the standard energy class \cite{Galdi,Bostan2002}. However, when
the GCC fails,  consequences depend on the so called trapped set. 

If $\phi_j$ are quasimodes which concentrate away from the damping,
\[
\norm{(-\Delta-k_j^2)\phi_j}_{L^2}\le\eps_j,
\qquad
\norm{a^{1/2}\phi_j}_{L^2}\le\eps_j,
\qquad \norm{\phi_j}_{L^2}=1,
\]
then the energy vectors $(\phi_j,ik_j\phi_j)$ are approximate eigenvectors of $\cA_a$ at $ik_j$, and $$\norm{(ik_jI-\cA_a)^{-1}}\gtrsim\eps_j^{-1}$$ up to polynomial factors, whenever $ik_j\in\rho(\cA_a)$. The strength of the resulting lower bound tracks the nature of the trapping: stable trapping can produce super-polynomial or exponential growth \cite{CardosoPopov}, while hyperbolic-type trapping typically produces weaker rates \cite{ChristiansonSchenckVasyWunsch2014,BurqZworski2004}, with intermediate behavior possible in between \cite{AnantharamanLeautaud2014,LeautaudLerner2017,BurqHitrik2007}.
The comparison is apt because the heat-wave system is, in effect, an interface-damped wave equation in which the dissipation is supplied by the heat domain rather than by a cutoff function $a(\mathbf x)$. Eliminating the heat component at a fixed frequency $\omega$ leads formally to a frequency-dependent {\em impedance condition} on the wave domain/boundary,
\[
\partial_{\bm n}W+i\omega\,N_H(\omega)\,W|_\Gamma=0\quad\text{on }\Gamma,
\]
where $N_H(\omega)$ is a Dirichlet-to-Neumann mapping of $i\omega-\Delta$ on $\Om_H$. Here, the wave is damped only through its trace on $\Gamma$, just as the locally damped wave is damped only where $a>0$.

Both models then share the sampling layer described in Section~\ref{sec:taxonomy}, namely, however large the resolvent is near a trapped frequency $k_j$, a $T$-periodic forcing tests only the lattice $\{i\omega_n\}$. Thus, the fixed-period response is governed by the interaction of the trapped frequencies with $\frac{2\pi}{T}\Z$. Failure of the GCC robustly destroys uniform resolvent control, whereas fixed-period resonance demands an arithmetic alignment of the two. The recent \cite{LaurentRivas2026} develops the positive side of this picture, deducing periodic well-posedness with a finite loss of derivatives from prescribed rates in the non-uniform case, and from associated resolvent bounds on $i\R$. They apply the theory to the damped wave equation as well as to heat-wave systems. They also give a counterexample to PWP (i.e., a resonant scenario), whose mechanism is instructive; they consider the Schr\"odinger equation on $\mathbb S^2$, with damping vanishing near an equator, and produce, for every $k$, a dense $G_\delta$ of forcings in $L^1_{\#}(0,2\pi;H^k(\mathbb S^2))$ whose solutions are unbounded. The choice of equation is forced by the arithmetic. On the round sphere the eigenvalues of $-\Delta_g$ are the integers $j(j+1)$, so every trapped frequency is sampled exactly by the lattice of the single period $T=2\pi$---the needed alignment is automatic. For the wave equation, the corresponding frequencies are $\sqrt{j(j+1)}$ and no such alignment is available. The next section obtains alignment for a genuinely hyperbolic system. The forcings of \cite{LaurentRivas2026} are regular in space, but their theorem does not establish temporal smoothness and concerns the single period $T=2\pi$; Theorem~\ref{thm:ellipse} gives a resonant forcing that is $C^\infty$ in time at every period of a dense $G_\delta$ set, and Proposition~\ref{prop:explicit-R} gives it explicitly for the periods of a dense $G_\delta$ subset.

\section{Resonant ellipse: infinite derivative loss and smooth resonance}\label{sec:ellipse}

We now prove Theorem~\ref{thm:ellipse}, in the general form of Theorem~\ref{thm:ellipse-general} below. Part~(i) follows from the quasimode branch of Section~\ref{sec:ellipse-quasimode} by rounding the period so that a single quasimode frequency is sampled exactly; part~(ii) follows from Theorem~\ref{thm:hfc-baire}, and Proposition~\ref{prop:explicit-R} makes a subset of the resonant periods, and the resonant forcing, explicit.

For the argument only two things are required: near its minor axis the wave domain coincides with an ellipse, and the interface stays out of that slab. We formalize this and state the general form of Theorem~\ref{thm:ellipse}. Let $a>b>0$ and let
\[
 E:=\Bigl\{(x,y)\in\R^2:\ \frac{x^2}{a^2}+\frac{y^2}{b^2}<1\Bigr\},
\]
whose minor axis is the segment $\{0\}\times(-b,b)$; for $\eps>0$ let $S_\eps:=\{(x,y)\in\R^2:\ \abs x<\eps\}$.

\begin{definition}[Ellipse configuration]\label{def:ellipse}
A pair of domains $(\Om_W,\Om_H)$ with interface $\Gamma$, admissible in the sense of Section~\ref{sec:heatwave}, is an \emph{ellipse configuration} if $\Gamma$ contains a smooth open patch and there is $\eps>0$ such that, after a rigid motion of the plane,
\[
\Om_W\cap S_\eps=E\cap S_\eps
\qquad\text{and}\qquad
\overline\Gamma\cap\overline{S_\eps}=\emptyset .
\]
\end{definition}
In particular the minor axis $\{0\}\times(-b,b)$ is a bouncing-ball billiard orbit of $\Om_W$ which never meets $\overline\Gamma$, so no ray issuing along it reaches the heat domain and the geometric conditions of Section~\ref{sec:heatwave} fail. Two instances are of particular interest: (a) the resonant ellipse of Section~\ref{sec:taxonomy}, $\Om_W=E$ and $\overline\Gamma\subset\{x\ge\delta_\Gamma\}$, for which every $\eps<\delta_\Gamma$ works; (b) the \emph{truncated ellipse} $\Om_W^{d}=E\cap\{x>-d\}$, $0<d<a$, with the chord $\Gamma^{d}=\{x=-d\}\cap E$ as interface and a heat domain attached to the left of the chord, in the notation of Section~\ref{sec:taxonomy}, for which every $\eps<d$ works. Instance (b) is the family behind the sharpness discussion after Theorem~\ref{thm:ellipse} and in Section~\ref{sec:ellipse-sharp}.

\begin{theorem}[Resonant ellipse configurations]\label{thm:ellipse-general}
Let $(\Om_W,\Om_H)$ be an ellipse configuration. Then both conclusions \textup{(i)} and \textup{(ii)} of Theorem~\ref{thm:ellipse} hold for the heat-wave system \eqref{eq:hw} on $(\Om_W,\Om_H)$.
\end{theorem}
Theorem~\ref{thm:ellipse} is instance (a). Throughout the rest of this section $(\Om_W,\Om_H)$ is an arbitrary ellipse configuration, with $E$ and $\eps$ fixed as in Definition~\ref{def:ellipse}; every construction below is supported in $S_\eps$, where $\Om_W$ and $E$ coincide, and vanishes near $\overline\Gamma$.

Two admissibility remarks. First, $\Om_W$ is connected and Lipschitz, and $\Gamma$ contains a smooth open patch, so Lemma~\ref{lem:hw-nospec} applies and $i\R\subset\rho(\cA)$. Every sampled frequency therefore lies in the resolvent set, for every period, and the hypotheses of Theorem~\ref{thm:loss-secular} are met as soon as infinite derivative loss is established. Second, $\Gamma_W$ has positive measure, so the standing assumptions of Section~\ref{sec:heatwave} hold and $S(t)$ is strongly stable \cite[Theorem 4]{ZhangZuazua2007}. It is not uniformly stable: the resolvent is unbounded on $i\R$ by Corollary~\ref{lem:ell-resolvent} below, which excludes uniform stability by the Gearhart--Pr\"uss theorem.

\subsection{The quasimode branch}\label{sec:ellipse-quasimode}
The basis of our construction is concentrating eigenmodes; the precise construction is given in the Appendix; compare also \cite{Ralston77,CardosoPopov}. Exponential $L^p$ localization of the Dirichlet eigenfunctions of an ellipse in sectors around the minor axis, for a fixed angular index as the eigenvalue tends to infinity, is proved in \cite[Theorem 3.1]{NguyenGrebenkov2013} through the large parameter asymptotics of Mathieu functions; part (iii) below with $m=0$ is, up to constants, the case of the angular ground state there. We keep our own proof, which rests on an Agmon estimate and also yields the quantization (iv) and the bounds in the higher Sobolev norms used in Corollary~\ref{cor:obstruction}. The essential properties are collected in the following proposition. Since we also want to show that no gain of \emph{space} regularity of the forcing prevents the high-frequency catastrophe (Corollary~\ref{cor:obstruction}), we record the scale of Sobolev spaces adapted to the Dirichlet Laplacian on $E$,
\[
\mathcal{H}^\ell(E):=\{\phi\in H^\ell(E):\ \Delta^m\phi\in H^1_0(E)\text{ for all integers }0\le m<\tfrac{\ell}{2}\},\enspace \ell\in\N,
\]
with $\mathcal{H}^0(E):=L^2(E)$ and $\mathcal{H}^{-\ell}(E):=(\mathcal{H}^\ell(E))^*$.
These are Hilbert spaces with inner product
\[
\langle\phi,\psi\rangle_{\mathcal{H}^\ell(E)}=\begin{cases}
\int_E \Delta^{\ell/2}\phi\; \Delta^{\ell/2}\psi\, dx&\ell\text{ even},
\\[2pt]
\int_E \nabla \Delta^{(\ell-1)/2}\phi\cdot \nabla \Delta^{(\ell-1)/2}\psi\, dx&\ell\text{ odd},
\end{cases}
\]
whose norm is equivalent to the $H^\ell(E)$-norm on $\mathcal{H}^\ell(E)$ by elliptic regularity on the smooth domain $E$; negative orders are defined by duality, $\mathcal{H}^{-\ell}(E):=(\mathcal{H}^{\ell}(E))^*$.

The following proposition collects Lemma~\ref{lem:radial-zeros}, Lemma~\ref{lem:smooth-ext} and Lemma~\ref{lem:lowerbound} of the appendix (with the constants renamed).
\begin{proposition}[Concentrated Dirichlet eigenmodes]\label{prop:ell-branch}
There are $N\in\N$ and constants $c_0,C_1,C_2>0$, and for every $\eps'\in(0,a\sin(\pi/8))$ and every  $m\in \mathbb{N}_0$ constants $C_{m,\eps'}>0$ and $c_m>0$ (the latter independent of $\eps'$), all depending only on $a>b>0$, such that for every $j\ge N$ there is an eigenpair $(U_j,\kappa_j^2)$ of the Dirichlet Laplacian on $E$ with the following properties:
\begin{enumerate}[label=\textup{(\roman*)}]
 \item $U_j\in C^\infty(\overline E)\cap H^1_0(E)$ is real-valued, even in $x$ and in $y$, and $-\Delta U_j=\kappa_j^2U_j$ in $E$;
 \item $c_0\le\norm{U_j}_{L^2(E)}\le C_1$; consequently $\norm{U_j}_{\mathcal{H}^\ell(E)}=\kappa_j^\ell\norm{U_j}_{L^2(E)}\ge c_0\kappa_j^\ell$ for every $\ell\in\Z$;
 \item $\norm{U_j}_{H^m(E\setminus S_{\eps'})}\le C_{m,\eps'}\,\kappa_j^{2m}\,e^{-c_m\eps'^2\kappa_j}$;
 \item $j-C_2\le \kappa_j\,b/\pi \le j+C_2$; in particular $\kappa_j\to\infty$, and the $\kappa_j$ are strictly increasing.
\end{enumerate}
\end{proposition}
\begin{figure}[t]
  \centering
  \includegraphics[width=0.48\textwidth]{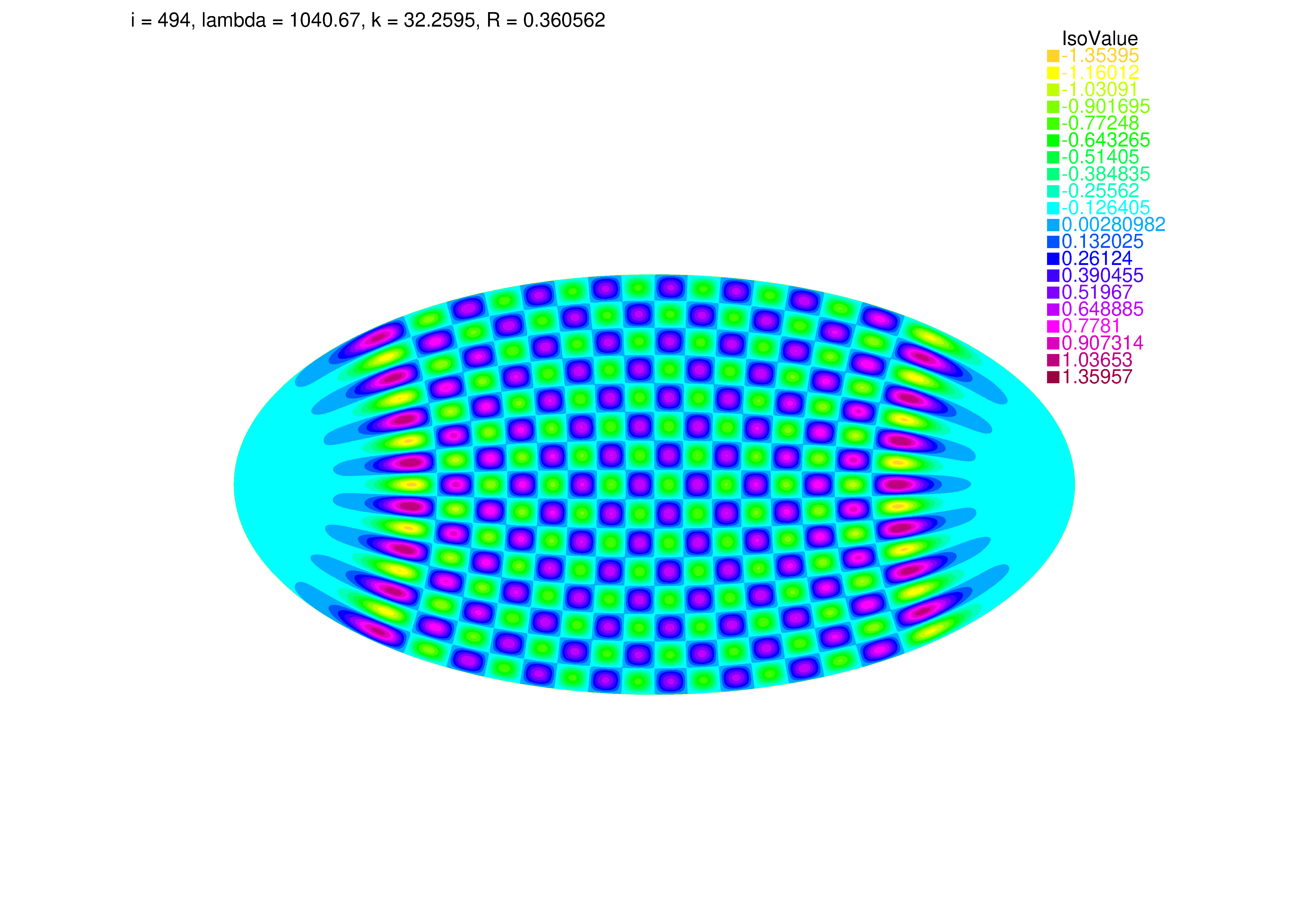}\hfill
  \includegraphics[width=0.48\textwidth]{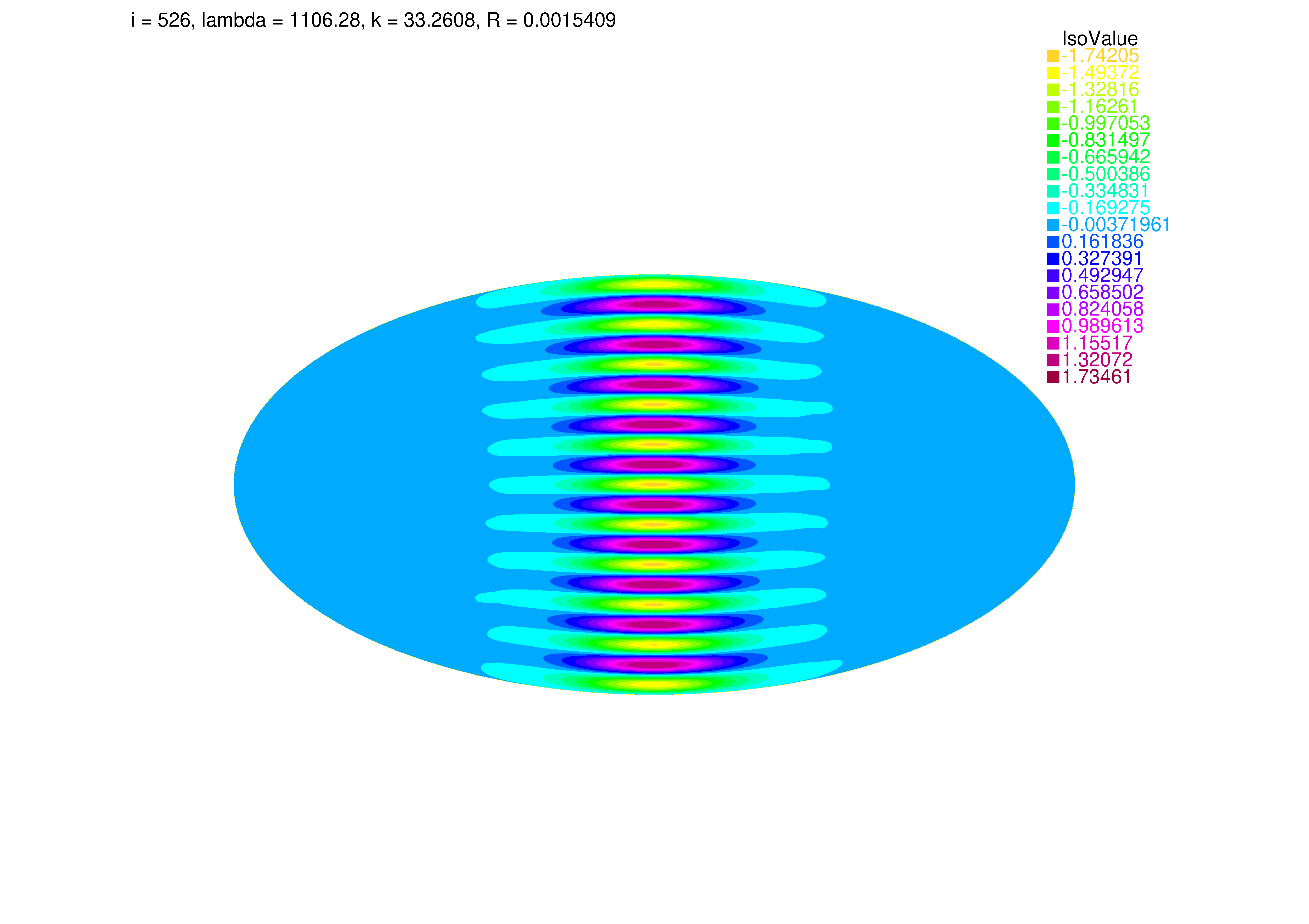}\\[2ex]
  \includegraphics[width=0.48\textwidth]{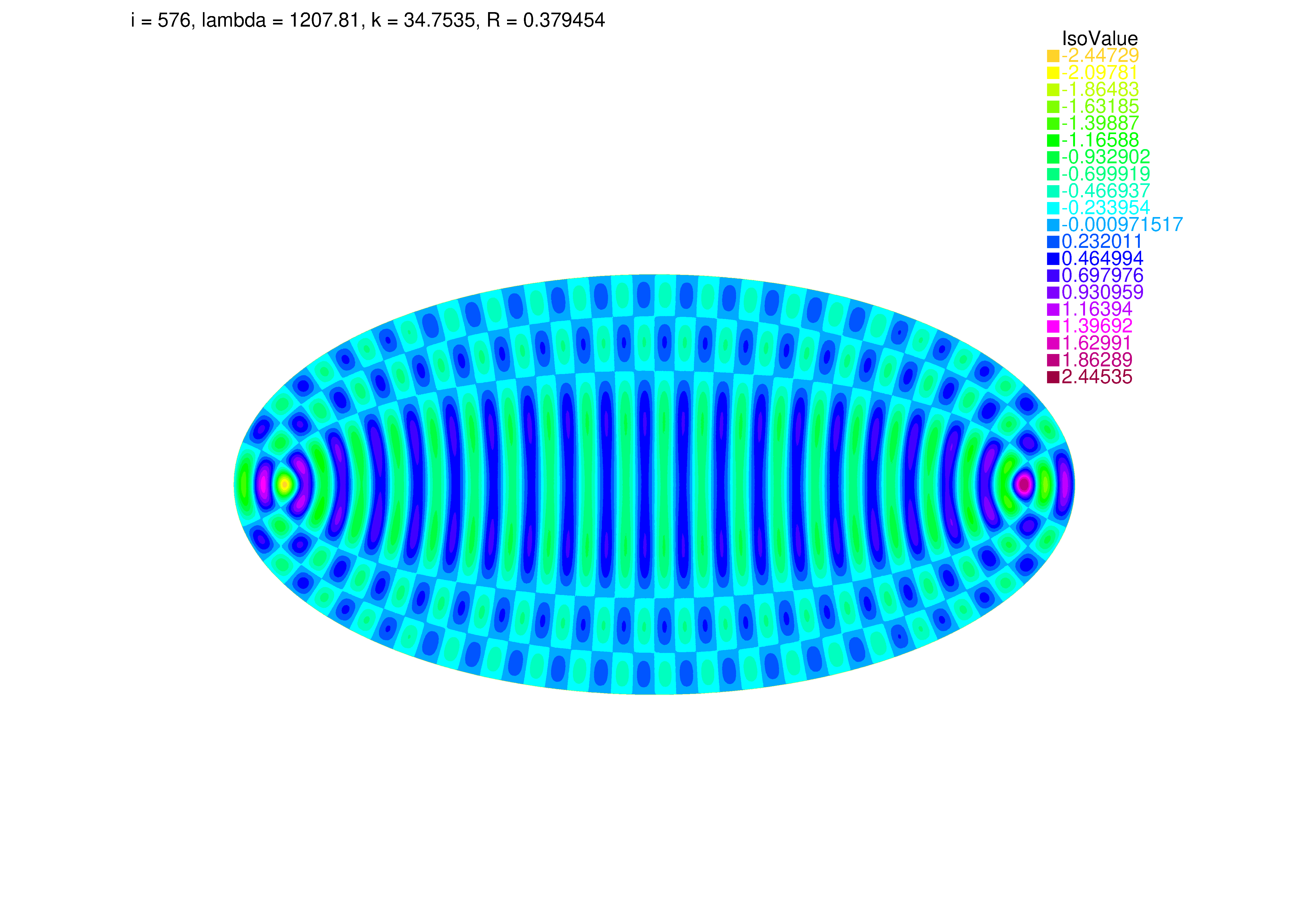}\hfill
  \includegraphics[width=0.48\textwidth]{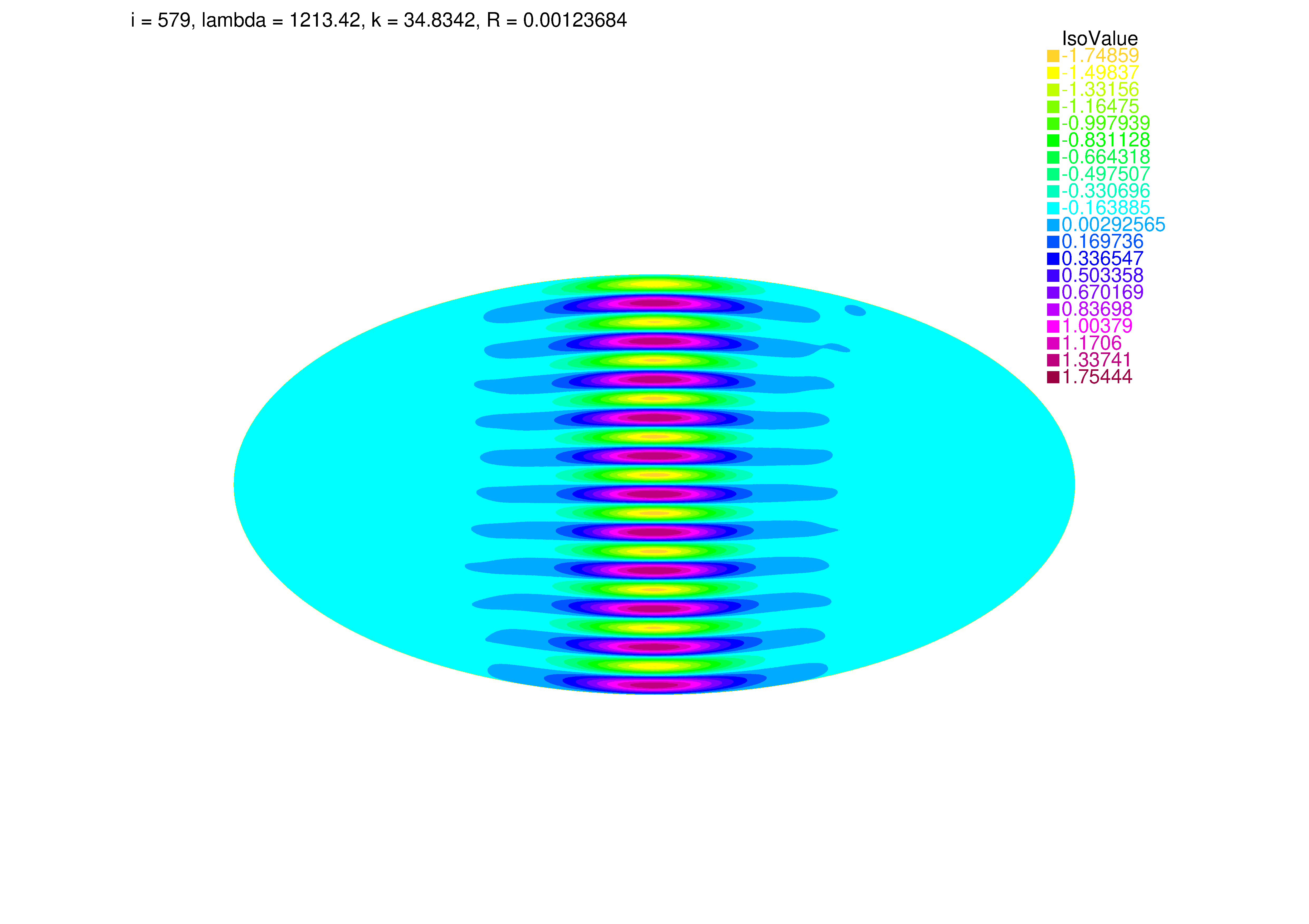}
  \caption{Numerically computed Dirichlet eigenfunctions of the ellipse $E$ with $a=2$, $b=1$ (the header of each panel gives the index of the eigenvalue, the eigenvalue $\lambda$ and $k=\sqrt\lambda$). Left column: two non-localized modes (indices $494$ and $576$). Right column: two modes localized along the minor axis (indices $526$ and $579$), as in Proposition~\ref{prop:ell-branch}(iii); the mass outside a thin slab around the minor axis is negligible, which is what makes them invisible to an interface placed outside the slab.}
  \label{fig:panel}
\end{figure}
These eigenmodes allow the construction of quasimodes with exponentially large resolvent. Fix $\eps_0>0$ with
\begin{equation}\label{eq:eps0}
\tfrac32\eps_0<\eps\qquad\text{and}\qquad \eps_0<a\sin(\pi/8),
\end{equation}
and an even cutoff $\phi\in C^\infty(\R)$ with $0\le\phi\le1$, $\phi(x)=1$ for $\abs x\le\eps_0$ and $\phi(x)=0$ for $\abs x\ge\tfrac32\eps_0$. For a function $g$ on $E$ we write $\phi g$ for the function $(x,y)\mapsto\phi(x)g(x,y)$, extended by zero to $\Om_W$; if $g$ vanishes on $\partial E$, this extension vanishes on $\partial\Om_W$ and in a neighborhood of $\overline\Gamma$, because $\supp\phi g\subset\overline{S_{3\eps_0/2}}$, $\Om_W\cap S_\eps=E\cap S_\eps$ and $\overline\Gamma\cap\overline{S_\eps}=\emptyset$. The cutoff is two-sided, so that the construction applies to every ellipse configuration, whichever side of the slab carries the interface.
\begin{corollary}[Energy-space quasimodes and resolvent growth]\label{lem:ell-resolvent}
Set $V_j:=(\phi U_j,\ i\kappa_j\phi U_j,\ 0)$, $j\ge N$. Then $V_j\in\dom(\cA)$,
\[
(\cA-i\kappa_jI)V_j=\bigl(0,\ \phi''U_j+2\phi'\partial_x U_j,\ 0\bigr),
\qquad
\tfrac12c_0\kappa_j\le\norm{V_j}_{\HH}\le C\kappa_j
\]
for all $j$ large. Consequently, with $\widehat V_j:=V_j/\norm{V_j}_{\HH}$ and
\begin{equation}\label{eq:ell-eps}
\eps_j:=\norm{(\cA-i\kappa_jI)\widehat V_j}_{\HH},
\end{equation}
there are $C_{\eps_0}>0$ and $\beta>0$, the latter independent of $\eps_0$, such that $0<\eps_j\le C_{\eps_0}e^{-\beta\eps_0^2\kappa_j}$ for all $j\ge N$; in particular $\eps_j\to0$, $\sum_j\eps_j<\infty$, and
\[
\norm{(i\kappa_jI-\cA)^{-1}}_{\cL(\HH)}\ \ge\ \frac1{\eps_j}\ \ge\ C_{\eps_0}^{-1}\,e^{\beta\eps_0^2\kappa_j}.
\]
\end{corollary}

\begin{remark}[Pseudospectral, not spectral]\label{rem:ell-pseudo}
Because $\cA$ is not self-adjoint, an exponentially accurate quasimode does \emph{not} place a spectral point, let alone an eigenvalue, exponentially close to $i\kappa_j$: for a non-normal generator the resolvent norm can be large far from $\sigma(\cA)$. In fact $\sigma(\cA)\cap i\R=\emptyset$. Since the frequencies $\kappa_j\to\infty$ do not accumulate at any finite point, the $\widehat V_j$ do not form a Weyl sequence for a fixed imaginary $\lambda$. The resonance produced below is thus {\bf non-spectral}.
\end{remark}

\begin{proof}[Proof of Corollary~\ref{lem:ell-resolvent}]
By construction $\phi U_j\in C^\infty(\overline{\Om_W})$ vanishes on $\partial\Om_W$ and near $\overline\Gamma$, so $V_j\in\dom(\cA)$ with $\cA V_j=(i\kappa_j\phi U_j,\Delta(\phi U_j),0)$, and $\Delta(\phi U_j)=\phi\Delta U_j+2\phi'\partial_xU_j+\phi''U_j=-\kappa_j^2\phi U_j+2\phi'\partial_xU_j+\phi''U_j$ gives the formula for $(\cA-i\kappa_jI)V_j$. Since $\norm{\nabla U_j}_{L^2(E)}=\kappa_j\norm{U_j}_{L^2(E)}$, Proposition~\ref{prop:ell-branch}(ii) gives $\norm{V_j}_{\HH}\le C\kappa_j$, while $\norm{V_j}_{\HH}\ge\kappa_j\norm{\phi U_j}_{L^2}\ge\kappa_j\bigl(\norm{U_j}_{L^2(E)}-\norm{U_j}_{L^2(E\setminus S_{\eps_0})}\bigr)\ge\tfrac12c_0\kappa_j$ for large $j$ by (ii) and (iii). The residual is supported in the strip $\{\eps_0\le\abs x\le\tfrac32\eps_0\}$, hence by (iii) with $m=1$
\[
\norm{(\cA-i\kappa_jI)V_j}_{\HH}\le\bigl(\norm{\phi''}_\infty+2\norm{\phi'}_\infty\bigr)\norm{U_j}_{H^1(E\setminus S_{\eps_0})}\le C_{\eps_0}\kappa_j^2e^{-c_1\eps_0^2\kappa_j},
\]
and dividing by $\norm{V_j}_{\HH}$ yields the bound on $\eps_j$ with $\beta:=c_1/2$ (the factor $\kappa_j$ being absorbed in $C_{\eps_0}$). Finally $\eps_j>0$ because $i\kappa_j\in\rho(\cA)$, and $1=\norm{(i\kappa_jI-\cA)^{-1}(i\kappa_jI-\cA)\widehat V_j}_{\HH}\le\norm{(i\kappa_jI-\cA)^{-1}}\,\eps_j$.
\end{proof}
\begin{remark}[Optimality of logarithmic decay]\label{rem:logdecay}
Let $\Om_H$ be a smooth simply connected domain with $\overline{\Om_H}\subset E\setminus\overline{S_\eps}$ and $\Om_W:=E\setminus\overline{\Om_H}$. This is an ellipse configuration whose interface does not meet $\partial E$, which is the setting of \cite{Fathallah2011}. There it is the case that $\norm{(\cA-i\mu I)^{-1}}_{\cL(\HH)}\le Ce^{C\abs\mu}$ for all $\mu\in\R$ \cite[Theorem 1.1]{Fathallah2011}, and consequently the energy of solutions with data in $\dom(\cA)$ decays logarithmically, $\norm{S(t)\cA^{-1}}_{\cL(\HH)}\le C/\log(2+t)$ \cite[Theorem 1.2]{Fathallah2011}. Our quasimode construction, Corollary~\ref{lem:ell-resolvent} together with Proposition~\ref{prop:ell-branch}(iv), provides the matching lower bounds. This translates into decay rates by following the proof of \cite[Proposition 1.3]{BattyDuyckaerts2008}, which converts a large resolvent at the frequency $\kappa_j$ into slow decay up to times of order $e^{\beta\eps_0^2\kappa_j}$. Thus
\[
\frac{c}{\log t}\le\norm{S(t)\cA^{-1}}_{\cL(\HH)}\le\frac{C}{\log t},\qquad t\ge t_0 ,
\]
and neither bound of \cite{Fathallah2011} can be improved beyond constants. For the transmission condition $w=u$ in place of $w_t=u$, the optimality of logarithmic decay is noted in \cite[Remark 1.5]{Duyckaerts2007} for a ball with an interior heat domain, through whispering gallery modes.
\end{remark}

\subsection{Failure of periodic well-posedness: the high-frequency catastrophe}
The following theorem provides, for every target period $T$, time-periodic solutions of the heat-wave system with vanishing heat component, supported in the slab $S_{3\eps_0/2}$, whose response is bounded from below while the forcing is exponentially small. Together with Corollary~\ref{cor:obstruction} it proves part~(i) of Theorems~\ref{thm:ellipse} and~\ref{thm:ellipse-general}.

\begin{theorem}[Witnessing pairs]\label{thm:obstruction}
Let $(\Om_W,\Om_H)$ be an ellipse configuration, let $\eps_0$ and $\phi$ be as in \eqref{eq:eps0}, and let $T>0$. For $n\ge N$ set $N_n:=\lfloor T\kappa_n/2\pi\rfloor$ and $T_n:=2\pi N_n/\kappa_n$, so that
\[
T-\frac{2\pi}{\kappa_n}<T_n\le T,
\]
and define
\begin{align*}
w_n(t,x,y)&:=\sin(\kappa_nt)\,\phi(x)U_n(x,y),\\
f_n&:=\partial_t^2w_n-\Delta w_n=-\sin(\kappa_nt)\bigl(\phi''U_n+2\phi'\partial_xU_n\bigr),
\end{align*}
both extended by zero to $\Om_W$. Then $W_n:=(w_n,\partial_tw_n,0)$ is a classical solution of $\dot W_n=\cA W_n+F_n$, $F_n:=(0,f_n,0)$, with values in $\dom(\cA)$; $w_n$ and $f_n$ vanish outside $S_{3\eps_0/2}$ and are $2\pi/\kappa_n$-periodic, hence $T_n$-periodic; and for every integer $m\ge0$ and all $n$ large,
\begin{equation}\label{eq:witness-bounds}
\begin{aligned}
\norm{\partial_tw_n}_{L^2((0,T_n)\times\Om_W)}\ &\ge\ c\,\kappa_n,\qquad
\norm{w_n}_{L^2((0,T_n)\times\Om_W)}\ \ge\ c,\\
\norm{f_n}_{H^{m}(0,T_n;L^2(\Om_W))}\ &\le\ C_{m,\eps_0}\,\kappa_n^{m+2}\,e^{-c_1\eps_0^2\kappa_n},
\end{aligned}
\end{equation}
with $c>0$ depending only on $T$, $a$, $b$, and $C_{m,\eps_0}$ depending on $m$, $\eps_0$, $T$, $a$, $b$. 
\end{theorem}

\begin{proof}
By Proposition~\ref{prop:ell-branch}(i), $u_n:=\sin(\kappa_nt)U_n$ solves $\partial_t^2u_n-\Delta u_n=0$ in $E$, and $w_n=\phi u_n$ satisfies $\partial_t^2w_n-\Delta w_n=-\phi''u_n-2\phi'\partial_xu_n=f_n$. As in Corollary~\ref{lem:ell-resolvent}, $\phi U_n\in C^\infty(\overline{\Om_W})$ vanishes on $\partial\Om_W$ and near $\overline\Gamma$, so $W_n(t)\in\dom(\cA)$ for every $t$, and $\dot W_n=\cA W_n+F_n$ is the first-order form of the wave equation for $w_n$ with vanishing heat component; the interface conditions hold trivially because $w_n\equiv0$ near $\Gamma$. Periodicity is clear since $\kappa_nT_n\in2\pi\Z$, with $N_n\ge1$ for $n$ large, and $T_n\to T$ since $\kappa_n\to\infty$.

For the lower bounds, $\int_0^{T_n}\sin^2(\kappa_nt)\,dt=\int_0^{T_n}\cos^2(\kappa_nt)\,dt=T_n/2\ge T/4$ for $n$ large (as $T_n$ is a multiple of $\pi/\kappa_n$ and $T_n\ge T-2\pi/\kappa_n$), so
\begin{align*}
\norm{w_n}^2_{L^2((0,T_n)\times\Om_W)}&=\frac{T_n}{2}\norm{\phi U_n}_{L^2(E)}^2,\\
\norm{\partial_tw_n}^2_{L^2((0,T_n)\times\Om_W)}&=\frac{T_n}{2}\kappa_n^2\norm{\phi U_n}_{L^2(E)}^2,
\end{align*}
and $\norm{\phi U_n}_{L^2(E)}\ge\norm{U_n}_{L^2(E)}-\norm{U_n}_{L^2(E\setminus S_{\eps_0})}\ge c_0/2$ for $n$ large by Proposition~\ref{prop:ell-branch}(ii),(iii). For the forcing, $\partial_t^kf_n=\pm\kappa_n^k\sigma_k(\kappa_nt)\bigl(\phi''U_n+2\phi'\partial_xU_n\bigr)$ with $\sigma_k\in\{\sin,\cos\}$, and the spatial factor is supported in $\{\eps_0\le\abs x\le\tfrac32\eps_0\}$, so Proposition~\ref{prop:ell-branch}(iii) with $m=1$ gives
\begin{align*}
\norm{\partial_t^kf_n}_{L^2((0,T_n)\times\Om_W)}&\le\kappa_n^k\sqrt{T_n}\bigl(\norm{\phi''}_\infty+2\norm{\phi'}_\infty\bigr)\norm{U_n}_{H^1(E\setminus S_{\eps_0})}\\
&\le C_{\eps_0}\sqrt T\,\kappa_n^{k+2}e^{-c_1\eps_0^2\kappa_n};
\end{align*}
summing over $0\le k\le m$ proves the last bound in \eqref{eq:witness-bounds}.
\end{proof}

\begin{proof}[Proof of Theorem~\textup{\ref{thm:ellipse-general}(i)} \textup{(}and of Theorem~\textup{\ref{thm:ellipse}(i)}\textup{)}]
Fix $T_*>0$ and an integer $m\ge1$, and apply Theorem~\ref{thm:obstruction} with $T=T_*$. The pairs $(F_n,W_n)$ are admissible for the periods $T_n\to T_*$, smooth in time with values in $\dom(\cA)$, and
\begin{align*}
\norm{W_n}_{L^2(0,T_n;\HH)}&\ge\norm{\partial_tw_n}_{L^2((0,T_n)\times\Om_W)}\ge c\kappa_n\ge c,\\
\norm{F_n}_{H^{m-1}_{\#}(0,T_n;\HH)}&=\norm{f_n}_{H^{m-1}(0,T_n;L^2(\Om_W))}\le C_{m-1,\eps_0}\kappa_n^{m+1}e^{-c_1\eps_0^2\kappa_n}\to0 .
\end{align*}
Hence $(F_n,W_n)$ are witnessing pairs in the sense of Definition~\ref{def:period-accum}, and $\cA$ exhibits a high-frequency catastrophe at $T_*$. The forcing acts only in $\Om_W$, at positive distance from $\Gamma$, and the heat component vanishes identically.
\end{proof}

The theorem shows that the PWP estimate is destroyed on every interval of periods, but the period $T_n$ at which the $n$-th witness lives moves with $n$; nothing thus far pins the obstruction to a single period. Before addressing this, we note that no gain of space regularity of the forcing helps either.

\begin{corollary}[Higher order]\label{cor:obstruction}
In the setting of Theorem~\ref{thm:obstruction}, for all integers $\ell,j,m,i\ge0$ and all $n$ large,
\begin{align*}
\norm{w_n}_{H^\ell(0,T_n;H^j(\Om_W))}&\ge c\,\kappa_n^{\ell+j}\qquad\text{and}\\
\norm{f_n}_{H^m(0,T_n;H^i(\Om_W))}&\le C_{m,i,\eps_0}\,\kappa_n^{m+2i+2}\,e^{-c_{i+1}\eps_0^2\kappa_n},
\end{align*}
with $c>0$ depending only on $T,a,b,\ell,j$.
\end{corollary}

\begin{proof}
Since $w_n=\sin(\kappa_nt)U_n$ on $E_0:=E\cap S_{\eps_0}$, differentiating $\ell$ times in $t$ and using $\int_0^{T_n}\sin^2(\kappa_nt)\,dt=\int_0^{T_n}\cos^2(\kappa_nt)\,dt=T_n/2$ gives
\[
\norm{w_n}_{H^\ell(0,T_n;H^j(\Om_W))}\ge\kappa_n^{\ell}\sqrt{T_n/2}\,\norm{U_n}_{H^j(E_0)} .
\] If $j=2k$, the pointwise bound $\abs{\Delta^kU_n}\le2^k\abs{\nabla^{2k}U_n}$ gives $\norm{U_n}_{H^j(E_0)}\ge2^{-k}\norm{\Delta^kU_n}_{L^2(E_0)}=2^{-k}\kappa_n^{j}\norm{U_n}_{L^2(E_0)}$; if $j=2k+1$, similarly $\norm{U_n}_{H^j(E_0)}\ge2^{-k}\norm{\nabla\Delta^kU_n}_{L^2(E_0)}=2^{-k}\kappa_n^{j-1}\norm{\nabla U_n}_{L^2(E_0)}$, and
\begin{align*}
\norm{\nabla U_n}_{L^2(E_0)}&\ge\norm{\nabla U_n}_{L^2(E)}-\norm{U_n}_{H^1(E\setminus S_{\eps_0})}\\
&=\kappa_n\norm{U_n}_{L^2(E)}-\norm{U_n}_{H^1(E\setminus S_{\eps_0})}.
\end{align*}
In both cases Proposition~\ref{prop:ell-branch}(ii),(iii) give $\norm{U_n}_{H^j(E_0)}\ge c\kappa_n^j$ for $n$ large. For the forcing, $f_n$ involves $U_n$ and $\partial_xU_n$ on the strip $\{\eps_0\le\abs x\le\tfrac32\eps_0\}\subset E\setminus S_{\eps_0}$, so $\norm{f_n(t)}_{H^i(\Om_W)}\le C_{i,\eps_0}\norm{U_n}_{H^{i+1}(E\setminus S_{\eps_0})}$, and Proposition~\ref{prop:ell-branch}(iii) with $m=i+1$, combined with the time derivatives as in Theorem~\ref{thm:obstruction}, gives the bound.
\end{proof}

\subsection{Proof of infinite derivative loss and smooth resonance}\label{sec:ellipse-loss}

\begin{proof}[Proof of Theorem~\textup{\ref{thm:ellipse-general}(ii)} \textup{(}and of Theorem~\textup{\ref{thm:ellipse}(ii)}\textup{)}]
By Lemma~\ref{lem:hw-nospec}, $i\R\subset\rho(\cA)$, and by part~(i), proved above, $\cA$ exhibits a high-frequency catastrophe at every period; alternatively, hypothesis (iii) of Theorem~\ref{thm:hfc-baire} is immediate from Corollary~\ref{lem:ell-resolvent} and $\kappa_j\to\infty$. Theorem~\ref{thm:hfc-baire} therefore shows that $\mathcal R_\infty$ is a dense $G_\delta$ subset of $(0,\infty)$. For every $T\in\mathcal R_\infty$, Theorem~\ref{thm:loss-secular} provides $f\in C^\infty_{\#}(0,T;\HH)$ such that $(T,f)$ is statically, and therefore secularly, resonant.
\end{proof}

The abstract argument shows that the periods with infinite derivative loss form a dense $G_\delta$ set, but it identifies neither these periods nor the lattice frequencies responsible for the loss. An explicit subset is obtained when a single lattice $\frac{2\pi}{T}\Z$ comes exponentially close to infinitely many quasimode frequencies. The first step is that the resolvent bound of Corollary~\ref{lem:ell-resolvent} survives on a band around $\kappa_j$, not merely at $\kappa_j$.

\begin{lemma}[Spreading]\label{lem:spread}
For every $\omega\in\R$,
\[
\norm{(i\omega I-\cA)^{-1}}_{\cL(\HH)}\ \ge\ \frac1{\eps_j+\abs{\omega-\kappa_j}},
\qquad j\ge N,
\]
with $\eps_j$ as in \eqref{eq:ell-eps}.
\end{lemma}

\begin{proof}
Recall that $i\omega\in\rho(\cA)$. Since $(\cA-i\omega I)\widehat V_j=(\cA-i\kappa_jI)\widehat V_j+i(\kappa_j-\omega)\widehat V_j$ and $\norm{\widehat V_j}_{\HH}=1$, the triangle inequality gives $\norm{(\cA-i\omega I)\widehat V_j}_{\HH}\le\eps_j+\abs{\omega-\kappa_j}$, and
$$1=\norm{(i\omega I-\cA)^{-1}(i\omega I-\cA)\widehat V_j}_{\HH}\le\norm{(i\omega I-\cA)^{-1}}(\eps_j+\abs{\omega-\kappa_j}).$$
\end{proof}

The band has exponentially small width $\eps_j$, and the question is: Which periods have a lattice entering infinitely many of these bands? The answer is that they form a topologically generic, measure-theoretically negligible set.

\begin{proposition}[Explicit resonant periods]\label{prop:explicit-R}
Let $(\Om_W,\Om_H)$ be an ellipse configuration, let $\eps_0$ and $\phi$ be as in \eqref{eq:eps0}, and let $\kappa_j$, $\eps_j$ be as in Corollary~\ref{lem:ell-resolvent}. The set
\[
\mathcal R:=\Bigl\{T>0:\ \dist\bigl(\kappa_j,\tfrac{2\pi}{T}\Z\bigr)<\eps_j\ \text{for infinitely many }j\Bigr\}
\]
is a dense $G_\delta$ subset of $\mathcal R_\infty$ of Lebesgue measure zero. Let $T\in\mathcal R$, and choose an infinite set ${\mathcal J}\subset\{j\in\N:\ j\ge N\}$ and pairwise distinct integers $n_j$, $j\in {\mathcal J}$, with $\abs{\omega_{n_j}-\kappa_j}<\eps_j$. Then $\norm{(i\omega_{n_j}I-\cA)^{-1}}_{\cL(\HH)}\ge(2\eps_j)^{-1}$ for $j\in {\mathcal J}$, and
\[
f(t):=\sum_{j\in {\mathcal J}}e^{i\omega_{n_j}t}\,\frac{(i\omega_{n_j}I-\cA)\widetilde V_j}{\norm{\widetilde V_j}_{\HH}},\qquad \widetilde V_j:=(\phi U_j,\ i\omega_{n_j}\phi U_j,\ 0),
\]
is of the form $f=(0,f_W,0)$, with $f_W\in C^\infty(\R\times\overline{\Om_W})$ $T$-periodic and supported in $\{\abs x\le\tfrac32\eps_0\}$, and $(T,f)$ is statically, hence secularly, resonant.
\end{proposition}

\begin{proof}
Define
\[
\mathcal{U}_j:=\Bigl\{T>0:\ \dist\bigl(\kappa_j,\tfrac{2\pi}{T}\Z\bigr)<\eps_j\Bigr\},
\qquad
\mathcal R:=\limsup_{j\to\infty}\mathcal{U}_j=\bigcap_{J\ge N}\bigcup_{j\ge J}\mathcal{U}_j .
\]
Each $\mathcal{U}_j$ is open, being the preimage of $(-\infty,\eps_j)$ under the continuous map $T\mapsto\dist(\kappa_j,\frac{2\pi}{T}\Z)$.

\smallskip\noindent\emph{Step 1: $\mathcal R$ is a dense $G_\delta$ set.}
Fix $J\ge N$ and a nonempty open interval $(p,q)\subset(0,\infty)$. Choose $j\ge J$ with $\kappa_j>2\pi/(q-p)$. The points $2\pi n/\kappa_j$, $n\in\N$, form an arithmetic progression of gap $2\pi/\kappa_j<q-p$ increasing to $\infty$, so some $T^\ast=2\pi n/\kappa_j$ lies in $(p,q)$. At that period $\kappa_j=\frac{2\pi n}{T^\ast}\in\frac{2\pi}{T^\ast}\Z$, so $\dist(\kappa_j,\frac{2\pi}{T^\ast}\Z)=0<\eps_j$ and $T^\ast\in \mathcal{U}_j\cap(p,q)$. As $(p,q)$ was arbitrary  it follows that in every neighborhood of any point there can be found an element of $\bigcup_{j\ge J}\mathcal{U}_j$. This implies that $\bigcup_{j\ge J}\mathcal{U}_j$ is dense and open for every $J$, and so $\mathcal R$ is a dense $G_\delta$ by the Baire category theorem.

\smallskip\noindent\emph{Step 2: $\abs{\mathcal R}=0$.}
Fix $I=[p,q]\subset(0,\infty)$ and $j$ so large that $\eps_j\le\kappa_j/2$. If $T\in\mathcal{U}_j\cap I$, there is $n\in\N$ with $\abs{2\pi n/T-\kappa_j}<\eps_j$, so $2\pi n/T\in(\kappa_j/2,2\kappa_j)$ and $n<q\kappa_j/\pi$; for each such $n$ the set of admissible $T$ is the interval
\[
\Bigl(\tfrac{2\pi n}{\kappa_j+\eps_j},\ \tfrac{2\pi n}{\kappa_j-\eps_j}\Bigr),
\qquad\text{of length }\ \frac{4\pi n\,\eps_j}{\kappa_j^2-\eps_j^2}\le\frac{16 q\,\eps_j}{3\kappa_j}.
\]
Summing over the at most $q\kappa_j/\pi$ values of $n$ gives $\abs{\mathcal{U}_j\cap I}\le C_q\eps_j$. Since $\sum_j\eps_j<\infty$ by Corollary~\ref{lem:ell-resolvent}, the Borel--Cantelli lemma yields $\abs{\mathcal R\cap I}=0$ for every such $I$, hence $\abs{\mathcal R}=0$.

\smallskip\noindent\emph{Step 3: infinite derivative loss at each $T\in\mathcal R$.}
Let $T\in\mathcal R$, so $T\in \mathcal{U}_j$ for infinitely many $j$; for each such $j$ pick $n_j\in\Z$ with $\abs{\omega_{n_j}-\kappa_j}<\eps_j$, where $\omega_{n_j}=2\pi n_j/T$. Lemma~\ref{lem:spread} and Corollary~\ref{lem:ell-resolvent} give
\[
\norm{(i\omega_{n_j}I-\cA)^{-1}}\ \ge\ \frac1{\eps_j+\abs{\omega_{n_j}-\kappa_j}}\ \ge\ \frac1{2\eps_j}\ \ge\ \frac1{2C_{\eps_0}}\,e^{\beta\eps_0^2\kappa_j}.
\]
Since $\abs{\omega_{n_j}-\kappa_j}<\eps_j\le1$ for large $j$ and $\kappa_j\to\infty$, we have $\omega_{n_j}\to\infty$ and $e^{\beta\eps_0^2\kappa_j}\ge e^{-\beta\eps_0^2}\,e^{\beta\eps_0^2\omega_{n_j}}$, which dominates every power of $\abs{\omega_{n_j}}$. Hence
$\limsup_{\abs n\to\infty}(1+\abs{\omega_n})^{-m}\norm{(i\omega_nI-\cA)^{-1}}=\infty$ for every $m$, which by Remark~\ref{rem:loss-resolvent} is infinite derivative loss at the period $T$; the sampled resolvent grows exponentially along the subsequence $(\omega_{n_j})$.

\smallskip\noindent\emph{Step 4: explicit smooth resonant forcing.}
Let $T\in\mathcal R$. Each $n\in\Z$ satisfies $\abs{\omega_n-\kappa_j}<\eps_j\le1$ for at most finitely many $j$, because $\kappa_j\to\infty$; hence ${\mathcal J}$ and the pairwise distinct $n_j$ of the statement exist, and we may assume that every $j\in {\mathcal J}$ satisfies $\eps_j\le1$, $\kappa_j\ge2$ and $\norm{\phi U_j}_{L^2(E)}\ge c_0/2$ (see the proof of Corollary~\ref{lem:ell-resolvent}). Fix $j\in {\mathcal J}$ and write $\omega:=\omega_{n_j}$. As in Corollary~\ref{lem:ell-resolvent}, $\widetilde V_j\in\dom(\cA)$ and $\cA\widetilde V_j=(i\omega\phi U_j,\Delta(\phi U_j),0)$, so that
\[
(i\omega I-\cA)\widetilde V_j=(0,g_j,0),\qquad g_j:=(\kappa_j^2-\omega^2)\phi U_j-\phi''U_j-2\phi'\partial_xU_j .
\]
Moreover $\norm{\widetilde V_j}_{\HH}\ge\abs\omega\norm{\phi U_j}_{L^2(E)}\ge c_0\kappa_j/4$, since $\abs\omega\ge\kappa_j-1\ge\kappa_j/2$. Let $s\ge0$ be an integer. Since $\abs{\kappa_j^2-\omega^2}\le\eps_j(2\kappa_j+1)\le3\kappa_j\eps_j$, the equivalence of the $H^s(E)$- and $\mathcal H^s(E)$-norms and Proposition~\ref{prop:ell-branch}(ii) give $\norm{(\kappa_j^2-\omega^2)\phi U_j}_{H^s(\Om_W)}\le C_{s,\eps_0}\kappa_j^{s+1}\eps_j$. The remaining terms of $g_j$ are supported in $\{\eps_0\le\abs x\le\tfrac32\eps_0\}\subset E\setminus S_{\eps_0}$, so Proposition~\ref{prop:ell-branch}(iii) with $m=s+1$ bounds their $H^s(\Om_W)$-norm by $C_{s,\eps_0}\kappa_j^{2s+2}e^{-c_{s+1}\eps_0^2\kappa_j}$. Together with $\eps_j\le C_{\eps_0}e^{-\beta\eps_0^2\kappa_j}$ from Corollary~\ref{lem:ell-resolvent}, the function $h_j:=g_j/\norm{\widetilde V_j}_{\HH}$ satisfies
\[
\norm{h_j}_{H^s(\Om_W)}\le C_{s,\eps_0}\,\kappa_j^{2s+1}e^{-\gamma_s\eps_0^2\kappa_j},\qquad \gamma_s:=\min\{\beta,c_{s+1}\}.
\]
Since $\kappa_j\ge cj$ by Proposition~\ref{prop:ell-branch}(iv) and $\abs{\omega_{n_j}}\le\kappa_j+1$, we have $\sum_{j\in {\mathcal J}}(1+\abs{\omega_{n_j}})^k\norm{h_j}_{H^s(\Om_W)}<\infty$ for all integers $k,s\ge0$. Hence the series defining $f$ converges absolutely in $C^k([0,T];H^s(\Om_W))$ for all $k,s$, and by the Sobolev embedding $f=(0,f_W,0)$ with $f_W\in C^\infty(\R\times\overline{\Om_W})$; $f_W$ is $T$-periodic since $\omega_{n_j}\in\frac{2\pi}{T}\Z$, and supported in $\{\abs x\le\tfrac32\eps_0\}$ together with $\phi$. The Fourier coefficients of $f$ are $\hat f_{n_j}=(0,h_j,0)$ for $j\in {\mathcal J}$, all others vanishing. If a $T$-periodic mild solution $u$ existed, Lemma~\ref{lem:fourier} and $i\omega_{n_j}\in\rho(\cA)$ would give $\hat u_{n_j}=(i\omega_{n_j}I-\cA)^{-1}(0,h_j,0)=\widetilde V_j/\norm{\widetilde V_j}_{\HH}$, of norm $1$ for every $j\in {\mathcal J}$, contradicting Parseval's identity. Theorem~\ref{thm:equiv} gives the remaining conclusions. The same bounds show that the infinite derivative loss of Step~3 is already witnessed by wave forcings: the pairs $\bigl(e^{i\omega_{n_j}t}(0,h_j,0),\,e^{i\omega_{n_j}t}\widetilde V_j/\norm{\widetilde V_j}_{\HH}\bigr)$, $j\in\mathcal J$, are admissible for the period $T$, the responses have norm $\sqrt T$ in $L^2(0,T;\HH)$, and the forcings tend to zero in $H^{m}_{\#}(0,T;\HH)$ for every $m$, by the bound on $h_j$ with $s=0$.

\end{proof}

\begin{remark}\label{rem:ell-sharpness}
Both the genericness and the measure nullity of $\mathcal R$ in Proposition~\ref{prop:explicit-R} are intrinsic to the mechanism at hand. Density comes from the commensurabilities $T=\frac{2\pi n}{\kappa_j}$: for each $j$ they form a progression of gap $2\pi/\kappa_j\to0$, so for every $J$ their union over $j\ge J$ is dense in $(0,\infty)$; nullity comes from the exponential smallness of the bands, which makes $\sum_j\eps_j$ converge. A quasimode branch with only polynomially small residual would give a larger set of periods but a weaker resolvent bound, which does not certify infinite loss. The construction is thus not improvable to a positive-measure statement by these techniques; whether infinite loss can hold on a set of positive measure, for this or any geometry, is then left open.
\end{remark}

\begin{remark}\label{rem:ell-baire-forcing}
Fix $T\in\mathcal R_{\infty}$ and $m\in\N$. Yet another Baire argument gives that the forcings in $H^m_{\#}(0,T;\HH)$ admitting no $T$-periodic mild solution form a residual set; this is the type of statement obtained in \cite{LaurentRivas2026}.
\end{remark}

\subsection{Sharpness of the geometric conditions}\label{sec:ellipse-sharp}

Nothing in Sections~\ref{sec:ellipse-quasimode}--\ref{sec:ellipse-loss} used the wave domain beyond the slab $S_\eps$: the eigenfunctions $U_j$ are exponentially small outside it by Proposition~\ref{prop:ell-branch}(iii), the cutoff $\phi$ is supported inside it, and the interface is excluded from it by Definition~\ref{def:ellipse}. This is why the arguments above prove Theorem~\ref{thm:ellipse-general} and not only Theorem~\ref{thm:ellipse}. We record the consequences relevant for the periodic theory.

\begin{corollary}[Infinite regularity loss on ellipse configurations]\label{cor:ellipse-family}
Let $(\Om_W,\Om_H)$ be an ellipse configuration in the sense of Definition~\ref{def:ellipse}. Then:
\begin{enumerate}
\item[\textup{(a)}] for every $T_*>0$, every $m\in\N$ and every $M>0$ there are a period $T$ with $\abs{T-T_*}<1/M$ and a $T$-periodic pair $(f,u)$, smooth in time, solving \eqref{eq:hw} with
$\norm{u}_{L^2(0,T;\HH)}\ge M\norm{f}_{H^m_{\#}(0,T;\HH)}$;\footnote{Space regularity of the forcing does not help either, see Corollary~\ref{cor:obstruction}.}
\item[\textup{(b)}] for every $T$ in the dense $G_\delta$ set $\mathcal R_{\infty}$ of periods, no estimate of the form $\norm{u}_{L^2(0,T;\HH)}\le C\norm{f}_{H^m_{\#}(0,T;\HH)}$ holds, for any $m$ and any $C$.
\end{enumerate}
\end{corollary}

Consequently no a priori bound with any finite loss of derivatives can hold for the heat-wave system once a configuration of Definition~\ref{def:ellipse} is admitted. This delimits the geometric hypotheses under which periodic well-posedness in Sobolev scales has been obtained: for the truncated ellipses $\Om_W^{d}$ of Section~\ref{sec:taxonomy}, the passage from finite loss (the graph condition of \cite{MMSW2024}, $d\le0$) to infinite loss ($d>0$) takes place exactly at the semi-ellipse, as discussed after Theorem~\ref{thm:ellipse}, and any geometric or optical hypothesis ensuring finite loss must exclude every configuration of Definition~\ref{def:ellipse}. Finally, by Corollary~\ref{cor:poly-smooth}, no ellipse configuration is rationally (nor, a fortiori, uniformly) stable: the quantitative decay theory for the heat-wave system, open in general when the geometric conditions fail, has a definite negative answer here.

\section{Conclusions and new perspectives in infinite-dimensions}\label{sec:conclusions}

For a partially dissipative system the range $\mathcal R(I-S(T))$ is dense for every $T>0$ but, in the absence of uniform stability, not closed for some $T>0$. Resonance is therefore always available for some $b\in H$ outside the range, the trick in \cite[p.222]{Galdi} realizes  $b$ as the $b_f$ for $f \in L^1_{\#}$, and Theorem~\ref{thm:equiv} converts this into an unbounded response. The question is not \emph{whether} resonance occurs, but how one reaches it, and how regular the forcing is.
That question splits in two. 

One has to produce a periodic well-posedness space, a dense set of forcings whose kicks land inside the range; and one has to leave it in a controlled way. Proposition~\ref{prop:Galdi} settles the first task abstractly but names no space. Proposition~\ref{prop:weighted-pwp} names one, built from the resolvent on the sampling lattice, and, under polynomial growth, it contains a Sobolev PWP class \cite{GMW2026,LaurentRivas2026}. Neither construction describes $\mathcal R(I-S(T))$ itself, and in none of our examples do we describe it: the range is approached from inside, through spaces that sit in it, and from outside, through explicit forcings that miss it.

Derivative loss is what brings together these two ``sides". It measures how far the periodic solution mapping is from being an isomorphism in the natural energy scale, and by Remark~\ref{rem:loss-resolvent} it is the sampled resolvent growth in disguise. At finite loss, the two tasks are both solvable and the outcome is sharp in an unexpected way: smooth forcings are protected (Proposition~\ref{prop:conditional-converse-loss}), yet resonance still occurs, by stacking near-resonant harmonics into a single forcing of finite regularity ``loss" (Proposition~\ref{prop:finite-stacking}). The loss order says how rough the resonant forcing must be. 

At infinite derivative loss, the protection fails, and the stacking argument returns a $C^\infty$ forcing with no periodic response (Theorem~\ref{thm:loss-secular}).  By Corollary~\ref{cor:smooth-iff}, infinite derivative loss is exactly the condition for smooth resonance at a period whose frequency lattice lies in the resolvent set. Theorem~\ref{thm:ellipse} shows this is not just an abstract possibility. On the resonant ellipse, a trapped billiard orbit carries quasimodes that the interface cannot see, the sampled resolvent grows exponentially, and, on a dense set of periods, a smooth forcing drives the energy to infinity. The dissipation is genuine throughout: the semigroup is strongly stable and there is no spectrum on the imaginary axis. {\em Damping may eliminate free oscillations without eliminating resonance.}

Several questions remain open. We do not know which temporal growth rates are attainable within the strongly stable class. The linear rate is excluded by Corollary~\ref{cor:sublinear-ss}, the continuous-spectrum example of Proposition~\ref{prop:notconverse} gives $\sqrt t$, and the resonant ellipse gives no identified rate at all. The Sobolev threshold between the protected and the resonant regime at finite loss is not determined by Propositions~\ref{prop:conditional-converse-loss} and~\ref{prop:finite-stacking}. Whether infinite loss can hold on a set of periods of positive measure is also open, and by Remark~\ref{rem:ell-sharpness} it is not reachable by our present constructions. Finally, the quantitative decay theory for the heat-wave system when the geometric conditions fail, remains, in general, open.  However, Corollary~\ref{cor:ellipse-family} settles it negatively for the specific geometries  here.

\appendix

\section{Proof of Proposition~\ref{prop:ell-branch}}

We give a self-contained proof for completeness. We utilize: separation of variables in elliptic coordinates, existence and quantitative localization (a so called Agmon estimate) of the ground state of the resulting eigenproblem, and existence of an unbounded family of wavenumbers. For the relation to \cite{NguyenGrebenkov2013} see Section~\ref{sec:ellipse-quasimode}.

\subsection{Elliptic coordinates}\label{ssec:setup}
We consider here an ellipse $E$ given by $\frac{x^2}{a^2}+\frac{y^2}{b^2}\leq 1$, with $a>b>0$.

Set $c:=\sqrt{a^2-b^2}>0$ as the stretching radius. Then the elliptic coordinates
$(\xi,\eta)\in[0,\xi_0]\times[0,2\pi)$ are defined by
\begin{equation}\label{eq:ellipt-coord}
 x=c\cosh\xi\cos\eta,\qquad y=c\sinh\xi\sin\eta,
\end{equation}
where $\xi_0$ is determined by $c\cosh\xi_0=a$, i.e.\ $\cosh\xi_0=a/c$. This implies that 
\[
\sinh^2\xi_0=\cosh^2\xi_0-1=\frac{a^2-c^2}{c^2}={\frac{b^2}{c^2}}.
\]
Consequently $\sinh\xi_0=b/c$ and $c\sinh\xi_0=b$. 
Hence they cover $E$, sending
$\{\xi=\xi_0\}$ to $\partial E$. For a fixed arc coordinate $\eta$ the coordinate $\xi$ monotonically parametrizes a confocal hyperbola arc between $(c\cos\eta,0)$ and $(a\cos \eta, b\sin \eta)$. 
The curves parametrized by $\eta$ and $\xi$ are orthogonal and have the same speed of motion. Indeed:
\begin{align}
\label{eq:coord}
\begin{aligned}
&\frac{\partial x}{\partial \xi}(\xi,\eta)= c\sinh(\xi)\cos(\eta)=:\alpha(\xi,\eta),
\, \frac{\partial y}{\partial \xi}(\xi,\eta)= c\cosh(\xi)\sin(\eta)=:\beta(\xi,\eta)
\\
& \frac{\partial x}{\partial \eta}(\xi,\eta)= -c\cosh(\xi)\sin(\eta)=-\beta(\xi,\eta)
\, \frac{\partial y}{\partial \eta}(\xi,\eta)= c\sinh(\xi)\cos(\eta)=\alpha(\xi,\eta)
\end{aligned}
\end{align}
Now 
\begin{align*}
\abs{\partial_\xi(x,y)}^2&=c^2 (\sinh^2\xi\cos^2\eta+\cosh^2\xi \sin^2\eta)=c^2 (\cosh^2\xi-\cos^2\eta)\\
&=\alpha^2(\xi,\eta)+\beta^2(\xi,\eta)=\abs{\partial_\eta(x,y)}^2=:h^2
\end{align*}
 and 
$\partial_\xi(x,y)=\partial_\eta(x,y)^\perp$.
The resulting coordinate change is a stretched rotation. Introducing the rotational matrix 
\[
Q=\frac{1}{h}\begin{pmatrix}
\alpha & \beta
\\
-\beta &  \alpha
\end{pmatrix}
\]
implies that for any function $f\in C^2(E)$
\begin{equation}\label{eq:laplace}
\begin{split}
\mathrm{tr}(\nabla^2f(x(\xi,\eta),y(\xi,\eta)))&=\mathrm{tr}(Q^T \nabla^2f(x(\xi,\eta),y(\xi,\eta)) Q)\\
&=\frac{1}{h^2}(\partial_\xi^2 \tilde{f}(\xi,\eta)+\partial_\eta^2\tilde{f}(\xi,\eta)),
\end{split}
\end{equation}
where $\tilde{f}(\xi,\eta):=f(x(\xi,\eta),y(\xi,\eta))$.
Hence
$\Delta f (x(\xi,\eta),y(\xi,\eta))=\frac{1}{h^2}\Delta \tilde{f}(\xi,\eta)$.
 Please note that such a coordinate change would not be possible in the circular case, where $c=0$.

\subsection{Separation of variables}
For this setting we wish to create a family of eigenvalues.
\[
-\Delta u=\kappa^2 u\text{ in }E,\quad u=0\text{ on }\partial E.
\]
We construct a sequence of eigenvalues $\kappa_n^2$, $n\ge N$, with $\kappa_n=n\pi/b+O(1)$. 

 For that we use the separated Ansatz $\tilde{u}(\xi,\eta)=F(\xi)G(\eta)$ into
$-\Delta \tilde{u}=\kappa^2h^2\tilde{u}$.
Dividing by $FG$, which we assume to be nonzero, we find using the definition of $h$ that 
\[
 -\frac{F''(\xi)}{F(\xi)}-\frac{G''(\eta)}{G(\eta)}=\kappa^2c^2(\cosh^2\xi-\cos^2\eta),
\]
which implies that there is a constant $\Lambda\in \R$, such that
\begin{equation}\label{eq:separated}
 L_\kappa G(\eta)=\Lambda G(\eta),\qquad L_\kappa G(\eta):=-G''(\eta)+\kappa^2c^2\cos^2\eta G(\eta),
\end{equation}
and
\begin{equation}\label{eq:Lk-def}
 F''(\xi)+\bigl(\kappa^2c^2\cosh^2\xi-\Lambda\bigr)F(\xi)=0.
\end{equation}
Both  systems become ordinary differential equations, with different boundary values. 
First $G$ needs to be periodic $G(\eta+2\pi)=G(\eta)$.
For $F$ we  first need to require $F(\xi_0)=0$, which are the zero boundary values of $u$ on $\partial E$, second we require that $F'(0)
 =0$: the points $(0,\eta)$ and $(0,-\eta)$ have the same image (the focal segment), and, $G$ being even, $F'(0)=0$ is exactly what makes $F(\xi)G(\eta)$ smooth across the focal segment (see Section~\ref{ssec:smooth-ext}); otherwise the normal derivative would jump there and $FG$ would not solve the eigenvalue equation across the segment.
 Now if solutions for these boundary conditions were unique for all $\kappa,\Lambda$ it would not be possible to get an eigenfunction using the separation Ansatz, as $F\equiv 0$. This is of course not true. Nevertheless one has to construct $\kappa$ and $\Lambda$ in a very careful manner to produce a non-trivial eigenfunction. 

The following steps are performed:

\begin{enumerate}
\item For all $\kappa\in \R$, we show the existence of a  $\Lambda(\kappa)$ the smallest eigenvalue of $L_\kappa$ with positive and $2\pi$ periodic eigenfunction $G_\kappa(\eta)$.
This so called ground state will indeed with increasing $\kappa$ concentrate around the $y$-axes and have an exponential decay in radial direction.
\item We then consider for all $\kappa,\Lambda_\kappa$ the solution $F_\kappa$ to the {\em initial value problem}
\begin{align}
\label{eq:initial}
F''_\kappa (\xi)+\bigl(\kappa^2c^2\cosh^2\xi-\Lambda_\kappa\bigr)F_\kappa(\xi)=0, \quad F_\kappa(0)=1,\quad F_\kappa'(0)=0.
\end{align}
\item We then deduce the existence of an increasing sequence $\kappa_n=n\pi/b+O(1)$ such that $F_{\kappa_n}(\xi_0)=0$.
\item Finally we show that the resulting separated function $F(\xi)G(\eta)$ then
defines a genuine element of $H_0^1(E)\cap C^\infty(\overline{E})$ satisfying {$-\Delta u=\kappa^2 u$} in $E$, and we prove the lower and localization bounds.
\end{enumerate} 

\subsection{Step 1: The angular ground state}\label{ssec:ground-state}

\begin{lemma}\label{lem:ground-state}
For every $\kappa>0$ the smallest eigenvalue $\Lambda_\kappa>0$ is simple, with a real, $2\pi$-periodic, strictly
positive eigenfunction $G_\kappa\in C^\infty(\T_{2\pi})$, normalized by
$\norm{G_\kappa}_{L^2(\T_{2\pi})}=1$, that satisfies both geometric axial symmetries, namely $G(-\eta)=G(\eta)=G(\eta+\pi)=G(\pi-\eta)$.
\end{lemma}

\begin{proof}
\emph{Existence and positivity.} The operator $L_\kappa$ with domain $H^2(\T_{2\pi})$ is a self-adjoint operator on $L^2(\T_{2\pi})$.  It is injective, since $V_\kappa\ge0$. Now $L_\kappa^{-1}:L^2(\T_{2\pi})\to L^2(\T_{2\pi})$ is self-adjoint, positive and compact. Hence the spectrum of $L_\kappa$ is discrete, real, strictly positive and the eigenfunctions provide a smooth orthonormal basis of $L^2(\T_{2\pi})$.
We introduce
\[
V_\kappa(\eta):=\kappa^2c^2\cos^2\eta
\]
Hence the variational
characterization of the first eigenvalue
\[
 \Lambda_\kappa=\inf_{\phi\in H^1(\T_{2\pi})\setminus\{0\}}
 \frac{\int_{\T_{2\pi}}((\phi')^2+V_\kappa\phi^2)\,d\eta}{\int_{\T_{2\pi}}\phi^2\,d\eta}
\]
is attained by some $G_\kappa\in C^\infty(\T_{2\pi})$ with $\norm{G_\kappa}_{L^2(\T_{2\pi})}=1$.
The Euler--Lagrange equation gives {$L_\kappa G_\kappa=\Lambda_\kappa G_\kappa$}.

Since $V_\kappa \phi^2$ depends only on $|\phi|$ and the weak derivative of $|\phi|$
satisfies $|\,|\phi|'\,|\le|\phi'|$ a.e.\ (Stampacchia's truncation),
$|G_\kappa|$ attains the same Rayleigh quotient $\Lambda_\kappa$ and is therefore
also a (smooth) eigenfunction. Hence we can choose
$G_\kappa \ge 0$. Further if $G_\kappa(\eta^*)=0$ for some $\eta^*\in\T_{2\pi}$, then $\eta^*$ is
a minimum forcing $G_\kappa'(\eta^*)=0$. Now the uniqueness of the initial value problem 
$-G_\kappa''+(V_\kappa-\Lambda_\kappa)G_\kappa=0$ with the initial data $G_\kappa(\eta^*)=G_\kappa'(\eta^*)=0$ would imply $G_\kappa\equiv 0$ contradicting $\norm{G_\kappa}_{L^2}=1$. Hence $G_\kappa>0$ on $\T_{2\pi}$.

\emph{Simplicity.} Suppose that there was another $G_1\in H^1(\T_{2\pi})$ an eigenfunction with the
same eigenvalue $\Lambda_\kappa$ with
$\int G_0G_1\, d\eta =0$. By the same argument as before one would find that also  $|G_1|$ is a (smooth)
eigenfunction with eigenvalue $\Lambda_\kappa$, with $|G_1|>0$ on $\T_{2\pi}$. But this implies that $G_1$ has a constant sign, which implies that $\int G_0G_1\, d\eta\ne 0$, hence the eigenspace is one-dimensional.

\emph{Symmetries.} Observe that $V_k(-\eta)=V_k(\eta)=V_k(\eta+\pi)=V_k(\pi-\eta)$. Hence $G_\kappa(-\eta)$, $G_\kappa(\pi+\eta)$ and $G_\kappa(\pi-\eta)$ have the same Rayleigh quotient. The uniqueness of $G_\kappa$ implies that it satisfies the symmetries.
\end{proof}

\begin{lemma}[Boundedness]\label{lem:bounded}
There is an $M>0$ such that
\begin{equation}\label{eq:Lambda0-bound}
 0<\Lambda_\kappa\le M\kappa\qquad\text{for all }\kappa\geq 1.
\end{equation}
\end{lemma}

\begin{proof} The situation is illustrated in
Figure~\ref{fig:angular-potential}.
The potential
$V_\kappa(\eta)=\kappa^2c^2\cos^2\eta$
has wells at $\eta=\pi/2$ and $\eta=3\pi/2$,
where it vanishes, and maxima of height
$\kappa^2c^2$ at $\eta=0$ and $\eta=\pi$.

\begin{figure}[htbp]
\centering


\begin{tikzpicture}
\begin{axis}[
    width=11cm,
    height=6.5cm,
    domain=0:2*pi,
    samples=300,
    axis lines=left,
    xlabel={$\eta$},
    ylabel={$\cos^2\eta$},
    xmin=0, xmax=2*pi,
    ymin=0, ymax=1.1,
    xtick={0,pi/2,pi,3*pi/2,2*pi},
    xticklabels={$0$,$\frac{\pi}{2}$,$\pi$,
        $\frac{3\pi}{2}$,$2\pi$},
    ytick={0,1},
    tick style={black},
]
\addplot[thick,black] {cos(deg(x))^2};
\end{axis}
\end{tikzpicture}

\caption{Normalized angular potential
$V_\kappa(\eta)/(\kappa^2c^2)=\cos^2\eta$.
The potential vanishes at the two wells
$\eta=\pi/2$ and $\eta=3\pi/2$, and attains
its maximum at $\eta=0,\pi$ modulo $2\pi$.}
\label{fig:angular-potential}
\end{figure}

Consequently, with increasing $\kappa$, the minimizer of the
Rayleigh quotient concentrates around $\eta=\pi/2$ and
$\eta=3\pi/2$.

Near either well, the potential is approximately quadratic,
since $\cos(\pi/2+\theta)\sim-\theta$.
This suggests a Gaussian competitor concentrated at the
minimum, as illustrated schematically in
Figure~\ref{fig:gaussian-competitor}.
The concentration becomes sharper with increasing $\kappa$,
on a spatial scale of order $\kappa^{-1/2}$.

\begin{figure}[htbp]
\centering


\begin{tikzpicture}
\begin{axis}[
    width=11cm,
    height=6.5cm,
    domain=-pi/2:pi/2,
    samples=200,
    xmin=-pi/2, xmax=pi/2,
    ymin=0, ymax=2.6,
    axis lines=middle,
    xlabel={$x$},
    ylabel={$y$},
    xtick={-pi/2,0,pi/2},
    xticklabels={$-\frac{\pi}{2}$,$0$,
        $\frac{\pi}{2}$},
    legend style={
        at={(0.98,0.98)},
        anchor=north east
    },
    grid=both,
]
\addplot[black,thick,solid] {x^2};
\addlegendentry{$x^2$}

\addplot[black,thick,dashed] {exp(-x^2)};
\addlegendentry{$e^{-x^2}$}

\end{axis}
\end{tikzpicture}

\caption{Local quadratic potential and Gaussian
trial profile. The parabola $x^2$ approximates
the potential near a well, while $e^{-x^2}$
illustrates Gaussian localization at its minimum.
The actual competitor has a $\kappa$-dependent
width of order $\kappa^{-1/2}$.}
\label{fig:gaussian-competitor}
\end{figure}

We now construct an explicit Gaussian competitor that provides
an upper bound for $\Lambda_\kappa$. In fact, this construction
already captures the essential behavior of the ground state
$G_\kappa$.

Setting $\theta=\eta-\pi/2$, we use the approximation
$\cos(\pi/2+\theta)\sim-\theta$ near the first potential well.
For $\gamma>0$, observe that
\[
\partial_\theta e^{-\gamma\theta^2}
=-2\gamma\theta e^{-\gamma\theta^2},
\qquad
\partial_\theta^2 e^{-\gamma\theta^2}
=(-2\gamma+4\gamma^2\theta^2)e^{-\gamma\theta^2}.
\]

Consequently,
\[
\int c^2\kappa^2 \cos^2(\frac{\pi}{2}+\theta)(e^{-\theta^2c\kappa})^2\sim\int (\partial_\theta e^{-\theta^2c\kappa})^2\, d\eta=
\int 4\theta^2c^2\kappa^2 e^{-2\theta^2c\kappa}\, d\eta
\sim \sqrt{\kappa}
\]
where we used that (for any $\delta$)
\begin{align*}
\int^\delta_{-\delta} \theta^2 e^{-a\theta^2}\, d\theta
&=\frac{1}{a \sqrt{a}}\int^{\delta\sqrt{a}}_{-\delta\sqrt{a}}z^2e^{-z^2}\, dz\\
&= \frac{1}{a \sqrt{a}}\int^{\delta\sqrt{a}}_{-\delta\sqrt{a}}-\partial_z \frac{ze^{-z^2}}{2} + \frac{e^{-z^2}}{2}\, dz
\leq \frac{\sqrt{\pi}}{2a \sqrt{a}}.
\end{align*}
Now as $ \int (e^{-c\kappa \theta^2})^2\, dx= \frac{\sqrt{\pi}}{\sqrt{2c\kappa}}$, we find that the competitor is essentially of the form $\kappa^\frac{1}{4}e^{-c\kappa \theta^2}$ and
\[
\Lambda_\kappa\leq M \kappa.
\]
 This can be made explicit by using cut-off functions. For that choose
$\chi\in C_c^\infty(\R)$ even with $\chi=1$ on $[-1/4,1/4]$,
$\supp\chi\subset(-1/2,1/2)$, and $0\le\chi\le 1$ and setting
\[
 \psi_k(\eta):=\chi(\eta-\pi/2)\,e^{-(kc/2)(\eta-\pi/2)^2}
 +\chi(\eta-3\pi/2)\,e^{-(kc/2)(\eta-3\pi/2)^2}.
\]
It is easy to check that the above calculations applied to this function result in a valid explicit competitor with Rayleigh quotient of size $\kappa$.
\end{proof}
The competitor $\psi_\kappa$ already predicts the essence of $G_\kappa$: with increasing $\kappa$ the ground state concentrates in Gaussian fashion, on a scale $\kappa^{-1/2}$, around the wells $\eta=\pi/2$ and $\eta=3\pi/2$ of the potential $V_\kappa(\eta)=\kappa^2c^2\cos^2\eta$, which vanishes there and has maxima $\kappa^2c^2$ at $\eta=0,\pi$. This is made quantitative in the next subsection.

\subsection{Step 2: Agmon localization of the ground state}\label{ssec:agmon}
Let \[
 S_\delta:=\bigl\{\eta\in[0,2\pi):\,|\eta-\pi/2|>\delta\text{ and }|\eta-3\pi/2|>\delta\bigr\},
\]
\begin{lemma}\label{lem:agmon}
There are constants $C, c_1, c_2>0$, depending only on $c$, such that for all $\delta\in (0,\frac{\pi}{4}]$ and all $\kappa\geq \frac{c_2}{\delta^2}$,
\begin{equation}\label{eq:agmon-bound}
 \int_{S_\delta}|G_\kappa|^2\,d\eta
 \le \frac{C}{\delta^2\kappa}\,e^{-c_1\delta^2 \kappa}\qquad\text{ and }\qquad
\int_{S_\delta}|G_\kappa'|^2\,d\eta
 \le C\kappa\,e^{-c_1\delta^2 \kappa}.
\end{equation}
\end{lemma}

\begin{proof}
Fix $\delta\in(0,\pi/4]$ and recall $0<\Lambda_\kappa \le M\kappa$ (Lemma~\ref{lem:bounded}) and $\norm{G_\kappa}_{L^2(\T_{2\pi})}=1$.

Let $W_\delta:=\{|\eta-\pi/2|\le\delta/2\}\cup\{|\eta-3\pi/2|\le\delta/2\}$ and $T_\delta :=\{\delta/2<|\eta-\pi/2|<\delta\}\cup\{\delta/2<|\eta-3\pi/2|<\delta\}$, so that, up to a null set, $\T_{2\pi}=W_\delta\cup T_\delta\cup S_\delta$. Let $\varphi\ge0$ be the $2\pi$-periodic Lipschitz function with
\begin{gather*}
 \varphi=0\text{ on }W_\delta,\qquad
 \varphi=\alpha \kappa\text{ on }S_\delta,\\
\abs{\varphi'}= \frac{2\alpha \kappa}{\delta}\chi_{T_\delta}\ \text{(i.e.\ it is linear on each component of $T_\delta$)},
\end{gather*}
where $\alpha>0$ will be fixed in \eqref{eq:alpha-choice} below. Multiply $L_\kappa G_\kappa=\Lambda_\kappa G_\kappa$ by $e^{2\varphi}G_\kappa$ and
integrate; periodicity kills all boundary terms in the integration by
parts, and we find
\[
 \int_{\T_{2\pi}}e^{2\varphi}(G_\kappa')^2\,d\eta
 +2\int_{\T_{2\pi}}e^{2\varphi}\varphi'G_\kappa G'_\kappa\,d\eta
 +\int_{\T_{2\pi}}(V_\kappa-\Lambda_\kappa)e^{2\varphi}G_\kappa^2\,d\eta=0.
\]
Using Young's inequality, namely $2\varphi'GG'\ge-\frac{(G')^2}{2}-2(\varphi')^2G^2$, we obtain
\begin{equation}\label{eq:agmon-master}
 \int_{\T_{2\pi}}\frac{e^{2\varphi}}{2}(G_\kappa')^2\,d\eta+\int_{\T_{2\pi}}\bigl(V_\kappa-\Lambda_\kappa-2(\varphi')^2\bigr)e^{2\varphi}G_\kappa^2\,d\eta\le 0.
\end{equation}
Let $d(\eta)\in[0,\pi/2]$ denote the distance from $\eta$ to $\{\pi/2,3\pi/2\}$ on $\T_{2\pi}$. Then $\abs{\cos\eta}=\sin d(\eta)\ge2d(\eta)/\pi$, and since $d>\delta/2$ on $W_\delta^c$,
\begin{equation}\label{eq:V-lower}
 V_\kappa(\eta)\ge  \tilde{c}_1 \kappa^2c^2\delta^2\quad\text{on }W_\delta^{c}=T_\delta\cup S_\delta,\qquad \tilde c_1:=\pi^{-2}.
\end{equation}
Choose
\begin{equation}\label{eq:alpha-choice}
 \alpha:=\frac{c\sqrt{\tilde{c}_1}}{4}\,\delta^2 .
\end{equation}
Then on $T_\delta$ we have $2(\varphi')^2=\frac{8\alpha^2\kappa^2}{\delta^2}=\frac{\tilde c_1}{2}\kappa^2c^2\delta^2\le\frac{V_\kappa}{2}$. Put $c_2:=4M/(\tilde c_1c^2)$; for $\kappa\ge c_2/\delta^2$ we have $\Lambda_\kappa\le M\kappa\le\frac{\tilde c_1}{4}\kappa^2c^2\delta^2\le\frac{V_\kappa}{4}$ on $W_\delta^c$, hence
\[
 V_\kappa-\Lambda_\kappa-2(\varphi')^2\ge
 \begin{cases}
 -\Lambda_\kappa  & \text{on }W_\delta,\\
 0 & \text{on }T_\delta,\\
 \tfrac{3\tilde{c}_1}{4} \kappa^2c^2\delta^2 & \text{on }S_\delta .
 \end{cases}
\]
Discarding the nonnegative contribution of $T_\delta$ in \eqref{eq:agmon-master}, and using $\varphi=0$ on $W_\delta$ and $\varphi=\alpha\kappa$ on $S_\delta$,
\[
 e^{2\alpha \kappa}\Bigl(\frac{1}{2}\int_{S_\delta}\abs{G_\kappa'}^2\,d\eta+\frac{3\tilde c_1}{4}\kappa^2c^2\delta^2\int_{S_\delta}G_\kappa^2\,d\eta\Bigr)
 \le\Lambda_\kappa\int_{W_\delta}G_\kappa^2\,d\eta\le M\kappa .
\]
With $c_1:=2\alpha/\delta^2=c\sqrt{\tilde c_1}/2$ and $C:=\max\{2M,\,4M/(3\tilde c_1c^2)\}$ this is \eqref{eq:agmon-bound}.
\end{proof}
\subsection{Step 3: The ``radial equation"}\label{ssec:radial}
Next we analyze the following family of initial value problems:
\begin{align}
\label{eq:ivp}
\begin{aligned}
-F''_\kappa(\xi)=(\kappa^2c^2\cosh^2\xi-\Lambda_\kappa)F_\kappa=:R_\kappa(\xi)F_\kappa 
\\
F_\kappa(0)=1,\quad F'_\kappa=0
\end{aligned}
\end{align}
The following lemma shows the continuity of $R_\kappa$ in dependence of $\kappa$, which directly implies the continuity of $F_\kappa$ in dependence of $\kappa$ (and $\xi$), by standard ODE theory.
\begin{lemma}
We find that $R_\kappa(\xi)$ is continuous and positive for all 
$(\xi,\kappa)\in[0,\infty)\times[\frac{2M}{c^2},\infty)$.
\end{lemma}
\begin{proof}
Since $\cosh^2\xi\ge 1$ and $\Lambda_\kappa\le M\kappa$ by
Lemma~\ref{lem:bounded},
\[
 R_\kappa(\xi)\ge \kappa^2c^2-M\kappa\ge\tfrac12 \kappa^2c^2
 \qquad\text{whenever }\kappa \ge 2M/c^2.
\]
We henceforth restrict to $k> 2M/c^2$, so that $R_\kappa>0$ uniformly on
$[0,\infty)$.

Next we show that the map $k\mapsto\Lambda_\kappa$ is continuous on
$(0,\infty)$. Recall that
\[
 \Lambda_\kappa=\inf_{\phi\in H^1(\T_{2\pi})\setminus\{0\}}
 \frac{\int_{\T_{2\pi}}\bigl((\phi')^2+V_k\phi^2\bigr)d\eta}{\int_{\T_{2\pi}}\phi^2\,d\eta}.
\]
Now denoting $J_\kappa(\phi)=\int_{\T_{2\pi}}\bigl((\phi')^2+V_k\phi^2\bigr)d\eta$, we find assuming that 
$\Lambda_{\kappa}\geq\Lambda_{\kappa_0}$, 
that
\begin{equation}\label{eq:Lambda0-Lipschitz}
\begin{split}
\abs{\Lambda_{\kappa}-\Lambda_{\kappa_0}}&=J_\kappa(G_\kappa)-J_{\kappa_0}(G_{\kappa_0})\leq J_\kappa(G_{\kappa_0})-J_{\kappa_0}(G_{\kappa_0})\\
&=\int_{\T_{2\pi}}(V_\kappa-V_{\kappa_0})G_{\kappa_0}^2\, d\theta\leq \norm{V_\kappa-V_{\kappa_0}}_\infty.
\end{split}
\end{equation}
This implies not only the continuity but also the monotonicity of $\Lambda_\kappa$ as the above shows that $\kappa_0\leq \kappa$ implies $\Lambda_{\kappa_0}\leq \Lambda_{\kappa}$. 
This implies now directly the continuity of $R_\kappa$ and the proof is complete.
\end{proof}

\begin{lemma}\label{lem:radial-zeros}
There is an increasing sequence $\kappa_n\to\infty$,
such that the unique solution to \eqref{eq:ivp} satisfies $F_{\kappa_n}(\xi_0)=0$.
Moreover $\lim_{n\to \infty} \frac{\kappa_n}{n}=\frac{\pi}{b}$, with explicit control
\[
n-C_1 \leq \kappa_n \frac{b}{\pi } \leq n +C_2,
\]
 for all $n$ large enough, with $C_1,C_2$ depending on $M$, $c$ and $\xi_0$.
\end{lemma}

\begin{proof}
The proof relies on polar coordinates for the solution, the so called Pr\"ufer transformation. For $\kappa>\frac{2M}{c^2}$, $R_\kappa(\xi)>0$ on $[0,\xi_0]$. However, for technical reasons the polar coordinates are chosen such that $F$ is situated at the $y$-axis and $F'$ at the $x$-axis.
Since $F_\kappa$ and $F'_\kappa$ cannot simultaneously vanish (by uniqueness for the linear
homogeneous initial value problem), there are continuous functions $\rho(\xi,\kappa)>0$ and
$\theta(\xi,\kappa)$ on $[0,\xi_0]$ with
\begin{equation}\label{eq:prufer}
 F_\kappa(\xi)=\rho(\xi,\kappa)\sin\theta(\xi,\kappa),\qquad
 F'_\kappa(\xi)=\rho(\xi,\kappa)\sqrt{R_\kappa(\xi)}\cos\theta(\xi,\kappa).
\end{equation}
The amplitude is determined by the energy
$\rho^2(\cdot,\kappa)=F^2_\kappa+(F'_\kappa)^2/R_\kappa$, and the phase by
$\cot\theta(\cdot,\kappa)=\frac{\sqrt{R_\kappa}F_\kappa'}{F_\kappa}$. Further note that $\rho(0,\kappa)=1$ and $\theta(0,\kappa)=\frac{\pi}{2}$.
Hence in the vicinity of $\xi=0$, we find
$\theta(\cdot,\kappa)=\textrm{arccot}(\sqrt{R_\kappa}\,F_\kappa'/F_\kappa)$ (once $F_\kappa$ approaches $0$, the $\arctan$ could be used).  

{\em Determining an ODE for $\theta$.}
We wish to find an explicit ODE for $\theta(\cdot, \kappa)$, as $F_\kappa(\xi_0)=0$ is equivalent for $\sin(\theta_\kappa(\xi_0))=0$, which means that $\theta_\kappa(x_0)=n\pi$, for some $n\in \N$. This will eventually produce the infinitely many $\kappa_n$.

For that the first identity in \eqref{eq:prufer} is differentiated and compared with the second one,
\begin{align}
\label{eq:1}
\partial_\xi\rho\sin\theta+\rho\partial_\xi \theta\cos\theta=F'_\xi=\rho\sqrt{R_\kappa}\cos\theta.
\end{align}
Differentiating the second and using $F_\kappa''=-R_\kappa F$ yields
\begin{align}
\label{eq:2}
\partial_\xi\rho\sqrt{R_\kappa}\cos\theta+\rho\frac{R_\kappa'}{2\sqrt{R_\kappa}}\cos\theta
-\rho\sqrt{R_\kappa}\partial_\xi \theta \sin\theta=-R_\kappa F_\kappa=-R_\kappa\rho\sin\theta
\end{align}
Multiplying \eqref{eq:1} $\cos\theta$ implies
\[
\partial_\xi\rho\sin\theta\cos\theta+\rho\partial_\xi \theta\cos^2\theta=\rho\sqrt{R_\kappa}\cos^2\theta
\]
Multiplying \eqref{eq:2} with $\sin\theta/\sqrt{R_k}$ implies
\[
\rho \partial_\xi\theta \sin^2\theta=\partial_\xi\rho\sin(\theta)\cos(\theta)+\rho\frac{R'_\kappa}{2R_\kappa}\sin(\theta)\cos(\theta)+\rho\sqrt{R_\kappa}\sin^2\theta
\]
Adding them removes the $\partial_\xi \rho$ term such that we reach
\[
\rho \partial_\xi \theta = \rho\frac{R'_\kappa}{2R_\kappa}\sin(\theta)\cos(\theta)+\rho\sqrt{R_\kappa}
\] 
which simplifies to
\begin{equation}\label{eq:prufer-ODE}
 \partial_\xi \theta=\sqrt{R_\kappa}+\frac{R_\kappa'}{4R_\kappa}\sin(2\theta).
\end{equation}
Both $\rho$ and $\theta$ depend jointly continuously on $(\xi,\kappa)\in
[0,\xi_0]\times[K_0,\infty)$. Indeed, since on this set $R_\kappa$ is continuous and bounded
below by $\kappa^2c^2/2>0$, so the algebraic relations defining $(\rho,\theta)$
from $(F_\kappa,F'_\kappa)$ are smooth functions of $(F_\kappa,F'_\kappa,R_\kappa)$, which themselves
depend continuously on $\kappa$ as noted in the paragraph preceding the lemma.

\emph{Phase growth.}
We find the following quantitative bound:
\begin{align}
\label{eq:Rbound}
R_\kappa(\xi)=\kappa^2c^2\cosh^2\xi-\Lambda_\kappa\geq \kappa^2c^2 (\cosh^2\xi-1)+(\kappa c^2-M)\kappa
\end{align}
This implies in particular that $R_\kappa\geq \kappa^2c^2/2>0$ for $\kappa\geq \frac{2M}{c^2}$. Further, one can deduce that
\begin{align}
\label{eq:assymptotics}
\begin{aligned}
\kappa^2c^2\Big(\cosh^2\xi - \frac{M}{c^2\kappa}\Big)&\leq R_\kappa(\xi)\leq \kappa^2c^2\cosh^2\xi,
\intertext{which implies, defining $\tilde{M}:=\frac{M}{c^2}$ that}
c^2 \kappa^2 \cosh^2\xi \Big(1 - \frac{\tilde{M}}{\kappa}\Big)&\leq R_\kappa(\xi)\leq c^2 \kappa^2 \cosh^2\xi
\end{aligned}
\end{align}
This essentially describes the dynamics of $\theta$ (and as such the zeros of $F$). Indeed, as also
$0\leq R_\kappa'(\xi)=2\kappa^2c^2\sinh\xi\cosh\xi\leq 4\kappa^2c^2\sinh\xi_0\cosh\xi_0$ on $[0,\xi_0]$ we find that the "correction term" that is related to the inhomogeneity of the energy is bounded by a constant independent of $\kappa$:
\begin{align}
\label{eq:R}
0\leq \frac{R_k'}{4R_k}\leq C_0.
\end{align}
 Now integrating \eqref{eq:prufer-ODE} and using
the bounds implies that
\begin{align*}
 \theta(\xi_0,\kappa)-\theta(0,\kappa)
 &=\int_0^{\xi_0}\sqrt{R_\kappa(\xi)}+\frac{R_k'}{4R_k}\sin(2\theta)\, d\xi\\
 &\geq\kappa c \sqrt{1-\frac{\tilde{M}}{\kappa}} \int_0^{\xi_0}\cosh\xi\,d\xi- \xi_0 C_0
\\
 &=\kappa c \sqrt{1-\frac{\tilde{M}}{\kappa}}\sinh\xi_0- \xi_0 C_0
 =\kappa b  \sqrt{1-\frac{\tilde{M}}{\kappa}}- \xi_0 C_0\geq \frac{\kappa b}{2}
\end{align*}
for $\kappa$ sufficiently large. Since $\theta(\xi_0,\cdot)$ is continuous on $[K_0,\infty)$, $K_0:=2M/c^2$, and tends to $+\infty$, the intermediate value theorem provides, for every integer $n>\theta(\xi_0,K_0)/\pi$, a solution of $\theta(\xi_0,\kappa)=n\pi$ in $(K_0,\infty)$; let $\kappa_n$ be the smallest one. Then $\sin\theta(\xi_0,\kappa_n)=0$, i.e.\ $F_{\kappa_n}(\xi_0)=0$, and $\kappa_n<\kappa_{n+1}$: otherwise $\theta(\xi_0,\cdot)$ would take the value $n\pi$ on $(K_0,\kappa_{n+1})\subset(K_0,\kappa_n)$, contradicting minimality. The values $\kappa_n$ have a precise asymptotic. By \eqref{eq:assymptotics} we also have the upper bound
\begin{align*}
 \theta(\xi_0,\kappa)-\theta(0,\kappa)
 &=\int_0^{\xi_0}\sqrt{R_\kappa(\xi)}+\frac{R_\kappa'}{4R_\kappa}\sin(2\theta)\, d\xi\\
 &\leq \kappa c \int_0^{\xi_0}\cosh\xi\,d\xi+ \xi_0 C_0
 =b\kappa+ \xi_0 C_0.
\end{align*}
Using $\theta(0,\kappa)=\frac{\pi}{2}$, $\theta(\xi_0,\kappa_n)=n\pi$ and $\sqrt{1-z}\ge1-z$ for $z\in[0,1]$, we deduce
\[
b\kappa_n-b\tilde{M} - \xi_0 C_0+\frac{\pi}{2}\leq 
\theta(\xi_0,\kappa_n)=\pi n \leq b\kappa_n+\xi_0 C_0+\frac{\pi}{2},
\]
that is, $n-C_1\le\kappa_nb/\pi\le n+C_2$ with $C_1:=(\xi_0C_0+\pi/2)/\pi$ and $C_2:=(b\tilde M+\xi_0C_0-\pi/2)/\pi$.
\end{proof}
We conclude the section by showing some uniform bounds on the radial solution.
\begin{lemma}
\label{lem:boundedF}
The solution $F_\kappa$ to the initial value problem \eqref{eq:ivp} satisfies
\[
\abs{F'_\kappa}^2(\xi)+R_\kappa(\xi)\abs{F_\kappa(\xi)}^2\leq R_\kappa(0)e^{4C_0\xi},
\]
and so there is a constant $B,\tilde{b}$ independent of $\kappa$ such that
\[
\norm{F_\kappa}_\infty\leq Be^{\tilde{b}\xi_0}
,\quad \norm{F_\kappa'}_\infty\leq \kappa Be^{\tilde{b}\xi_0}.
\]
If moreover $F_\kappa(\xi_0)=0$, as is the case for $\kappa=\kappa_n$ by Lemma~\ref{lem:radial-zeros}, then
\[
\norm{F'_\kappa}_{L^2(0,\xi_0)}=\norm{\sqrt{R_\kappa}F_\kappa}_{L^2(0,\xi_0)}\geq \frac{\sqrt{\xi_0} c \kappa}{2}.
\]
\end{lemma}
\begin{proof}
Multiplying the equation with $F'_\kappa$ implies using \eqref{eq:R} that
\begin{align*}
\partial_\xi\Big(\frac{\abs{F'_\kappa}^2}{2}+R_\kappa\frac{ \abs{F_\kappa}^2}{2}\Big)= R'_\kappa\frac{\abs{F_\kappa}^2}{2}\leq  4C_0 \Big(R_\kappa \frac{\abs{F_\kappa}^2}{2}+\frac{\abs{F'_\kappa}^2}{2}\Big).
\end{align*}
Hence the upper bounds follow by Gronwall and \eqref{eq:assymptotics}.
For the lower bound we use that that $R_\kappa'\geq 0$ and \eqref{eq:assymptotics}, which implies by the above for $\kappa$ large enough that
\[
\abs{F'_\kappa}^2(\xi)+R_\kappa(\xi)\abs{F_\kappa}^2(\xi)\geq R_\kappa(0)\geq \frac{c^2 \kappa^2}{2}.
\]
Further as $F(\xi_0)=0$ and $F'(0)=0$ equipartition is available, $\int_0^{\xi_0}\abs{F'_\kappa}^2\,d\xi=\int_0^{\xi_0}R_\kappa\abs{F_\kappa}^2\,d\xi$, and so
\[
2\int_0^{\xi_0} \abs{F'_\kappa}^2\, d\xi =\int_0^{\xi_0}\bigl(\abs{F'_\kappa}^2+R_\kappa\abs{F_\kappa}^2\bigr)\, d\xi \geq \xi_0R_\kappa(0)\ge\frac{\xi_0 c^2 \kappa^2}{2}.
\]
\end{proof}

\subsection{Step 4: Construction of the eigenfunctions}\label{ssec:smooth-ext}

We now justify rigorously that the formula $F(\xi)G_\kappa(\eta)$ produces a
genuine smooth Dirichlet eigenfunction on the ellipse. The key observation is
that the elliptic coordinates form a \emph{conformal} system in two dimensions,
so the Dirichlet form has no metric prefactor; the focal-segment degeneracy is
then handled cleanly by passing to the weak formulation and invoking elliptic
regularity on the smooth domain $E$.

Assume now that $\kappa\ge \frac{2M}{c^2}$ and let $G_\kappa:C^\infty(\T_{2\pi})\to\R$ be the angular ground state (the unique and positive eigenfunction to the smallest eigenvalue) of
Lemma~\ref{lem:ground-state}, and let $F_\kappa\in C^\infty([0,\xi_0])$ satisfy
the radial equation $F_\kappa''+(\kappa^2c^2\cosh^2\xi-\Lambda_\kappa)F=0$ for any $\kappa$ such that
\[
 F_\kappa'(0)=0,\qquad F_\kappa(0)=1,\quad F_\kappa(\xi_0)=0.
\]
Define $U_\kappa:E\to\R$ by
\[
 U_\kappa(x,y):=F_\kappa(\xi(x,y))\,G_\kappa(\eta(x,y))
 \qquad\text{for }(x,y)\in E.
\]
This is well defined on the focal segment $\{\xi=0\}$, where $(0,\eta)$ and $(0,-\eta)$ have the same image, because $G_\kappa(-\eta)=G_\kappa(\eta)$; moreover $U_\kappa(x,y)=U_\kappa(x,-y)=U_\kappa(-x,y)$ by the symmetries of $G_\kappa$ in Lemma~\ref{lem:ground-state}.
In order to show continuity at on all of the ellipse take $\tilde{x}\in [-c,c]$. We find that
 $(\tilde{x},0)=(c\cos(\tilde{\eta}),0)$, 
for a unique $\tilde{\eta}\in [0,\pi)$. Now consider a sequence
 $\lim_i(x_i,y_i)= (\tilde{x},0)$, but this is equivalent to a sequence $(\eta_i,\xi_i)$ such that $\lim_i \abs{\eta_i}= \tilde{\eta}$ and $\lim_i \xi_i=0$, where $(\eta_i,\xi_i)$ are uniquely determined by the coordinate functions 
 \[
 x_i=c\cosh\xi_i\cos\eta_i,\qquad y_i=c\sinh\xi_i\sin\eta_i.
\]
Hence
\[
\lim_{i} U_\kappa(x_i,y_i)=\lim_{i}F_\kappa(\xi_i)G_\kappa(\eta_i)=G_\kappa(\tilde{\eta})=:U_\kappa(x,0).
\] 
Further as $(x,y)\in\partial E$ exactly when $\xi=\xi_0$, we find that $U_\kappa=0$ on $\partial E$ we conclude that $U_\kappa\in C^0_0(E)$. Next we show it is actually a smooth solution to the eigenvalue problem.
\begin{lemma}\label{lem:smooth-ext}
Let $\kappa\ge K_1:=\max\{2M/c^2,\,16c_2/\pi^2\}$, with $c_2$ from Lemma~\ref{lem:agmon}, satisfy $F_\kappa(\xi_0)=0$. Then:
\begin{enumerate}[label=\textup{(\roman*)}]
 \item $U_\kappa\in H_0^1(E)$;
 \item $U_\kappa$ is a weak solution of $-\Delta U_\kappa=\kappa^2U_\kappa$ in $E$;
 \item $U_\kappa\in C^\infty(\overline E)$, and $U_\kappa$ is even in $x$ and in $y$;
 \item $c_0\le\norm{U_\kappa}_{L^2(E)}^2\le C_1^2$, with $c_0:=\xi_0c^2/16$ and $C_1$ independent of $\kappa$;
\item for every $\eps\in(0,a\sin(\pi/8))$ there is $C_\eps>0$, independent of $\kappa$, such that $\norm{U_\kappa}_{H^1(E\cap\{\abs{x}>\eps\})} \leq C_\eps\,\kappa^2 e^{-\tilde{c}\eps^2\kappa}$, where $\tilde c:=c_1/(2a^2)$ with $c_1$ from Lemma~\ref{lem:agmon}.
\end{enumerate}
\end{lemma}

\begin{proof}
 The map $(\xi,\eta)\mapsto(x,y)$ of \eqref{eq:ellipt-coord}, extended to $\xi\in(-\xi_0,\xi_0)$, is a local diffeomorphism at every point where $h\neq0$, that is, away from $(\xi,\eta)\in\{(0,0),(0,\pi)\}$, whose images are the foci $(\pm c,0)$; it is two-to-one, $(\xi,\eta)$ and $(-\xi,-\eta)$ having the same image. The even extension of $F_\kappa$ to $(-\xi_0,\xi_0)$ (still denoted $F_\kappa$) is smooth and solves \eqref{eq:ivp} there, because $F'_\kappa(0)=0$ and $R_\kappa$ is even; together with $G_\kappa(-\eta)=G_\kappa(\eta)$ this shows that $F_\kappa(\xi)G_\kappa(\eta)$ is invariant under $(\xi,\eta)\mapsto(-\xi,-\eta)$, so that $U_\kappa$ is well defined and smooth on $E\setminus\{(\pm c,0)\}$, including across the open focal segment. By \eqref{eq:laplace}, \eqref{eq:separated} and \eqref{eq:Lk-def},
\begin{align*}
-\Delta U_\kappa (x(\xi,\eta),y(\xi,\eta))&=-\frac{1}{h^2}\big(F_\kappa''(\xi) G_\kappa(\eta)+F_\kappa(\xi) G_\kappa''(\eta) \big)\\
&= {\kappa^2 U_\kappa} (x(\xi,\eta),y(\xi,\eta))
\end{align*}
on $E\setminus\{(\pm c,0)\}$.

{\em Proof of (i):} From $F'_\kappa G_\kappa =\alpha\,\partial_x U_\kappa + \beta\,\partial_y U_\kappa$ and $F_\kappa G'_\kappa =-\beta\,\partial_x U_\kappa + \alpha\,\partial_y U_\kappa$ we get
\[
\partial_x U_\kappa =\frac{\alpha F'_\kappa G_\kappa-\beta F_\kappa G'_\kappa}{h^2},\qquad \partial_y U_\kappa =\frac{\beta F_\kappa 'G_\kappa+\alpha F_\kappa G_\kappa'}{h^2},
\]
so that $\abs{\nabla U_\kappa}^2=h^{-2}\bigl((F'_\kappa G_\kappa)^2+(F_\kappa G'_\kappa)^2\bigr)$,
since $\alpha^2+\beta^2=h^2$. As the area element is $h^2\,d\xi\,d\eta$, we obtain
\[
\int_{E}\abs{\nabla U_\kappa}^2\, dx\, dy=\int_0^{\xi_0}\int_0^{2\pi} \bigl(\abs{F'_\kappa}^2\abs{G_\kappa}^2+\abs{G_\kappa'}^2\abs{F_\kappa}^2\bigr)\, d\eta\,d\xi<\infty
\]
by Lemma~\ref{lem:boundedF}. Hence $U_\kappa\in H^1(E)$; being continuous on $\overline E$ with $U_\kappa=0$ on $\partial E$, it belongs to $H^1_0(E)$.

{\em Proof of (ii):}  Multiplying the pointwise equation by $\phi\in C^\infty_0(E\setminus \{(\pm c,0)\})$ and integrating by parts,
\[
\int_E\nabla U_\kappa\cdot\nabla\phi\, dx\,dy-\kappa^2\int_E  U_\kappa \phi\, dx\, dy=0 .
\]
Since points have zero capacity in two dimensions, $C^\infty_0(E\setminus\{(\pm c,0)\})$ is dense in $H^1_0(E)$, and the identity extends to all $\phi\in H^1_0(E)$. Thus $U_\kappa$ is a weak Dirichlet eigenfunction, and elliptic regularity on the smooth domain $E$ gives $U_\kappa\in C^\infty(\overline E)$. 

{\em Proof of (iii):} The symmetries follow from $G_\kappa(-\eta)=G_\kappa(\eta)=G_\kappa(\pi-\eta)$, since $(x,y)\mapsto(x,-y)$ and $(x,y)\mapsto(-x,y)$ correspond to $\eta\mapsto-\eta$ and $\eta\mapsto\pi-\eta$. 

{\em Proof of (iv):} By the symmetries of $G_\kappa$ the four arcs $[\frac{\pi}{4},\frac{\pi}{2}]$, $[\frac{\pi}{2},\frac{3\pi}{4}]$, $[\frac{5\pi}{4},\frac{3\pi}{2}]$ and $[\frac{3\pi}{2},\frac{7\pi}{4}]$ carry the same $L^2$-mass, and their union is the complement of $S_{\pi/4}$; Lemma~\ref{lem:agmon} with $\delta=\pi/4$ (applicable since $\kappa\ge16c_2/\pi^2$) gives
\[
4 \int_{\pi/4}^{\pi/2}\abs{G_\kappa}^2\,d\eta=1-\int_{S_{\pi/4}}\abs{G_\kappa}^2\,d\eta\geq \frac{1}{2}
\]
after enlarging $K_1$ if necessary. On these four arcs $\cos^2\eta\le\frac12$, so $h^2\ge c^2(\cosh^2\xi-\frac12)\ge\frac{c^2}{2}\cosh^2\xi\ge\frac{R_\kappa(\xi)}{2\kappa^2}$ by \eqref{eq:assymptotics}. Hence, using Lemma~\ref{lem:boundedF},
\begin{align}
\label{eq:lower bound}
\norm{U_\kappa}_{L^2(E)}^2&=\int_0^{2\pi}\abs{G_\kappa(\eta)}^2\int_0^{\xi_0}\abs{F_\kappa(\xi)}^2 h^2(\xi,\eta)\,d\xi\, d\eta\notag\\
&\geq 
4\int_\frac{\pi}{4}^{\frac{\pi}{2}}\abs{G_\kappa}^2\, d\eta\int_0^{\xi_0}\abs{F_\kappa}^2 \frac{R_\kappa(\xi)}{2\kappa^2}\,d\xi\notag\\
&\geq \frac{1}{4\kappa^2}\cdot\frac{\xi_0c^2\kappa^2}{4}=\frac{\xi_0c^2}{16}=c_0 .
\end{align}
The upper bound $\norm{U_\kappa}_{L^2(E)}^2\le\xi_0a^2\norm{F_\kappa}_\infty^2=:C_1^2$ follows from $h^2\le c^2\cosh^2\xi_0=a^2$, $\norm{G_\kappa}_{L^2}=1$ and Lemma~\ref{lem:boundedF}.

{\em Proof of (v):}  Let $\eps\in(0,a\sin(\pi/8))$ and $\delta:=\arcsin(\eps/a)\in(0,\pi/8)$. On $E\cap\{\abs x>\eps\}$ we have $\abs{\cos\eta}=\abs x/(c\cosh\xi)>\eps/a$, i.e.\ $\eta\in S_\delta$. Hence, by Lemma~\ref{lem:boundedF} and Lemma~\ref{lem:agmon}, for $\kappa\ge c_2/\delta^2$,
\begin{align*}
\norm{U_\kappa}_{L^2(E\cap\{\abs x>\eps\})}^2&\le \int_{S_\delta}\abs{G_\kappa}^2\,d\eta\ \sup_\eta\int_0^{\xi_0}\abs{F_\kappa}^2h^2\,d\xi\\
&\le a^2\xi_0\norm{F_\kappa}_\infty^2\,\frac{C}{\delta^2\kappa}\,e^{-c_1\delta^2\kappa},\\
\norm{\nabla U_\kappa}_{L^2(E\cap\{\abs x>\eps\})}^2&\le \int_{S_\delta}\int_0^{\xi_0}\bigl(\abs{F'_\kappa}^2\abs{G_\kappa}^2+\abs{F_\kappa}^2\abs{G'_\kappa}^2\bigr)\,d\xi\,d\eta\\
&\le \xi_0\bigl(\norm{F'_\kappa}_\infty^2+\norm{F_\kappa}_\infty^2\bigr)\,C\kappa\,e^{-c_1\delta^2\kappa}\le C'\kappa^3e^{-c_1\delta^2\kappa}.
\end{align*}
\end{proof}
We can deduce also higher order norms both in positive and negative spaces. Recall from Section~\ref{sec:ellipse-quasimode} that
\[
\mathcal{H}^j(E)=\{\phi\in H^j(E):\Delta^m\phi\in H^1_0(E)\text{ for all }m<\frac{j}{2}\},
\]
for $j>0$ and as $\mathcal{H}^{-j}(E)=(\mathcal{H}^j(E))^*$.
The next lemma shows a uniform lower bound of $U_\kappa$ in all these spaces.
\begin{lemma}
\label{lem:lowerbound}
For all $\ell\in \Z$ and $\kappa$ as in Lemma~\ref{lem:smooth-ext},
\[
\norm{U_\kappa}_{\mathcal{H}^{\ell}(E)}=\kappa^\ell\norm{U_\kappa}_{L^2(E)}\ge\sqrt{c_0}\,\kappa^\ell,
\]
with $c_0$ from Lemma~\ref{lem:smooth-ext}. Moreover, for every integer $m\ge0$ and every $\eps\in(0,a\sin(\pi/8))$ there are $C_{m,\eps}>0$ and $c_m>0$, the latter independent of $\eps$, such that
\[
\norm{U_\kappa}_{H^m(E\cap\{\abs{x}>\eps\})} \leq C_{m,\eps}\,\kappa^{2m} e^{-c_m\eps^2\kappa}.
\]
\end{lemma}
\begin{proof}
Since $\Delta^kU_\kappa=(-\kappa^2)^kU_\kappa\in H^1_0(E)$ for every $k\ge0$, we have $U_\kappa\in\mathcal H^\ell(E)$ for all $\ell\ge0$, and, using $\norm{\nabla U_\kappa}_{L^2(E)}=\kappa\norm{U_\kappa}_{L^2(E)}$ (test the eigenvalue equation with $U_\kappa$),
\begin{align*}
\norm{U_\kappa}_{\mathcal H^{2k}(E)}&=\norm{\Delta^kU_\kappa}_{L^2(E)}=\kappa^{2k}\norm{U_\kappa}_{L^2(E)},\\
\norm{U_\kappa}_{\mathcal H^{2k+1}(E)}&=\norm{\nabla \Delta^k U_\kappa}_{L^2(E)}=\kappa^{2k+1}\norm{U_\kappa}_{L^2(E)}.
\end{align*}
For negative orders let $\ell\ge1$ and $\phi\in\mathcal H^{\ell}(E)$. Moving the Laplacians from $U_\kappa$ to $\phi$ by Green's identity (all boundary terms vanish because $\Delta^mU_\kappa$ and $\Delta^m\phi$, $m<\ell/2$, vanish on $\partial E$) one finds $\int_EU_\kappa\phi\,dx=\kappa^{-2\ell}\langle U_\kappa,\phi\rangle_{\mathcal H^\ell(E)}$. Hence, by Cauchy--Schwarz and the choice $\phi=U_\kappa/\norm{U_\kappa}_{\mathcal H^\ell(E)}$,
\[
\norm{U_\kappa}_{\mathcal H^{-\ell}(E)}=\sup_{\norm{\phi}_{\mathcal H^\ell(E)}\le1}\Bigl|\int_EU_\kappa\phi\,dx\Bigr|=\kappa^{-2\ell}\norm{U_\kappa}_{\mathcal H^{\ell}(E)}=\kappa^{-\ell}\norm{U_\kappa}_{L^2(E)} .
\]

For the localization estimate we argue by induction on $m$; the cases $m=0,1$ are Lemma~\ref{lem:smooth-ext}(v), with $c_0=c_1:=\tilde c$ (constants renamed). Let $m\ge2$ and assume the bound for $m-1$, for every $\eps$. Let $\zeta=\zeta(x)\in C^\infty(\R)$ with $\zeta=1$ on $\{\abs x\ge\eps\}$, $\zeta=0$ on $\{\abs x\le\eps/2\}$ and $\abs{\zeta^{(k)}}\le C_k\eps^{-k}$. Then $\zeta U_\kappa\in H^1_0(E)$ and
\[
\Delta(\zeta U_\kappa)=-\kappa^2\zeta U_\kappa+2\zeta'\partial_xU_\kappa+\zeta''U_\kappa ,
\]
so that elliptic regularity for the Dirichlet problem on the smooth domain $E$ (\cite[Theorem 8.13]{GilbargTrudinger2001}) gives
\begin{align*}
\norm{U_\kappa}_{H^m(E\cap\{\abs x>\eps\})}&\le\norm{\zeta U_\kappa}_{H^m(E)}\le C\norm{\Delta(\zeta U_\kappa)}_{H^{m-2}(E)}\\
&\le C\bigl(\kappa^2+C_{m}\eps^{-m}\bigr)\norm{U_\kappa}_{H^{m-1}(E\cap\{\abs x>\eps/2\})}.
\end{align*}
By the induction hypothesis with $\eps/2$ in place of $\eps$, the right-hand side is at most $C_{m,\eps}\kappa^{2m}e^{-c_{m-1}\eps^2\kappa/4}$, which is the claim with $c_m:=c_{m-1}/4$.
\end{proof}
Please observe that the lemma directly implies, by the pointwise bound $\abs{\Delta^kv}\le2^k\abs{\nabla^{2k}v}$ in two dimensions,
\begin{align}
\label{eq:properHj}
\norm{\nabla^{j}U_\kappa}_{L^2(E)}\geq 2^{-\lfloor j/2\rfloor}\sqrt{c_0}\,\kappa^j,
\end{align}
for positive integers $j$.

\section*{Acknowledgments}
We thank Marija Kranj\v{c}evi\'{c} for numerically computing the Dirichlet eigenfunctions of the ellipse and preparing Figure~\ref{fig:panel}. Author 1 was supported by the Croatian Science Foundation project IP-2022-10-2962, by the Croatia-USA bilateral grant “The mathematical framework for the diffuse interface method applied to coupled problems in fluid dynamics”, and by the European Union – NextGenerationEU through the National Recovery and Resilience Plan 2021–2026, institutional grant of the University of Zagreb Faculty of Science IK IA 1.1.3 Impact4Math. Author 2 was supported by the VR-Grant 2022-03862 of the Swedish Research Council, by the ERC-CZ grant LL2105 of the Czech Ministry of Education, Youth and Sports, and by the Charles University Research Centre program No. UNCE/24/SCI/005. Author 2 is a member of the Ne\v{c}as Centre for Mathematical Modeling. Author 3 was partially supported by NSF-DMS 2307538 and UMBC's Strategic Award for Research Transitions (START).

\end{document}